\documentclass[11pt]{article}

\usepackage{verbatim,latexsym,amsfonts,amsmath,amssymb,graphicx,fancyhdr,hyperref,asymptote}
\usepackage{appendix,latexsym,amsfonts,amsmath,amssymb,graphicx,hyperref,amsthm,soul,verbatim,authblk}
\usepackage[framemethod=tikz]{mdframed}

\newcommand{\EE}{\mathbb{ E}}
\newcommand{\PP}{\mathbb{P}}

\newcommand{\R}{\mathbb{R}}
\newcommand{\C}{\mathbb{C}}
\newcommand{\Q}{\mathbb{Q}}

\newcommand{\HH}{\mathbb{H}}
\newcommand{\N}{\mathbb{N}}
\newcommand{\D}{\mathbb{D}}

\newcommand{\pa}{\partial}

\newcommand{\F}{{\cal F}}

\newcommand{\no}{\noindent}

\newcommand{\aeq}{\overset{\mathrm{ae}}{=}}

\def\eps{\varepsilon}
\def\til{\widetilde}
\def\ha{\widehat}
\def\sem{\setminus}
\def\lin{\overline}
\def\ulin{\underline}

\def\del{\delta}

\def\L{{\cal L}}

\def\I{{\cal I}}
\def\lr{\leftrightarrow}

\DeclareMathOperator{\rad}{rad}
\DeclareMathOperator{\sign}{sign} \DeclareMathOperator{\diam}{diam}
\DeclareMathOperator{\dist}{dist} 
\DeclareMathOperator{\hcap}{hcap} 
\DeclareMathOperator{\Imm}{Im } \DeclareMathOperator{\Ree}{Re }

\DeclareMathOperator{\mA}{m} 
 
 \DeclareMathOperator{\cc}{c}
\DeclareMathOperator{\bb}{b} \DeclareMathOperator{\doub}{doub}
\DeclareMathOperator{\disj}{disj} \DeclareMathOperator{\Hull}{Hull}

\theoremstyle{plain}
\newtheorem{Theorem}{Theorem}[section]
\newtheorem{Lemma}[Theorem]{Lemma}
\newtheorem{Corollary}[Theorem]{Corollary}
\newtheorem{Proposition}[Theorem]{Proposition}
\theoremstyle{definition}
\newtheorem{Definition}[Theorem]{Definition}
\newtheorem{Remark}[Theorem]{Remark}

\numberwithin{equation}{section}
\newcommand{\BGE}{\begin{equation}}
\newcommand{\BGEN}{\begin{equation*}}
\newcommand{\EDE}{\end{equation}}
\newcommand{\EDEN}{\end{equation*}}

\begin{document}
\title{Two-curve Green's function for $2$-SLE: the boundary case}
\author{Dapeng Zhan
}
\affil{Michigan State University}
\maketitle

\begin{abstract}
   We prove that for $\kappa\in(0,8)$, if $(\eta_1,\eta_2)$ is a $2$-SLE$_\kappa$ pair in a simply connected domain $D$ with an analytic boundary point $z_0$, then {as $r\to 0^+$,} $ \PP[\dist(z_0,\eta_j)<r,j=1,2]$ converges to a positive number for some $\alpha>0$, which is called the  two-curve Green's function. The exponent $\alpha$ equals $\frac{12}{\kappa}-1$ or $2(\frac{12}{\kappa}-1)$ depending on whether $z_0$ is  one of the endpoints of $\eta_1$ {or} $\eta_2$. We also find the convergence rate and the exact formula {for} the Green's function up to a multiplicative constant. To derive these results, we construct   two-dimensional diffusion processes and use orthogonal polynomials to obtain their transition density.
\end{abstract}

\tableofcontents

\section{Introduction}
\subsection{Main results}
This paper is the follow-up of   \cite{Two-Green-interior}, in which we proved the existence of {the} two-curve Green's function for $2$-SLE$_\kappa$ at an interior point, and obtained the formula for the Green's function up to a multiplicative constant. In the present paper, we will study the case when the interior point is replaced by a boundary point. 

As a particular case of multiple SLE$_\kappa$, a $2$-SLE$_\kappa$ consists of two random curves in a simply connected domain connecting two pairs of boundary points (more precisely, prime ends), which satisfy the property that, when any one curve is given, the conditional law of the other curve is that of a chordal SLE$_\kappa$ in a complement{ary} domain of the first curve.

The two-curve Green's function of a $2$-SLE$_\kappa$ is {defined to be} the rescaled limit of the probability that the two curves in the $2$-SLE$_\kappa$ both approach a marked point in $\lin D$. More specifically, it was proved in \cite{Two-Green-interior} that, for any $\kappa\in(0,8)$, if $(\eta_1,\eta_2)$ is a $2$-SLE$_\kappa$ in $D$, and $z_0\in D$, then the limit
\BGE G(z_0):=\lim_{r\to 0^+} r^{-\alpha}\PP[\dist(\eta_j,z_0)<r,j=1,2] \label{Gz0}\EDE
{is} a positive number, where the exponent $\alpha$ is $\alpha_0:=\frac{(12-\kappa)(\kappa+4)}{8\kappa}$. The limit $G(z_0)$ is called the (interior) two-curve Green's function for  $(\eta_1,\eta_2)$. The paper \cite{Two-Green-interior} also derived the convergence rate and the exact formula for $G(z_0)$ up to an unknown constant.

In this paper we study the limit in the case that $z_0\in\pa D$ assuming that $\pa D$ is analytic near $z_0${\footnote{By saying that $\D$ is analytic near $z_0$, we mean that there is a conformal map $f$ defined on $\{|z|<1\}$ such that $f(0)=z_0$, $f(\{|z|<1,\Imm z>0\})=f(\{|z|<1\})\cap D$, and $f(\{|z|<1,\Imm z\le 0\})=f(\{|z|<1\})\cap D^c$.}}. Below is our  main theorem.

\begin{Theorem}
Let $\kappa\in(0,8)$. Let $(\eta_1,\eta_2)$ be a $2$-SLE$_\kappa$ in a simply connected domain $D$. Let $z_0\in\pa D$. Suppose $\pa D$ is analytic near $z_0$.   We have the following results in two cases.
\begin{enumerate}
	\item [(A)] If $z_0$ is not {an} endpoint of $\eta_1$ or $\eta_2$, then the limit in (\ref{Gz0}) exists and lies in $(0,\infty)$ for $\alpha= \alpha_1= \alpha_2:=2(\frac{12}{\kappa}-1)$.
	\item [(B)] If $z_0$ is one of the endpoints of $\eta_1$ {or} $\eta_2$, then the limit in (\ref{Gz0}) exists and lies in $(0,\infty)$ for $\alpha=\alpha_3:= \frac{12}{\kappa}-1$.
\end{enumerate}
Moreover, in each case we may compute $G_D(z_0)$   up to some constant $C>0$ as follows. Let $F$ denote the hypergeometric function $_2F_1(\frac 4\kappa,1-\frac 4\kappa;\frac 8\kappa,\cdot)$.
Let $f$ map $D$ conformally onto $\HH$ such that $f(z_0)=\infty$. Let $J$ denote the map $z\mapsto -1/z$.
\begin{enumerate}
	\item [(A1)] Suppose Case (A) happens and {neither} $\eta_1$ {nor} $\eta_2$ separates $z_0$ from the other curve. We label the $f$-images of the four endpoints of $\eta_1$ and $\eta_2$ by $v_-<w_-<w_+<v_+$. Then
	$$G_D(z_0)= C_1 |(J\circ f)'(z_0)|^{\alpha_1} G_{1 }(\ulin w;\ulin v),$$
	where $C_1>0$ is a constant depending only on $\kappa$, and
\BGE G_1(\ulin w;\ulin v):=\prod_{\sigma\in\{+,-\}}( |w_\sigma-v_\sigma|^{\frac 8\kappa-1} |w_\sigma-v_{-\sigma}|^{\frac 4\kappa} ) F\Big(\frac{(w_+-w_-)(v_+-v_-)}{(w_+-v_-)(v_+-w_-)}\Big)^{-1}.\label{G1(w,v)}\EDE
	\item [(A2)] Suppose Case (A) happens and  one of $\eta_1$ and $\eta_2$ separates $z_0$ from the other curve. We label the $f$-images of the four endpoints of $\eta_1$ and $\eta_2$ by $v_-<w_-<w_+<v_+$. Then
	$$G_D(z_0)= C_2 |(J\circ f)'(z_0)|^{\alpha_2} G_{2}(\ulin w;\ulin v)$$
where $C_2>0$ is a constant depending only on $\kappa$, and
 \BGE G_2(\ulin w;\ulin v):=\prod_{u\in\{w,v\}} |u_+-u_-|^{\frac 8\kappa -1} \prod_{\sigma\in\{+,-\}}  |w_\sigma-v_{-\sigma}|^{\frac {4}\kappa} F\Big(\frac{(v_+-w_+)(w_--v_-)}{(w_+-v_-)(v_+-w_-)}\Big)^{-1}.  \label{G2(w,v)}\EDE
	\item [(B)] Suppose Case (B) happens. We label the $f$-images of the other three endpoints of $\eta_1$ and $\eta_2$ by $w_+,w_-,v_+$, such that $f^{-1}(v_+)$ and $z_0$ are endpoints of the same curve, and $w_+,v_+$ lie on the same side of $w_-$. Then
	$$G_D(z_0)= C_3 |(J\circ f)'(z_0)|^{\alpha_3} G_{3}(\ulin w;v_+ ),$$
where $C_3>0$ is a constant depending only on $\kappa$, and
\BGE G_3(\ulin w;v_+)=|w_+-w_-|^{\frac 8\kappa -1}   |v_+-w_{-}|^{\frac {4}\kappa}  F\Big(\frac{ v_+-w_+ }{v_+-w_- }\Big)^{-1}.  \label{G3(w,v)}\EDE
\end{enumerate}
{In each case, the function $G_D$ does not depend on the choice of $f$ because $G_1$ and $G_2$ are homogeneous of degree $2(\frac 8\kappa -1+\frac 4\kappa)=\alpha_1=\alpha_2$, and $G_3$ is homogeneous of degree $\frac 8\kappa -1+\frac 4\kappa=\alpha_3$.}
\label{main-Thm1}
\end{Theorem}

Our long-term goal is to prove the existence of Minkowski content of double points of chordal SLE$_\kappa$ for $\kappa\in(4,8)$, which may be transformed into the existence of Minkowski content of the intersection of the two curves in a $2$-SLE$_\kappa$. Following the approach in \cite{LR}, we need to prove the existence of two-curve two-point Green's function for $2$-SLE$_\kappa$, where Theorem \ref{main-Thm1} is expected to serve as the boundary estimate in the proof.

{The paper uses a two-curve technique introduced in \cite{Two-Green-interior}, which will be described in Section \ref{Strategy}. Besides \cite{Two-Green-interior} and this paper, the  technique was recently used in \cite{Green-cut} to study the Green's function for the cut points of chordal SLE$_\kappa$, $\kappa\in(4,8)$. One future application of the technique is the interior two-curve Green's function for a commuting pair of SLE$_\kappa(2,\ulin\rho)$ curves (which may arise as a commuting pair of flow lines in the imaginary geometry theory (\cite{MS1}) in the case $\kappa\ne 4$, or a commuting pair of Gaussian free field level lines (\cite{WW-level}) in the case $\kappa =4$,  cf.\ Section \ref{section-commuting-SLE-kappa-rho}). The result is expected to lead to the existence of the Minkowski content of the intersection of the two curves (subject to the existence of the two-curve two-point Green's function for this commuting pair).
}

\subsection{Strategy} \label{Strategy}
\begin{figure}
	\begin{center}
		\includegraphics[width=1.9in]{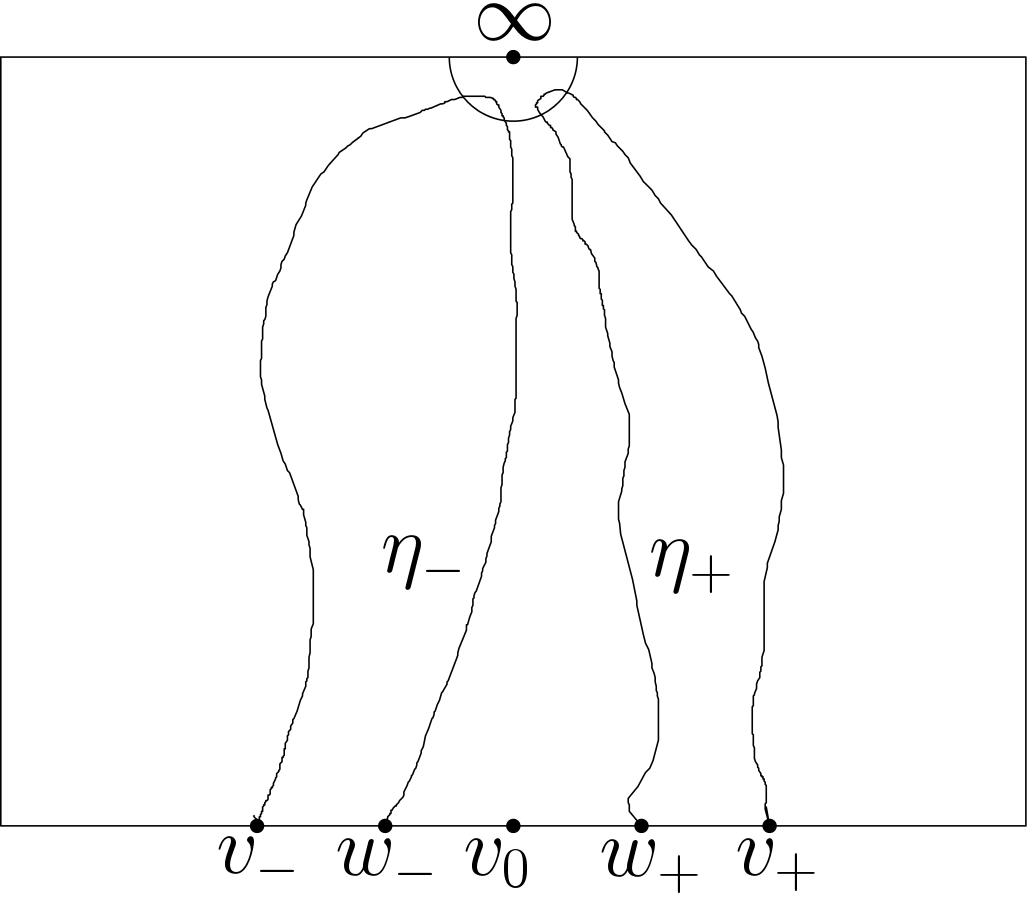} \quad
		\includegraphics[width=1.9in]{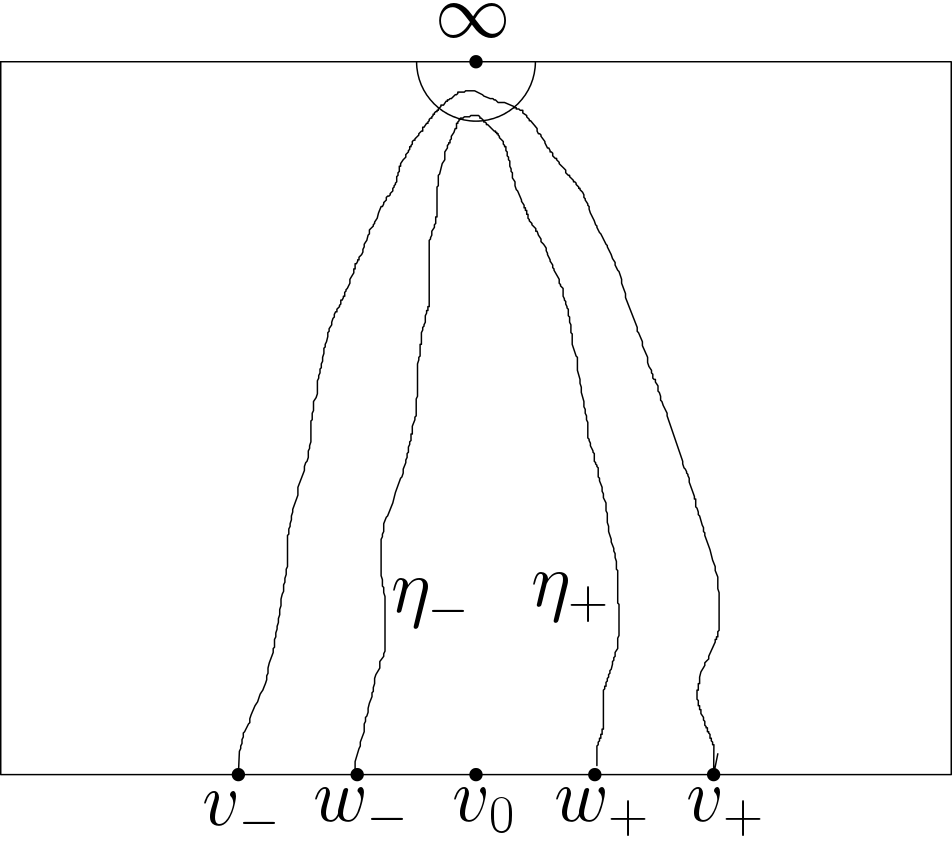}\quad
		\includegraphics[width=1.9in]{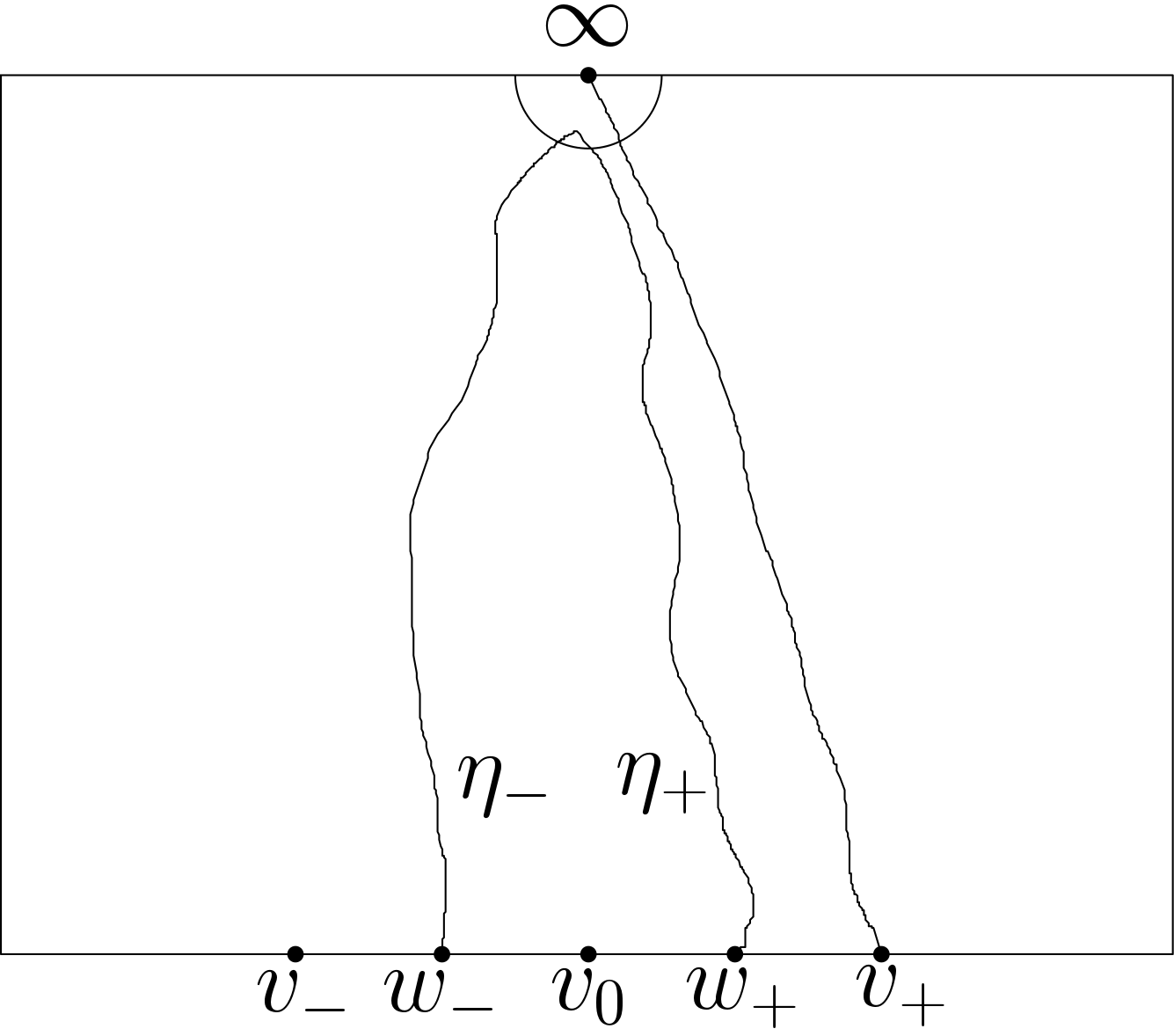}
	\end{center}
	\caption{{The three figures above respectively illustrate Case (A1), Case (A2), and Case (B).}} \label{fig-1}
\end{figure}

{We now briefly explain how the two-curve technique works for the boundary two-curve Green's function.} 
By conformal invariance of $2$-SLE$_\kappa$, we may assume that  $D=\HH:=\{z\in\C:\Imm z>0\}$, and $z_0=\infty$. It suffices to consider the limit
\BGE \lim_{L\to \infty} L^{\alpha} \PP[\eta_j\cap \{|z|>L\}\ne\emptyset,j=1,2].\label{limit-L}\EDE  In Case (A) of Theorem \ref{main-Thm1}, we label the four endpoints of $\eta_1$ and $\eta_2$ by $v_+>w_+>w_->v_-$. There are two possible link patterns: $(w_+\lr v_+;w_-\lr v_-)$ and $(w_+\lr w_-;v_+\lr v_-)$, which respectively correspond to Case (A1) and Case (A2) of Theorem \ref{main-Thm1}. {See Figure \ref{fig-1} for illustrations of the thee cases of Theorem \ref{main-Thm1}.}

For the first link pattern, we label the two curves by $\eta_+$ and $\eta_-$. By translation and dilation, we may assume that $v_\pm=\pm1$. {Additionally}, we assume that {$\frac{v_++v_-}2\in [w_-,w_+]$. After converting $v_\pm$ to $\pm 1$ using translation and dilation, we get}
 $0\in[w_-,w_+]$.  Let $v_0=0$. {We orient $\eta_+$ and $\eta_-$ from $w_+$ and $w_-$ to $v_+$ and $v_-$, and grow the two curves simultaneously with some speeds to be described later. The growing process stops at the time $T^u$ when either curve reaches its target, or separates $v_+$ or $v_-$ from $\infty$.} For each $t$ in the lifespan $[0,T^u)$, let $H_t$ denote the unbounded connected component of $\HH\sem (\eta_+[0,t]\cup \eta_-[0,t])$. During the lifespan $[0,T^u)$ of the process, the speeds of  $\eta_+$ and $\eta_-$ are controlled by two factors:
\begin{enumerate}
  \item [(F1)] the harmonic measure of $[v_-,v_+]\cup \eta_+[0,t]\cup \eta_-[0,t]$ in $H_t$ viewed from $\infty$ increases in $t$ exponentially with factor $2$, and
  \item [(F2)]   $[v_-,v_0]\cup \eta_-[0,t]$ and $[v_0,v_+]\cup \eta_+[0,t]$ have the same harmonic measure viewed from $\infty$.
\end{enumerate}
Suppose $g_t$ maps $H_t$ conformally onto $\HH$ and satisfies $g_t(z)=z+o(1)$ as $z\to\infty$. Define
$V_+(t)=\lim_{x\downarrow \max([v_0,v_+]\cup \eta_+[0,t]\cap \R)} g_t(x)$ and $V_-(t)=\lim_{x\uparrow \max([v_-,v_0]\cup \eta_-[0,t]\cap \R)} g_t(x)$. Then (F1) is equivalent to {the condition} that
$V_+(t)-V_-(t)=e^{2t}(v_+-v_-)$. The inverse $g_t^{-1}$ extends continuously to $\lin\HH$. We will see that there is a unique $V_0(t)\in (V_-(t),V_+(t))$ such that   $g_t^{-1}$  maps $[ V_0(t),V_\sigma(t)]$ into $[v_0,v_\sigma]\cup \eta_\sigma[0,t]$ for $\sigma\in\{+,-\}$. Then (F2) is equivalent to {the condition} that $V_+(t)-V_0(t)=V_0(t)-V_-(t)$.

{If $\kappa\in(0,4]$, $\eta_+$ and $\eta_-$ are disjoint, and do not disconnect $v_+,v_-,v_0$ from $\infty$. In this case,} $V_\sigma(t)$ is simply $g_t(v_\sigma)$ for $\sigma\in\{+,-,0\}$. {If $\kappa\in(4,8)$, $\eta_+$ and $\eta_-$ may or may not intersect, and they together may disconnect some $v_\sigma$, $\sigma\in\{+,-,0\}$ from $\infty$. When the disconnection happens, the function $V_\sigma(t)$ is more complicated. The two curves do not cross each other even if they intersect. }

At the time $T^u$, one of the two curves, say $\eta_+$, separates $v_+$ or $v_-$ from $\infty$. If $\eta_+$ separates $v_+$, the rest of $\eta_+$ grows in a bounded connected component of $\HH\sem \eta_+[0,T^u)$; if $\eta_+$ separates $v_-$, the whole $\eta_-$ is disconnected from $\infty$ by $\eta_+[0,T^u)$. Thus, after $T^u$, at least one curve {cannot} get closer to $\infty$. So we may focus on the parts of $\eta_+$ and $\eta_-$ before $T^u$.
Using Koebe's $1/4$ theorem (applied to $g_t$ at $\infty$) and Beurling's estimate (applied to a planar Brownian motion started near $\infty$), we find that for $0\le t<T^u$, the diameter of both $\eta_+[0,t]$ and $\eta_-[0,t]$ are comparable to $e^{2t}$.

We define a two-dimensional diffusion process $\ulin R(t)=(R_+(t),R_-(t))\in [0,1]^2$, $0\le t<T^u$, by $R_\sigma(t)=\frac{W_\sigma(t)-V_0(t)}{V_\sigma(t)-V_0(t)}$, $\sigma\in\{+,-\}$, where $W_\sigma(t)=g_t(\eta_\sigma(t))\in [V_0(t),V_\sigma(t)]$. Here $\eta_\sigma(t)$ is understood as a prime end of $H_t$.
We then use the knowledge of $2$-SLE$_\kappa$ partition function and a technique of orthogonal polynomials to derive the transition density  of $(\ulin R)$. 

{From the transition density, we find that $(\ulin R)$ has a quasi-invariant distribution, which means that if $(\ulin R)$ starts with this distribution, then the lifetime $T^u$ follows an exponential distribution, and the distribution of $\ulin R(t)$ conditionally on the event $\{T^u>t\}$ stays unchanged. Moreover, if we start $(\ulin R)$ from any other distribution, then the distribution of $\ulin R(t)$ conditionally on  $\{T^u>t\}$ converges exponentially to the quasi-invariant distribution.}

{To prove the existence of the limit in (\ref{limit-L}), we first prove that the limit exists if the condition $\eta_j\cap \{|z|>L\}\ne\emptyset$ , $j=1,2$, is replaced by the condition that $e^{2 T^u}>L$. The value $e^{2 T^u}$ plays the role of the conformal radius: by Koebe's $1/4$ theorem, the supremum of the set of $L>0$ such that $\eta_j\cap \{|z|>L\}\ne\emptyset$ , $j=1,2$, is comparable to $e^{2T^u}$. Suppose  $(\ulin R)$ starts from its quasi-invariant distribution. Then $L^{\alpha} \PP[e^{2 T^u}>L]$ stays constant for $L\in(0,\infty)$ by the property of the quasi-invariant distribution, and so its limit as $L\to\infty$ exists. If $(\ulin R)$ starts from a deterministic point, then the existence of the limit follows from the convergence of the conditional distribution of $\ulin R(t)$ to its quasi-invariant distribution. After this step, we remove the additional assumption that $\frac{v_++v_-}2\in [w_-,w_+]$ by growing a segment of one of the two curves. Finally, we use a technique in \cite{LR} (obtaining the Euclidian distance Green's function from the conformal radius Green's function) to prove the existence of the limit in (\ref{limit-L}).
}

For the link pattern $(w_+\lr w_-;v_+\lr v_-)$, we label the curves by $\eta_w$ and $\eta_v$. We observe that $\eta_v$ disconnects $\eta_w$ from $\infty$. Thus, for $L>\max\{|v_+|,|v_-|\}$, $\eta_w$ intersects $\{|z|>L\}$ implies that $\eta_v$ does the intersection as well. Then the two-curve Green's function reduces to a single-curve Green's function. But we will still use a two curve approach. We  assume that $v_\pm=\pm1$ and $0\in(w_-,w_+)$, and let $v_0=0$ as in the previous case. This time, we grow $\eta_+$ and $\eta_-$ simultaneously along the same curve $\eta_w$ such that $\eta_\sigma$ runs from $w_\sigma$ towards $w_{-\sigma}$, $\sigma\in\{+,-\}$. The growth is stopped if $\eta_+$ and $\eta_-$ together exhaust the range of $\eta_w$, or any of them disconnects its target from $\infty$. The speeds of the curves are also controlled by  (F1) and (F2). {The relation between $\eta_1$ and $\eta_2$ is similar to that in Case (A1) depending on whether $\kappa\in(0,4]$ or $\kappa\in(4,8)$.} Then we define $V_0,V_\pm,W_\pm,R_\pm$ in the same way as before, and derive the transition density of $\ulin R=(R_+,R_-)$, {which will be used to prove the existence of the limit in (\ref{limit-L}) following the same approach as in Case (A1)}.

In Case (B),  we may assume that $v_+=1$ and $w_++w_-=0$. Now we introduce two new points: $v_0=0$ and $v_-=-1$. Unlike the previous cases, $v_-$ is not an end point of any curve. For this case, we grow $\eta_+$ and $\eta_-$ simultaneously from $w_+$ and $w_-$ along the same curve $\eta_w$ as in Case (A2). The rest of the proof almost follows the same approach as in Case ({A1}).

\subsection{Outline}
Below is the outline of the paper. In Section \ref{section-prel}, we recall definitions, notations, and some basic results that will be needed in this paper. In Section \ref{section-deterministic} we develop a framework on a commuting pair of deterministic chordal Loewner curves, which do not cross but may touch each other. The work extends the disjoint ensemble of Loewner curves that appeared in \cite{reversibility,duality}. At the end of the section, we describe the way to grow the two curves simultaneously with properties (F1) and (F2). In Section \ref{section-commuting-SLE-kappa-rho}, we use the results from the previous section to study a pair of multi-force-point SLE$_\kappa(\ulin\rho)$ curves, which commute with each other in the sense of \cite{Julien}. We obtain a two-dimensional diffusion process $\ulin R(t)=(R_+(t),R_-(t))$, $0\le t<\infty$,   and derive its transition density using orthogonal two-variable polynomials. In Section \ref{section-other-commut}, we study three types of commuting pairs of hSLE$_\kappa$ curves, which correspond to the three cases in Theorem \ref{main-Thm1}. We  prove that each of them is {\it locally} absolutely continuous w.r.t.\  a commuting pair of SLE$_\kappa(\ulin\rho)$ curves for certain force values, and also find the Radon-Nikodym derivative at different times. For each commuting pair of hSLE$_\kappa$ curves, we obtain a two-dimensional diffusion process $\ulin R(t)=(R_+(t),R_-(t))$ with random finite lifetime, and derive its transition density and quasi-invariant density. In the last section we finish the proof of Theorem \ref{main-Thm1}.

\section*{Acknowledgments}
The author thanks Xin Sun for suggesting the problem  on the (interior and boundary) two-curve Green's function for $2$-SLE.

\section{Preliminar{ies}} \label{section-prel}
We first fix some notation. Let $\HH=\{z\in\C:\Imm z>0\}$. For $z_0\in\C$ and $S\subset \C$, let $\rad_{z_0}(S)=\sup\{|z-z_0|:z\in S\cup\{z_0\}\}$. If a function $f$ is absolutely continuous on a real interval $I$, and $f'=g$ a.e.\ on $I$, then we write $f'\aeq g$ on $I$.  This means that $f(x_2)-f(x_1)=\int_{x_1}^{x_2} g(x)dx$ for any $x_1<x_2\in I$. Here $g$ may not be defined on a subset of $I$ with Lebesgue measure zero. We will also use ``$\aeq$'' for PDE or SDE in some similar sense. {For example, if $B_t$ is a standard Brownian motion, the SDE $dX_t\aeq dB_t+ g_t dt$ means that a.s.\ $f_t:=X_t-B_t$ is absolutely continuous, and $f'=g$ a.e.\ in the lifespan. }

\subsection{$\HH$-hulls and chordal Loewner equation}
A  relatively closed subset $K$ of $\HH$ is called an $\HH$-hull if $K$ is bounded and $\HH\sem K$ is a simply connected domain. For a set $S\subset\C$, if there is an $\HH$-hull $K$ such that $\HH\sem K$ is
the unbounded connected component of $\HH\sem \lin S$, then we say that $K$ is  the $\HH$-hull generated by $S$, and write $K=\Hull(S)$.
For an $\HH$-hull $K$, there is a unique conformal map $g_K$ from $\HH\sem K$ onto $\HH$ such that $g_K(z)=z+\frac cz+O(1/z^2)$ as $z\to \infty$ for some $c\ge 0$. The constant $c$, denoted by $\hcap(K)$, is called the $\HH$-capacity of $K$, which is zero iff $K=\emptyset$. We write $\hcap_2(K)$ for $\hcap(K)/2$. If $\pa(\HH\sem K)$ is locally connected, then $g_K^{-1} $ extends continuously from $\HH$ to $\lin\HH$, and we use $f_K$ to denote the continuation. If $K=\Hull(S)$, then we write $g_S,f_S, \hcap(S),\hcap_2(S)$ for $g_K,f_K,\hcap(K),\hcap_2(K)$, respectively.

 If $K_1\subset K_2$ are two $\HH$-hulls, then we define $K_2/K_1=g_{K_1}(K_2\sem K_1)$, which is also an $\HH$-hull. Note that $g_{K_2}=g_{K_2/K_1}\circ g_{K_1}$ and  $\hcap(K_2)=\hcap(K_2/K_1)+\hcap(K_1)$, which impl{ies} that $\hcap(K_1),\hcap(K_2/K_1)\le \hcap (K_2)$. If $K_1\subset K_2\subset K_3$ are $\HH$-hulls, then $K_2/K_1\subset K_3/K_1$ and
\BGE (K_3/K_1)/(K_2/K_1)=K_3/K_2. \label{K123}\EDE

Let $K$ be a non-empty $\HH$-hull. Let $K^{\doub}=\lin K\cup\{\lin z:z\in K\}$, where $\lin K$ is the closure of $K$, and $\lin z$ is the  complex conjugate of $z$. By {the} Schwarz reflection principle, there is a compact set $S_K\subset\R$ such that $g_K$ extends to a conformal map from $\C\sem K^{\doub}$ onto $\C\sem S_K$. Let $a_K=\min(\lin{K}\cap \R)$, $b_K=\max(\lin{K}\cap\R)$,  $c_K=\min S_K$, $d_K=\max S_K$. Then the extended $g_K$ maps  $\C\sem (K^{\doub}\cup [a_K,b_K])$ conformally onto $\C\sem [c_K,d_K]$. Since $g_K(z)=z+o(1)$ as $z\to\infty$, by Koebe's $1/4$ theorem, $\diam(K)\asymp \diam(K^{\doub}\cup [a_K,b_K])\asymp d_K-c_K$.
\vskip 2mm

\no{\bf Example}.
Let $x_0\in\R$, $r>0$. Then $H:=\{z\in\HH:|z-x_0|\le r\}$ is an $\HH$-hull with $g_H(z)=z+\frac{r^2}{z-x_0}$, $\hcap(H)=r^2$, $a_H=x_0-r$, $b_H=x_0+r$, $H^{\doub}=\{z\in\C:|z-x_0|\le r\}$, $c_H=x_0-2r$, $d_H=x_0+2r$.
\vskip 2mm

The next proposition combines  \cite[Lemmas 5.2 and 5.3]{LERW}.
\begin{Proposition}
  If $L\subset K$ are two non-empty $\HH$-hulls, then $[a_K,b_K]\subset [c_K,d_K]$, $[c_L,d_L]\subset [c_K,d_K]$, and $[c_{K/L},d_{K/L}]\subset [c_K,d_K]$. \label{abcdK}
\end{Proposition}

\begin{Proposition}
  For any $x\in\R\sem K^{\doub}$, $0<g_K'(x)\le 1$. Moreover, $g_K'$ is decreasing on $(-\infty,a_K)$ and increasing on $(b_K,\infty)$.
   \label{Prop-contraction}
\end{Proposition}
\begin{proof}
 By \cite[Lemma C.1]{BSLE}, there is a  measure $\mu_K$ supported on $S_K$ with $|\mu_K|=\hcap(K)$ such that $g_K^{-1}(z)-z=\int_{S_K}  \frac{-1}{z-y}d\mu_K(y)$ for any $x\in\R\sem S_K$. Differentiating this formula and letting $z=x\in\R\sem S_K$, we get $(g_K^{-1})'(x)=1+\int_{S_K} \frac{1}{(x-y)^2}d\mu_K(y)\ge 1$. So  $0<g_K'\le 1$ on $\R\sem K^{\doub}$. Further differentiating  the integral formula w.r.t.\ $x$, we find that $(g_K^{-1})''(x)=\int_{S_K} \frac{-2}{(x-y)^3}d\mu_K(y)$ is positive on $(-\infty,c_K)$ and negative on $(d_K,\infty)$, which means that $(g_K^{-1})'$ is increasing on $(-\infty,c_K)$ and decreasing on $(d_K,\infty)$. Since $g_K$ maps  $(-\infty,a_K)$ and $(b_K,\infty)$ onto  $(-\infty,c_K)$ and $(d_K,\infty)$, respectively, we get the monotonicity of $g_K'$ .
\end{proof}

\begin{Proposition}
  If $K$ is an $\HH$-hull with $\rad_{x_0}(K)\le r$ for some $x_0\in\R$, then $\hcap(K)\le r^2$, $\rad_{x_0}(S_K)\le 2r$, and $ |g_K(z)-z|\le 3r$ for any $z\in\C\sem K^{\doub}$. \label{g-z-sup}
\end{Proposition}
\begin{proof}
  We have $K\subset H:=\{z\in\HH:|z-x_0|\le r\}$. By  Proposition \ref{abcdK}, $\hcap(K)\le \hcap(H)=r^2$, $S_K\subset [c_K,d_K]\subset[c_H,d_H]=[x_0-2r,x_0+2r]$. {If $x_0=0$, the inequality $\hcap(K)\le r^2$ is just \cite[Formula 3.9]{Law-SLE}; and the inequality $ |g_K(z)-z|\le 3r$ for any $z\in\HH\sem K$ is just \cite[Formula 3.12]{Law-SLE}. By translation, continuation, and reflection, we then extend these inequalities to general $x_0\in\R$ and all $z\in \C\sem K^{\doub}$.}
\end{proof}

\begin{Proposition}
For two nonempty $\HH$-hulls  $  K_1\subset K_2$ such that $\lin{K_2/K_1}\cap [c_{K_1},d_{K_1}]\ne \emptyset$, we have $|c_{K_{1}}-c_{K_{2}}|,|d_{K_{1}}-d_{K_{2}}|\le 4 \diam(K_{2}/K_{1})$.
   \label{Prop-cd-continuity}
\end{Proposition}
\begin{proof}
  By symmetry it suffices to estimate $|c_{K_1}-c_{K_2}|$.   Let $c_1'=\lim_{x\uparrow a_{K_2}} g_{K_1}(x)$ and $\Delta K=K_2/K_1$. Since $g_{K_1}$ maps $\HH\sem K_2$ onto $\HH\sem \Delta K$, we have $c_1'=\min\{c_{K_1},a_{\Delta K}\}$. Since $\lin{\Delta K}\cap  [c_{K_1},d_{K_1}]\ne \emptyset$, $ c_1'\ge c_{K_1}-\diam(\Delta K)$. Thus, by Proposition \ref{g-z-sup},
  $$c_{K_2}=\lim_{x\uparrow a_{K_2}} g_{\Delta K}\circ g_{K_1}(x)=\lim_{y\uparrow c_1'} g_{\Delta K}(y)\ge c_1'-3\diam(\Delta K)\ge c_{K_1}-4\diam(\Delta K).$$
 By Proposition \ref{abcdK}, $c_{K_2}\le c_{K_1}$. So we get $|c_{K_{1}}-c_{K_{2}}| \le 4 \diam(\Delta K)$.
\end{proof}

The following proposition is \cite[Proposition 3.42]{Law-SLE}.

\begin{Proposition}
  Suppose $K_0,K_1,K_2$ are $\HH$-hulls such that $K_0\subset K_1\cap K_2$. Then
  $$\hcap(K_1)+\hcap(K_2)\ge \hcap(\Hull(K_1\cup K_2))+\hcap(K_0).$$ \label{hcap-concave}
\end{Proposition}

Let $\ha w \in C([0,T),\R)$ for some $T\in(0,\infty]$. The chordal Loewner equation driven by $\ha w$  is
$$\pa_t g_t(z)=  \frac{2}{g_t(z)-\ha w(t)},\quad 0\le t<T;\quad g_0(z)=z.$$
For every $z\in\C$, let $\tau_z$ be the first time that the solution $g_\cdot(z)$ blows up; if such {a} time does not exist, then set $\tau_z=\infty$.
For  $t\in[0,T)$, let $K_t=\{z\in\HH:\tau_z\le t\}$. It turns out that each $K_t$ is an $\HH$-hull with $\hcap_2(K_t)=t$, $K_t^{\doub}=\{z\in\C:\tau_z\le t\}$, which is connected,  and each $g_t$   agrees with $g_{K_t}$. We call $g_t$ and $K_t$ the chordal Loewner maps and hulls, respectively, driven by $\ha w$ {(cf.\ \cite[Section 2.2]{Wer-SLE})}.

 If for every $t\in[0,T)$,  $f_{K_t}$ is well defined, and $\eta(t):=f_{K_t}({\ha w(t)})$, $0\le t<T$, is continuous in $t$, then we say that $\eta$ is the chordal Loewner curve driven by $\ha w$. Such $\eta$ may not exist in general. When it exists, we have $\eta(0)=\ha w(0)\in\R$, and $K_t=\Hull(\eta[0,t])$ for all $t$, and we say that $K_t$, $0\le t<T$, are generated by $\eta$.

Let $u$ be a continuous and strictly increasing function on $[0,T)$. Let $v$ be the inverse of $u-u(0)$. Suppose that   $g^u_t$ and $K^u_t$, $0\le t<T$, satisfy that $g^u_{v(t)}$ and $K^u_{v(t)}$, $0\le t<u(T)-u(0)$, are chordal Loewner maps and hulls, respectively, driven by $\ha w\circ v$. Then we say that $g^u_t$ and $K^u_t$, $0\le t<T$, are chordal Loewner maps and hulls, respectively, driven by $\ha w$ with speed $u$, and call $(K^u_{v(t)})$  the normalization of $(K^u_t)$. If $(K^u_t)$ {is} generated by a curve $\eta^u$, i.e., $K^u_t=\Hull(\eta^u[0,t])$ for all $t$, then $\eta^u$ is called a chordal Loewner curve driven by $\ha w$ with speed $u$, and $\eta^u\circ v$ is called the normalization of $\eta^u$.
If $u$ is absolutely continuous with $u'\aeq q$, then we also say that the speed is $q$. In this case, the chordal Loewner maps satisfy the differential equation  $\pa_t g^u_t(z)\aeq \frac{2q(t)}{g^u_t-\ha w(t)}$.  We omit the speed when it is constant {and equal to} $1$.

The following proposition is straightforward.

\begin{Proposition}
   Suppose $K_t$, $0\le t<T$, are chordal Loewner hulls driven by $\ha w(t)$, $0\le t<T$, with speed $u$.  Then for any $t_0\in[0,T)$, $K_{t_0+t}/K_{t_0}$, $0\le t<T-t_0$, are chordal Loewner hulls driven by $\ha w(t_0+t)$, $0\le t<T-t_0$, with speed $u(t_0+\cdot)$. One immediate consequence is that, for any $t_1<t_2\in[0,T)$, $\lin{K_{t_2}/K_{t_1}}$ is connected. \label{prop-connected}
\end{Proposition}

The following proposition is a slight variation of \cite[Theorem 2.6]{LSW1}.

\begin{Proposition}
  The $\HH$-hulls $K_t$, $0\le t<T$, are chordal Loewner hulls with some speed if and only if for any fixed $a\in[0,T)$, $\lim_{\delta\downarrow 0} \sup_{0\le t\le a} \diam(K_{t+\del}/K_t)=0$. Moreover, the driving function $\ha w$ satisfies {the property} that  $\{\ha w(t)\}=\bigcap_{\del>0} \lin{K_{t+\del}/K_{t}}$, $0\le t< T$; and the speed $u$ could be chosen to be $u(t)=\hcap_2(K_t)$, $0\le t<T$.
   \label{Loewner-chain}
\end{Proposition}

\begin{Proposition}
	Suppose $K_t$, $0\le t<T$, are chordal Loewner hulls driven by  $\ha w$ with some speed. Then  for any $t_0\in(0,T)$,   $c_{K_{t_0}}\le \ha w(t)\le d_{K_{t_0}}$ for all $  t\in[0, t_0]$. \label{winK}
\end{Proposition}
\begin{proof}
Let $t_0\in(0,T)$. If $0\le t<t_0$, by Propositions \ref{abcdK} and \ref{Loewner-chain}, $\ha w(t)\in[a_{K_{t_0}/K_t},b_{K_{t_0}/K_t}]\subset [c_{K_{t_0}/K_t},d_{K_{t_0}/K_t}]\subset [c_{K_{t_0} },d_{K_{t_0} }]$.
	By the continuity of $\ha w$, we also have $\ha w(t_0)\in [c_{K_{t_0} },d_{K_{t_0} }]$.
\end{proof}

The following proposition combines \cite[Lemma 2.5]{MS1} and \cite[Lemma 3.3]{MS2}.

\begin{Proposition}
	Suppose $\ha w\in C([0,T),\R)$ generates a chordal Loewner curve $\eta$ and chordal Loewner hulls $K_t$, $0\le t<T$. Then the set $\{t\in [0,T):\eta(t)\in\R\}$ has Lebesgue measure zero. Moreover, if the Lebesgue measure of $\eta[0,T)\cap\R$ is zero, then the functions $c(t)$ and $d(t)$ defined by $c(0)=d(0):=\ha w(0)$, and $c(t):=c_{K_t}$ and $d(t):=d_{K_t}$, $0< t<T$, {satisfy that (i) $c\le \ha w\le d$ on $[0,T)$; (ii) the set of $t\in[0,T)$ such that $c(t)=\ha w(t)$ or $\ha w(t)=d(t)$ (which implies that $\eta(t)\in\R$) has Lebesgue measure zero; and (iii) $c'(t)\aeq  \frac{2}{c(t)-\ha w(t)}$ and $d'(t)\aeq \frac{2}{d(t)-\ha w(t)}$ on $[0,T)$. Thus, $c$ and $d$ are respectively strictly decreasing and increasing on $[0,T)$}. Moreover, $c(t)$ and $d(t)$ are continuously differentiable  at the set of times $t$ such that $\eta(t)\not\in\R$, and in {this} case ``$\aeq$'' can be replaced by ``$=$''. \label{prop-Lebesgue}
\end{Proposition}

\begin{Definition} We define the following notation.
  \begin{enumerate}
    \item [(i)] Modified real line. For $w\in\R$, we define $\R_w^{}=(\R\sem\{w\})\cup \{w^-,w^+\}$,  which  has a total order endowed from $\R$ and the relation $x<w^- <w^+<y$ for any $x,y\in\R$ such that $x<w$ and $y>w$. It is assigned the topology such that $(-\infty,w^-]:=(-\infty,w)\cup\{w^-\}$  and $[w^+,\infty):=\{w^+\}\cup(w,\infty)$ are two connected components, and  are respectively homeomorphic to $(-\infty,w]$ and $[w,\infty)$ through the map $\pi_w:\R_w\to \R$ with $\pi_w(w^\pm)=w$ and $\pi_w(x)=x$ for $x\in\R\sem \{w\}$.
    \item [(ii)] Modified Loewner map. Let $K$ be an $\HH$-hull and $w\in\R$. Let $a^w_K=\min\{w,a_K\}$, $b^w_K=\max\{w,b_K\}$, $c^w_K=\lim_{x\uparrow a^w_K} g_K(x)$, and $d^w_K=\lim_{x\downarrow b^w_K} g_K(x)$. They are all equal to $w$ if $K=\emptyset$.
        Define $g_K^w$ on $\R_w\cup\{+\infty,-\infty\}$ such that $g_K^w(\pm\infty)=\pm\infty$, $g_K^w(x)=g_K(x)$ if $x\in\R\sem [a_K^w,b_K^w]$; $g^w_K(x)=c^w_K$ if $x=w^-$ or $x\in[a_K^w,b_K^w]\cap (-\infty,w)$; and $g_K^w(x)=d^w_K$ if $x=w^+$ or $x\in [a_K^w,b_K^w]\cap(w,\infty)$.
    Note that $g_K^w$ is  continuous and increasing.
  \end{enumerate} \label{Def-Rw}
\end{Definition}

\no{{\bf Example}.
  Let $x_0\in\R$, $r>0$, and $H=\{z\in\HH:|z-x_0|\le r\}$. Let $w\in[x_0-r,x_0+r]$. Then $a^w_H=x_0-r$, $b^w_H=x_0+r$, $c^w_H=x_0-2r$, $d^w_H=x_0+2r$, and
  $$g^{w}_H(x)=\left\{\begin{array}{ll} x+\frac{r^2}{x-x_0}, &\mbox{if }x\in \R\sem [x_0-r,x_0+r]\\
  x_0+2r,&\mbox{if }x\in \{w^+\}\cup (w,x_0+r]\\
  x_0-2r,&\mbox{if }x\in \{w^-\}\cup [x_0-r,w).
  \end{array}
  \right.$$
For the case $w\not\in [x_0-r,x_0+r]$, we assume $w\in (x_0+r,\infty)$ by symmetry. In this case,  $a^w_H=x_0-r$, $b^w_H=w$, $c^w_H=x_0-2r$, $d^w_H=g_H(w)=w+\frac{r^2}{w-x_0}$, and
  $$g^{w}_H(x)=\left\{\begin{array}{ll} x+\frac{r^2}{x-x_0}, &\mbox{if }x\in \R\sem [x_0-r,w]\\
 w+\frac{r^2}{w-x_0},&\mbox{if }x=w^+\\
  x_0-2r,&\mbox{if }x\in \{w^-\}\cup [x_0-r,w).
  \end{array}
  \right.$$
}

\begin{Proposition}
   Let $K_1\subset K_2$ be two $\HH$-hulls. Let $w\in\R$ and $\til w\in[c^w_{K_1},d^w_{K_1}]$. Revise   $g^w_{K_1}$  such that when $g^w_{K_1}(w)=\til w$, we define $g^w_{K_1}(x)=\til w^{\sign(x-w)}$.
   Then
\BGE g^{\til w}_{K_2/K_1}\circ g^w_{K_1} =g^w_{K_2},\quad \mbox{on } \R_w\cup\{+\infty,-\infty\}.\label{comp-g}\EDE
\label{prop-comp-g}
\end{Proposition}
\begin{proof}
By symmetry, it suffices to show that (\ref{comp-g}) holds on $[w^+,\infty]$. Since for $x\ge w^+$, $g^w_{K_1}(x)\ge d^w_{K_1}\ge \til w$, the revised $g^w_{K_1}$ is a continuous map from $[w^+,\infty]$ into $[\til w^+,\infty]$, and so both sides of (\ref{comp-g}) are continuous on  $[w^+,\infty]$.
If $x>b^w_{K_2}$, then $x>\max\{b^w_{K_1},b_{K_2}\}$, which implies that $g_{K_1}^w(x)=g_{K_1}(x)>\max\{d^w_{K_1},b_{K_2/K_1}\}\ge b^{\til w}_{K_2/K_1}$. Thus, $g^{\til w}_{K_2/K_1}\circ g_{K_1}^w(x)=g _{K_2/K_1}\circ g_{K_1} (x)=g_{K_2}(x)=g^w_{K_2}(x)$ on $(b^w_{K_2},\infty]$.
We know that $g^w_{K_2}$ is constant on $[w^+,b^w_{K_2}]$. To prove that (\ref{comp-g}) holds on $[w^+,\infty]$, by continuity it suffices to show that the LHS of (\ref{comp-g}) is constant on $ [w^+,b^w_{K_2}]$. This is obvious if $b^w_{K_1}=b^w_{K_2}$ since $g^w_{K_1}$ is constant on $[w^+,b^w_{K_1}]$. Suppose $b^w_{K_1}<b^w_{K_2}$. Then we have $b_{K_1},w<b^w_{K_2}=b_{K_2}$. So  $[w^+,b^w_{K_2}]$ is mapped by $g^w_{K_1}$ onto $[d^w_{K_1}, b_{K_2/K_1}]$ (or $[\til w^+,b_{K_2/K_1}]$), which is in turn mapped by $g^{\til w}_{K_2/K_1}$ to a constant.
\end{proof}

\begin{Proposition}
	Let $K_t$ and $\eta(t)$, $0\le t<T$, be chordal Loewner hulls and curve driven by $\ha w$ with speed $q$. Suppose the Lebesgue measure of $\eta[0,T)\cap\R$ is $0$. Let $w=\ha w(0)$, and $x \in\R_w$. Define $X(t)=g_{K_t}^w(x)$, $0\le t<T$. Then $X$ satisfies {(i)} $X'(t)\aeq  \frac{2q(t)}{X(t)-\ha w(t)}$ on $[0,T)$; {(ii) the set of $t$ such that $X(t)=\ha w(t)$ has Lebesgue measure zero; and (iii)} if $x>w$ (resp.\ $x<w$), then $X(t)\ge \ha w(t)$ (resp.\ $X(t)\le \ha w(t)$) on $[0,T)$,  and so $X$
is {strictly} increasing (resp.\ decreasing) on $[0,T)$. Moreover, for any $0\le t_1<t_2<T$, $|X(t_1)-X(t_2)|\le 4 \diam(K_{t_2}/K_{t_1})$.
	\label{Prop-cd-continuity'}
\end{Proposition}	
\begin{proof}
	We may assume that the speed $q$ is constant {and equal to} $1$.
	By symmetry,  we may assume that $x\in(-\infty,w^-]$. If $x=w^-$, then $X(t)=c_{K_t}$ for $t>0$ and $X(0)=\ha w(0)$. Then the conclusion follows from Propositions \ref{Prop-cd-continuity} and \ref{prop-Lebesgue}.  Now suppose $x\in(-\infty,w)$.

Fix $0\le t_1<t_2<T$. We first prove the upper bound for $|X(t_1)-X(t_2)|$. There are three cases. Case 1. $x\not\in \lin{K_{t_j}}$, $j=1,2$. In this case,
	$X (t_2)=g_{K_{t_2}/K_{t_1}}(X(t_1))$, and the upper bound for $|X(t_1)-X(t_2)|$ follows from Proposition \ref{g-z-sup}. Case 2. $x\in\lin{K_{t_1}}\subset \lin{K_{t_2}}$. In this case $X(t_j)=c_{K_{t_j}}$, $j=1,2$, and the conclusion follows from Proposition \ref{Prop-cd-continuity}. Case 3. $x\not\in\lin{K_{t_1}}$ and $x\in \lin{K_{t_2}}$. Then $X(t_1)=g_{K_{t_1}}(x_0)<c_{K_{t_1}}$  and $X(t_2)=c_{K_{t_2}}$. Moreover, we have $\tau_{x}\in(t_1,t_2]$, $\lim_{t\uparrow \tau_{x}} X(t)=\ha w(\tau_{x})$, and $X(t)$ satisfies $X'(t)=\frac{2}{X(t)-\ha w(t)}<0$ on $[t_1,\tau_{x})$.  By Propositions \ref{winK} and \ref{abcdK},  $c_{K(t_1)}>X(t_1)\ge \ha w(\tau_{x})\ge c_{K_{\tau_{x}}}\ge c_{K_{t_2}}=X(t_2)$. So we have $|X(t_1)-X(t_2)|\le |c_{K_{t_1}}-c_{K_{t_2}}|\le 4 \diam(K_{t_2}/K_{t_1})$ by Propositions \ref{Prop-cd-continuity}. By Proposition \ref{Loewner-chain}, $X$ is continuous on $[0,T)$.
	
Since $X(t)=g_{K_t}(x)$ satisfies {$X(t)<\ha w(t)$ and} the chordal Loewner equation driven by $\ha w$ up to $\tau_{x}$, we know that $X'(t)=\frac{2}{X(t)-\ha w(t)}{<0}$ on $[0,\tau_{x})$. {So $X$ is strictly decreasing on $[0,\tau_x)$.}  From Proposition \ref{prop-Lebesgue} we know that {$X(t)=c_{K_t}$ is strictly decreasing, the set of $t\in[\tau_x,T)$ such that $X(t)=\ha w(t)$ has Lebesgue measure zero, and} $X'(t)\aeq  \frac{2}{X(t)-\ha w(t)}$ on ${[}\tau_{x},T)$. {By the continuity of $X$, we conclude that $X$ has these properties throughout $[0,T)$.}   \end{proof}

\subsection{Chordal SLE$_\kappa$ and $2$-SLE$_\kappa$}
If $\ha w(t)=\sqrt\kappa B(t)$, $0\le t<\infty$, where $\kappa>0$ and $B(t)$ is a standard Brownian motion, then the chordal Loewner curve $\eta$ driven by $\ha w$ is known to exist (cf.\ \cite{RS}). We now call it a standard chordal SLE$_\kappa$ curve. It satisfies {the property} that $\eta(0)=0$ and $\lim_{t\to\infty} \eta(t)=\infty$. The  behavior of $\eta$ depends on $\kappa$: if $\kappa\in(0,4]$, $\eta$ is simple and intersects $\R$ only at $0$; if $\kappa\ge 8$, $\eta$ is space-filling, i.e., $\lin\HH=\eta(\R_+)$; if $\kappa\in(4,8)$,  $\eta$ is neither simple nor space-filling. If $D$ is a simply connected domain with two distinct marked boundary points (or more precisely, prime ends) $a$ and  $b$, the chordal SLE$_\kappa$ curve in $D$ from $a$ to $b$ is defined to be the conformal image of a standard chordal SLE$_\kappa$ curve  under a conformal map from $(\HH;0,\infty)$ onto $(D;a,b)$.

Chordal SLE$_\kappa$ satisfies {the} Domain Markov Property (DMP): if $\eta$ is a   chordal  SLE$_\kappa$ curve in $D$ from $a$ to $b$, and $T$ is a stopping time, then conditionally on the part of $\eta$ before $T$ and the event that {$\eta$ has not reached $b$ by time $T$}, the part of $\eta$ after $T$ is a   chordal SLE$_\kappa$ curve from $\eta(T)$ to $b$ in a connected component of $D\sem \eta[0,T]$.

We will focus on the range $\kappa\in(0,8)$ so that SLE$_\kappa$ is non-space-filling. One remarkable property of these chordal SLE$_\kappa$ is reversibility: the time-reversal of a chordal SLE$_\kappa$ curve in $D$ from $a$ to $b$ is a chordal SLE$_\kappa$ curve in $D$ from $b$ to $a$, up to a time-change (\cite{reversibility,MS3}).
Another fact that is important to us is the existence of $2$-SLE$_\kappa$.

\begin{Definition}
Let $D$ be a simply connected domain with pairwise distinct boundary points  $a_1,b_1,a_2,b_2$ such that $a_1$ and $b_1$ together do not separate $a_2$ from $b_2$ on $\pa D$ (and vice versa).  A $2$-SLE$_\kappa$ in $D$ with link pattern $(a_1\lr b_1;a_2\lr b_2)$ is a pair of random curves $(\eta_1,\eta_2)$ in $\lin D$ such that $\eta_j$ connects $a_j$ with $b_j$  for $j=1,2$, and conditionally on {the whole image of} any one curve, the other {curve} is a chordal SLE$_\kappa$ curve in a complement{ary} domain of the given curve in $D$.
\end{Definition}

Because of reversibility, we do not need to specify the orientation of $\eta_1$ and $\eta_2$. If we want to emphasize the orientation, then we use an arrow  $a_1\to b_1$ in the link pattern.
The existence of $2$-SLE$_\kappa$ was proved in \cite{multiple} for $\kappa\in(0,4]$ using {the} Brownian loop measure and in \cite{MS1,MS3} for $\kappa\in(4,8)$ using {the} imaginary geometry theory. The uniqueness of $2$-SLE$_\kappa$ (for a fixed domain and link pattern) was proved in \cite{MS2} (for $\kappa\in(0,4]$) and \cite{MSW} (for $\kappa\in(4,8)$). 

\subsection{SLE$_\kappa({\protect\ulin{\rho}})$ processes}
First introduced in \cite{LSW-8/3}, SLE$_\kappa(\ulin\rho)$ processes are natural variations of SLE$_\kappa$, where one keeps track of additional marked points, often called force points, which may lie on the boundary or interior. For the generality needed here, all force points will lie on the boundary. In this subsection, we review the definition and properties of SLE$_\kappa(\ulin \rho)$ developed in \cite{MS1}.

Let $n\in\N$, $\kappa>0$, $\ulin \rho=(\rho_1,\dots,\rho_n)\in\R^n$. Let $w \in\R$ and  $\ulin v=(v_1,\dots,v_n)\in \R_w^n$. The chordal SLE$_\kappa(\ulin\rho)$ process in $\HH$ started from $w$ with force points $\ulin v$ is the chordal Loewner process driven by the function $\ha w(t)$, which drives chordal Loewner   hulls $K_t$, and solves the SDE
$$d\ha w(t)\aeq \sqrt{\kappa }dB(t)+\sum_{j=1}^n \frac{\rho_j}{\ha w(t)-g^w_{K_t}(v_j)}\,dt,\quad \ha w(0)=w,$$
where $B(t)$ is a standard Brownian motion, and we used Definition \ref{Def-Rw}. We require that for $\sigma\in\{+,-\}$, $\sum_{j:v_j=w^\sigma}\rho_j>-2$. The solution exists uniquely up to the first time (called a continuation threshold) that
$ \sum_{j: \ha v_j(t)=c_{K_t}} \rho_j\le -2$  or $\sum_{j: \ha v_j(t)=d_{K_t}} \rho_j \le -2$, whichever comes first. If a continuation threshold does not exist, then the lifetime is $\infty$.
For each $j$, $\ha v_j(t):=g^w_{K(t)}(v_j)$ is called the force point function started from $v_j$, satisfies  $\ha v_j'\aeq \frac2{\ha v_j-\ha w}$, and is monotonically increasing or decreasing depending on whether $v_j>w$ or $v_j<w$.

A chordal SLE$_\kappa(\ulin\rho)$ process generates a chordal Loewner curve $\eta$ in $\lin\HH$ started from $w$ up to the continuation threshold. If no force point is swallowed by the process at any time, this fact follows from the existence of {the} chordal SLE$_\kappa$ curve (\cite{RS}) and Girsanov{'s} Theorem. The existence of the curve in the general case was proved in \cite[{Theorem 1.3}]{MS1} {(for $\kappa\ne 4$) and \cite[Theorem 1.1.3]{WW-level} (for $\kappa=4$)}. By Proposition \ref{prop-comp-g} and the Markov property of Brownian motion, a chordal SLE$_\kappa(\ulin \rho)$ curve $\eta$ satisfies the following DMP. If $\tau$ is a stopping time for $\eta$, then conditionally on the process before $\tau$ and the event that $\tau$ is less than the lifetime $T$, $\ha w(\tau+t)$ and $\ha v_j(\tau+t)$, $1\le j\le n$, $0\le t<T-\tau$, are the driving function and force point functions for a chordal SLE$_\kappa(\ulin\rho)$ curve $\eta^\tau$ started from $\ha w(\tau)$ with force points at $\ha v_1(\tau),\dots,\ha v_n(\tau)$, and $\eta(\tau+\cdot)=f_{K_\tau}(\eta^\tau)$, where $K_\tau:=\Hull(\eta[0,\tau])$. Here if $\ha v_j(\tau)=\ha w(\tau)$, then $\ha v_j(\tau)$ as a force point is treated as $\ha w(\tau)^{\sign(v_j-w)}$.

We now relabel the force points $v_1,\dots,v_n$  by $v^{{(-)}}_{n_-}\le\cdots\le v^{{(-)}}_{1}{\le w^-<} w{<w^+\le} v^{{(+)}}_1\le \cdots\le v^{{(+)}}_{n_+}$, where $n_-+n_+=n$ ($n_-$ or $n_+$ could be $0$). Then for any $t$ in the lifespan, $\ha v^{{(-)}}_{n_-}(t)\le\cdots\le \ha v^{{(-)}}_1(t)\le \ha w(t)\le \ha v^{{(+)}}_1(t)\le \cdots\le \ha v^{{(+)}}_{n_+}(t)$. If for any $\sigma\in\{-,+\}$ and $1\le k\le n_\sigma$, $\sum_{j=1}^k \rho^{{(\sigma)}}_j>-2$, then the process will never reach a continuation threshold, and so its lifetime is $\infty$, in which case $\lim_{t\to\infty} \eta(t)=\infty$ {(cf.\ \cite[Theorem 1.3]{MS1},\cite[Theorem 1.1.3]{WW-level})}.
If for some $\sigma\in\{+,-\}$ and $1\le k\le n_\sigma$,  $\sum_{j=1}^k \rho^{(\sigma)}_j\ge \frac{\kappa}{2}-2$, then {(cf.\ \cite[Remark 5.3]{MS1})} $\eta$ does not hit $v^{{(\sigma)}}_k$ and the open interval between $v^{{(\sigma)}}_k$ and $v^{{(\sigma)}}_{k+1}$ ($v^{{(\sigma)}} _{n_\sigma+1}:=\sigma\cdot\infty$).
 If $\kappa\in(0,8)$ and for any $\sigma\in\{+,-\}$ and $1\le k\le n_\sigma$,  $\sum_{j=1}^k \rho^{{(\sigma)}}_j> \frac{\kappa}{2}-4$, then for every $x\in\R\sem \{w\}$, a.s.\ $\eta$ does not visit $x$, which implies by Fubini Theorem that a.s.\ $\eta\cap\R$ has Lebesgue measure zero.

 {The last statement is similar to \cite[Lemma 7.16]{MS1} and \cite[Lemma 2.5.2]{WW-level}, and its proof resembles that of \cite[Lemma 5.2]{MS1}. The key idea is that, for a fixed $x\in(w,\infty)$, on the event $E$ that $\eta$ visits $x$, we may compare $\eta$ with an SLE$_\kappa(\rho)$ curve $\eta'$ in $\HH$ started from $w$ with the force point at $x$, where $\ha\rho=\sum\{\rho^{(+)}_j: v^{(+)}_j\le x\}$. It can be show that the law of $\eta$ is absolutely continuous w.r.t.\ that of $\eta'$ on $E$. We have $\rho>\frac\kappa 2-4$ by assumption. Let $f$ be a M\"obius automorphism of $\HH$, which fixes $w$, maps $x$ to $\infty$, and maps $\infty$ to some $y\in(-\infty,x)$. By \cite{SW}, after a time-change, $f\circ \eta'$ up to the time that $\eta'$ separates $x$ from $\infty$ becomes a chordal SLE$_\kappa(\til\rho)$ curve $\til\eta$ in $\HH$ started from $w$ with the force point at $y$ up to the time that it separates $y$ from $\infty$, where $\til\rho:=\kappa-6-\rho<\frac \kappa 2-2$. Let $\til w$ and $\til v$ be respectively the driving function and force point function. Then $\til w-\til v$ is a rescaled Bessel process of dimension $\frac{2\til\rho+4}\kappa+1<2$. So $\til \eta$ a.s.\ hits $(-\infty,y]$, which implies that $\eta'$ a.s.\ does not hit $x$. The same statement then holds for $\eta$ by the absolute continuity between the laws of $\eta$ and $\eta'$ on $E$.
 }


\subsection{Hypergeometric SLE} \label{hSLE}
For $a,b,c\in\C$ such that $c\not\in\{0,-1,-2,\cdots\}$, the hypergeometric function $\,_2F_1(a,b;c;z)$ (cf.\ \cite[Chapter 15]{NIST:DLMF}) is defined  by the Gauss series on the disc $\{|z|<1\}$:
\BGE \,_2F_1(a,b;c;z)=\sum_{n=0}^\infty \frac{(a)_n(b)_n}{(c)_nn!}\,z^n,\label{series}\EDE
where $(x)_n$ is rising factorial: $(x)_0=1$ and $(x)_n=x(x+1)\cdots(x+n-1)$ if $n\ge 1$. It satisfies the ODE
	\BGE z(1-z)F''(z)-[(a+b+1)z-c]F'(z)-ab F(z)=0.\label{ODE-hyper}\EDE
For the purpose of this paper, we {set} $a,b,c$ by $a=\frac{4}\kappa$, $b=1-\frac{4}\kappa$, $c=\frac{8}\kappa$, and define $ F(x) =\,_2F_1(1-\frac{4}\kappa,\frac{4}\kappa; \frac{8}\kappa;x )$.
{Since $c-a-b=\frac 8\kappa-1>0$, by \cite[15.4.20]{NIST:DLMF}, $F$ extends continuously to $[0,1]$ with
$$F(1)=\frac{\Gamma(c)\Gamma(c-a-b)}{\Gamma(c-a)\Gamma(c-b)}=\frac{\Gamma(\frac 8\kappa)\Gamma( \frac\kappa 8-1)}{\Gamma(\frac4\kappa)\Gamma(\frac{12}\kappa-1)}>0.$$}
{We claim that $F>0$ on $[0,1]$.  If $b\ge 0$, then since $a,c>0$, every term of the series in (\ref{series}) is nonnegative for $z\in[0,1)$, and equals $1$ for $n=0$, which implies that $F>0$ on $[0,1)$. Suppose $b<0$. If $F>0$ on $[0,1]$ does not hold, then from $F(1)>0$ and $F(0)=1>0$, we see that there is $x_0\in(0,1)$ such that $F(x_0)<0$, $F'(x_0)=0$, and $F''(x_0)\ge 0$. Since $ab<0$, (\ref{ODE-hyper}) does not hold at $x_0$, which is a contradiction. So the claim is proved.}   Let
$ \til G(x)=\kappa x \frac { F'(x)}{F(x)}+2$.

\begin{Definition}
Let $\kappa\in(0,8)$. 	Let $v_1\le v_2\in  [0^+,+\infty]$ or  $v_1\ge v_2\in [-\infty,0^-]$. Suppose $\ha w(t)$, $0\le t<\infty$, solves the following SDE:
$$
  d \ha w(t)\aeq \sqrt\kappa dB(t)+\Big(\frac{1}{\ha w(t)-\ha v_1(t)}-\frac 1{\ha w(t)-\ha v_2(t)}\Big)\til G\Big(\frac{\ha w(t)-\ha v_1(t)}{\ha w(t)-\ha v_2(t)}\Big)dt,\quad \ha w(0)=0,
$$
where $B(t)$ is a standard Brownian motion,  $\ha v_j(t)=g_{K_t}^0(v_j)$, $j=1,2$, and $K_t$ are chordal Loewner hulls driven by $\ha w$. 
The chordal Loewner curve driven by $\ha w$ is called a hypergeometric SLE$_\kappa$, or simply hSLE$_\kappa$, curve in $\HH$ from $0$ to $\infty$  with force points $v_1,v_2$. We call $v_j(t)$ the force point function started from $v_j$, $j=1,2$.

{If $f$ maps $\HH$ conformally onto a simply connected domain $D$, then the $f$-image of an hSLE$_\kappa$ curve in $\HH$ from $0$ to $\infty$ with force points $v_1,v_2$ is called an hSLE$_\kappa$ curve in $D$ from $f(0)$ to $f(\infty)$ with force points $f(v_1),f(v_2)$}
\label{Def-hSLE}
\end{Definition}

{
\begin{Remark}
The existence of the solutions of the SDE follows from Girsanov's theorem. We start with an SLE$_\kappa(2,-2)$ curve $\eta$ in $\HH$ started from $w$ with force points $v_1,v_2$. We assume $0^+\le v_1<v_2$ by symmetry. Let $\ha w$ be the driving function, and $\ha v_j$ be the force point function started from $v_j$, $j=1,2$. Let $g_t$ be the chordal Loewner maps. Define $R(t)=\frac{\ha v_1(t)-\ha w(t)}{\ha v_2(t)-\ha w(t)}\in [0,1)$ and $I(t)=\frac{|\ha v_2(t)-\ha v_1(t)|}{g_t'(v_2)}$ for $0\le t<\tau_{v_2}$. By Loewner's equation, $I_j$ is decreasing. Then we define
 $$M(t)=\frac{F(R(t))}{F(R(0))}\Big(\frac{ I(t)}{I(0)}\Big)^{1-\frac 4\kappa}, \quad 0\le t<\tau_{v_2}.$$
Using It\^o's formula and (\ref{ODE-hyper}), we see that $M$ is a positive local martingale. If $\kappa\le 4$, then a.s.\ $\tau_{v_2}=\infty$, and $M$ is defined on $(0,\infty)$. If $\kappa>4$, then a.s.\ $\tau_{v_2}<\infty$, and as $t\uparrow \tau_{v_2}$, $M(t)$ converges to a positive number. We then extend $M$ to $[0,\infty)$ such that $M$ is constant and equals $\lim_{t\uparrow \tau_{v_2}} M(t)$ on $[\tau_{v_2},\infty)$. After the extension, $M$ is a positive continuous local martingale defined on $[0,\infty)$. Suppose $T$ is any stopping time such that $M_{T\wedge \cdot}$ is bounded. We may weight the underlying probability measure by $M(T)$ and get another probability measure. Under the new measure, the $\ha w,\ha v_1,\ha v_2$ satisfy the SDE in Definition \ref{Def-hSLE} for some standard Brownian notion $B$ up to the time $T$.
\end{Remark}
}

{
\begin{Remark}
  The definition of SLE using hypergeometric functions first appeared in \cite{kappa-rho}, in which the new SLE were called intermediate SLE$_\kappa(\rho)$. The notion of  hypergeometric SLE and  hSLE first appeared in W.\ Qian's \cite{QW}, which generalized the intermediate SLE$_\kappa(\rho)$. Here we follow the definition of hSLE used in the later paper \cite{Wu-hSLE}, in which hSLE$_\kappa$ agrees with intermediate SLE$_\kappa(2)$, and so is only a very special case of Qian's hSLE.
\end{Remark}
}

Hypergeometric SLE is important because
if $(\eta_1,\eta_2)$ is a $2$-SLE$_\kappa$ in $D$ with link pattern $(a_1\to b_1;a_2\to b_2)$, then for $j=1,2$, the marginal law of $\eta_j$ is that of an hSLE$_\kappa$ curve in $D$ from ${a_j}$ to ${b_j}$ with force points ${b_{3-j}}$ and ${a_{3-j}}$ (cf.\ \cite[Proposition 6.10]{Wu-hSLE}).

Using the standard argument in \cite{SW}, we obtain the following proposition describing an hSLE$_\kappa$ curve in $\HH$ {``}in the chordal coordinate{''} in the case that the target is not $\infty$.

\begin{Proposition}
  Let $w_0\ne w_\infty\in\R$. Let $v_1\in \R_{w_0}\cup\{\infty\}\sem \{w_\infty\}$ and $v_2\in\R_{w_\infty}\cup \{\infty\}\sem \{w_0\}$ be such that the cross ratio $R:=\frac{(w_0-v_1)(w_\infty-v_2)}{(w_0-v_2)(w_\infty-v_1)}\in [0^+,1)$. Let $\kappa\in(0,8)$. Let $\ha\eta$ be an hSLE$_\kappa$ curve in $\HH$ from $w_0$ to $w_\infty$ with force points at $v_1,v_2$. Stop $\ha\eta$ at the first time that it separates $w_\infty$ from $\infty$, and parametrize the stopped curve by $\HH$-capacity. Then the new curve, denoted by $\eta$, is the chordal Loewner curve driven by some function $\ha w_0$, which satisfies the following SDE with initial value $\ha w_0(0)=w_0$:
  \begin{align*}
d\ha w_0(t)\aeq &\sqrt\kappa dB(t)+ \frac{\kappa-6}{\ha w_0(t)-\ha w_\infty(t)}\, dt+ \\
 & + \Big( \frac 1{\ha w_0(t)-\ha v_1(t)}  -\frac 1 {\ha w_0(t)-\ha v_2(t)}\Big )\cdot \til G\Big( \frac{ (\ha w_0(t) -\ha v_1(t) ) (\ha v_2(t)-\ha w_\infty(t))}{ (\ha w_0(t)-\ha v_2(t)) (\ha v_1(t)-\ha w_\infty(t))}\Big)\,dt,
\end{align*}
where $B(t)$ is a standard Brownian motion, $\ha w_\infty(t)=g_{K_t} (w_\infty)$ and $\ha v_j(t)=g_{K_t}^{w_0}(v_j)$, $j=1,2$, and $K_t$ are the chordal Loewner hulls driven by $\ha w_0$.
\label{Prop-iSLE-2}
\end{Proposition}

\begin{Definition}
We call the $\eta$ in Proposition \ref{Prop-iSLE-2} an hSLE$_\kappa$ curve in $\HH$ from $w_0$ to $w_\infty$ with force points at $v_1,v_2$, {``}in the chordal coordinate{''}; call $\ha w_0$ the driving function; and call $\ha w_\infty$, $\ha v_1$ and $\ha v_2$ the force point functions respectively started from $w_\infty$, $v_1$ and $v_2$. \label{Def-iSLE-chordal}
\end{Definition}

\begin{Proposition}
We adopt the notation in the last proposition. Let $T$ be the first time that $w_\infty$ or $v_2$ is swallowed by the hulls. Note that $|\ha w_0-\ha w_\infty|$, $|\ha v_1-\ha v_2|$, $\ha w_0-\ha v_2|$, and $|\ha w_\infty-\ha v_1|$ are all positive on $[0,T)$. We define $M$ on $[0,T)$ by $M=G_1(\ha w_0,\ha v_1;\ha w_\infty,\ha v_2)$, where $G_1$ is given by (\ref{G1(w,v)}).
 Then $M$ is a positive local martingale, and if we tilt the law of $\eta$ by $M$, then we get the law of a chordal SLE$_\kappa(2,2,2)$ curve in $\HH$ started from $w_0$ with force points $w_\infty$, $v_1$ and $v_2$. More precisely, if $\tau<T$ is a stopping time such that $M$ is uniformly bounded on $[0,\tau]$, then if we weight the underlying probability measure by $M(\tau)/M(0)$, then we get a probability measure under which the law of $\eta$ stopped at the time $\tau$  is that of a chordal SLE$_\kappa(2,2,2)$ curve in $\HH$ started from $w_0$ with force points $w_\infty$, $v_1$ and $v_2$ stopped at the time $\tau$. \label{Prop-iSLE-3}
\end{Proposition}
\begin{proof}
	This follows from   straightforward applications of It\^o's formula and Girsanov{'s} Theorem, where we use  (\ref{ODE-hyper}), Propositions \ref{Prop-cd-continuity'} and \ref{Prop-iSLE-2}. Actually, the calculation could be simpler if we tilt the law of a chordal SLE$_\kappa(2,2,2)$ curve by $M^{-1}$ to get an hSLE$_\kappa $ curve.
\end{proof}

\subsection{Two-parameter stochastic processes}
In this subsection we briefly recall the framework in \cite[Section 2.3]{Two-Green-interior}. {The framework will help us to study stochastic processes  defined on some random subset $\cal D$ of $\R_+^2$, where every element in $\cal D$ is understood as a two-dimensional random variable.}
We assign a partial order $\le$ to $\R_+^2=[0,\infty)^2$  such that $\ulin t=(t_+,t_-)\le(s_+,s_-)= \ulin s$ iff $t_+\le s_+$ and $t_-\le s_-$. It has a minimal element $\ulin 0=(0,0)$. We write $\ulin t<\ulin s$ if $t_+<s_+$ and $t_-<s_-$. We define $\ulin t\wedge \ulin s=(t_1\wedge s_1,t_2\wedge s_2)$. Given $\ulin t,\ulin s\in\R_+^2$, we define $[\ulin t,\ulin s]=\{\ulin r\in{\R_+^2}:\ulin t\le \ulin r\le \ulin s\}$. Let $\ulin e_+=(1,0)$ and $\ulin e_-=(0,1)$. So $(t_+,t_-)=t_+\ulin e_++ t_-\ulin e_-$.

	\begin{Definition}
An ${\R_+^2}$-indexed filtration $\F$ on a measurable space $\Omega$ is a family of  $\sigma$-algebras $\F_{\ulin t}$, $\ulin t\in\R_+^2$, on  $\Omega$ such that $\F_{\ulin t}\subset \F_{\ulin s}$ whenever $\ulin t\le \ulin s$. Define $\lin\F$ by $\lin\F_{\ulin t}=\bigcap_{\ulin s>\ulin t}\F_{\ulin s}$, $\ulin t\in\R_+^2$. Then we call $\lin\F$ the right-continuous augmentation of $\F$. We say that  $\F$ is right-continuous if $\lin\F=\F$. A process $X=(X({\ulin t}))_{\ulin t\in\R_+^2}$ defined on $\Omega$ is called $\F$-adapted if for any $\ulin t\in\R_+^2$, $X({\ulin t})$ is $\F_{\ulin t}$-measurable. It is called continuous if   $\ulin t\mapsto X({\ulin t})$ is sample-wise continuous.
	\end{Definition}

For the rest of this subsection, let $\F$ be an $\R_+^2$-indexed filtration with right-continuous augmentation $\lin\F$, and let $\F_{\ulin\infty}=\bigvee_{\ulin t\in\R_+^2} \F_{\ulin t}$.

\begin{Definition}
A $[0,\infty]^2$-valued random element $\ulin T$ is called an  $\F$-stopping time if for any deterministic $\ulin t\in{\R_+^2}$, $\{\ulin T\le \ulin t\}\in \F_{\ulin t}$. It is called finite if $\ulin T\in\R_+^2$, and is called bounded if there is a deterministic $\ulin t\in \R_+^2$ such that $\ulin T\le \ulin t$.
For an  $\F$-stopping time $\ulin T$, we define a new $\sigma$-algebra $\F_{\ulin T}$ by
$\F_{\ulin T}=\{A\in\F_{\ulin\infty}:A\cap \{\ulin T\leq \ulin t\}\in \F_{\ulin t}, \forall \ulin t\in{\R_+^2}\}$. 
\end{Definition}

The following proposition follows from a standard argument.

\begin{Proposition}
The right-continuous augmentation of $\lin\F$ is itself, and so $\lin\F $ is right-continuous. A $[0,\infty]^2$-valued random map $\ulin T$ is an $\lin\F$-stopping time if and only if $\{\ulin T<\ulin t\}\in\F_{\ulin t}$ for any $\ulin t\in\R_+^2$. For an $\lin\F$-stopping time $\ulin T$, $A\in\lin\F_{\ulin T}$ if and only if $A\cap\{\ulin T<\ulin t\}\in\F_{\ulin t}$ for any $\ulin t\in\R_+^2$. If $(\ulin T^n)_{n\in\N}$ is a decreasing sequence of  $\lin\F$-stopping times, then $\ulin T:=\inf_n \ulin T^n$ is also an $\lin\F$-stopping time, and $\lin\F_{\ulin T}=\bigcap_n \lin\F_{\ulin T^n}$.
  \label{right-continuous-0}
\end{Proposition}

	\begin{Definition}
		A relatively open subset $\cal R$ of $\R_+^2$ is called a history complete region, or simply an HC region, if for any $\ulin t\in \cal R$, we have $[\ulin 0, \ulin t]\subset\cal R$. {We use the name because we view (the rectangle) $[\ulin 0, \ulin t]$ as the history of $\ulin t$.} Given an HC region $\cal R$, for $\sigma\in\{+,-\}$, define   $T^{\cal R}_\sigma:\R_+\to\R_+\cup\{\infty\}$ by $T^{\cal R}_\sigma(t)=\sup\{s\ge 0:s \ulin e_\sigma+t\ulin e_{-\sigma}\in{\cal R}\}$, where we set $\sup\emptyset =0$. {See Figure \ref{fig-HC} for an illustration}

An HC region-valued random element $\cal D$ is called an $\F$-stopping region if for any $\ulin t\in\R_+^2$, $\{\omega\in\Omega:\ulin t\in {\cal D}(\omega)\}\in\F_{ \ulin t}$. A random function $X({\ulin t})$ with a random domain $\cal D$ is called an $\F$-adapted HC process if $\cal D$ is an $\F$-stopping region, and for every  $\ulin t\in\R_+^2$, $X_{\ulin t}$ restricted to $\{\ulin t\in\cal D\}$ is $\F_{ \ulin t}$-measurable. \label{Def-HC}
	\end{Definition}

\begin{figure}
	\begin{center}
		\includegraphics[width=3.5in]{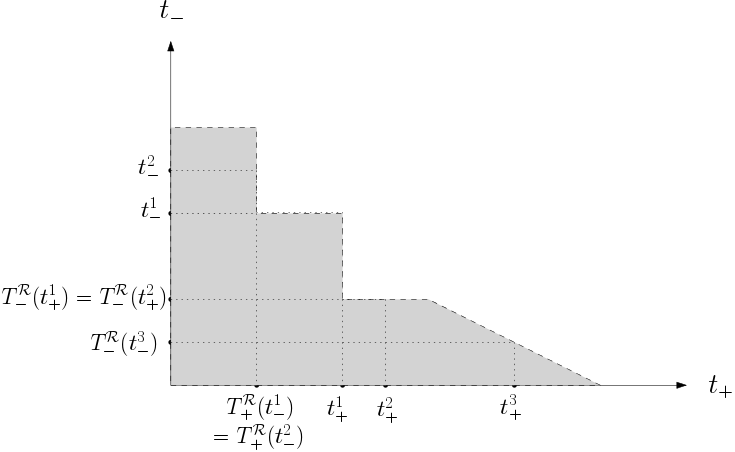}
	\end{center}
	\caption{{ The figure above illustrates an HC region $\cal R$ (grey), the function $T^{\cal R}_+$ valued at $t_-^1,t_-^2$, and the function $T^{\cal R}_-$ valued at $t_+^1,t_+^2,t_+^3$.}} \label{fig-HC}
\end{figure}

The following propositions are
 \cite[Lemmas 2.7 and  2.9]{Two-Green-interior}.
	
	\begin{Proposition}
		Let $\ulin T$ and $\ulin S$ be two $\F$-stopping times. Then (i)  $\{\ulin T\le \ulin S\} \in\F_{\ulin S}$; (ii) if $\ulin S$ is a constant $\ulin s\in\R_+^2$, then $\{\ulin T\le \ulin S\} \in\F_{ \ulin T}$; and (iii) if $f$ is an $\F_{ \ulin T}$-measurable function, then ${\bf 1}_{\{\ulin T\le \ulin S\}}f$ is $\F_{ \ulin S}$-measurable. In particular, if $\ulin T\le \ulin S$, then $\F_{ \ulin T}\subset \F_{ \ulin S}$. \label{T<S}
	\end{Proposition}

We will need the following proposition to do localization. The reader should note that for an $\F$-stopping time $\ulin T$ and a deterministic time $\ulin t\in\R_+^2$, $\ulin T\wedge \ulin t$ may not be an $\F$-stopping time. This is the reason why we introduce a more complicated stopping time.

\begin{Proposition}
	Let $\ulin T$ be an $\F$-stopping time. Fix a deterministic time $\ulin t\in\R_+^2$. Define $\ulin T^{\ulin t}$ such that if $\ulin T\le \ulin t$, then $\ulin T^{\ulin t}=\ulin T$; and if $\ulin T\not\le \ulin t$, then $\ulin T^{\ulin t}=\ulin t$. Then $\ulin T^{\ulin t}$ is an $\F$-stopping time bounded above by $\ulin t$, and $\F_{\ulin T^{\ulin t}}$ agrees with $\F_{\ulin T}$ on $\{\ulin T\le \ulin t\}$, i.e., $\{\ulin T\le \ulin t\}\in \F_{\ulin T^{\ulin t}}\cap \F_{\ulin T}$, and for any $A\subset  \{\ulin T\le \ulin t\}$, $A\in \F_{\ulin T^{\ulin t}}$ if and only if  $A\in \F_{\ulin T}$.  \label{prop-local}
\end{Proposition}
\begin{proof}
	Clearly $\ulin T^{\ulin t}\le \ulin t$. Let $\ulin s\in\R_+^2$. If $\ulin t\le \ulin s$, then $\{\ulin T^{\ulin t}\le \ulin s\}$ is the whole space. If $\ulin t\not\le \ulin s$, then $\{\ulin T^{\ulin t}\le \ulin s\}=\{\ulin T\le \ulin t\}\cap \{\ulin T\le \ulin s\}=\{\ulin T\le \ulin t\wedge \ulin s\}\in \F_{ \ulin t\wedge \ulin s}\subset \F_{ \ulin s}$. So $\ulin T^{\ulin t}$ is an $\F$-stopping time.

By Proposition \ref{T<S}, $\{\ulin T\le \ulin t\}\in \F_{\ulin T}$. Suppose $A\subset  \{\ulin T\le \ulin t\}$ and $A\in  \F_{\ulin T}$.  Let $\ulin s\in\R_+^2$. If $\ulin t\le \ulin s$, then   $A\cap \{\ulin T^{\ulin t}\le \ulin s\}=A=A\cap \{\ulin T\le \ulin t\} \in\F_{\ulin t}\subset \F_{\ulin s}$. If $\ulin t\not\le \ulin s$, then $A\cap \{\ulin T^{\ulin t}\le \ulin s\}= A\cap\{\ulin T\le \ulin t\wedge \ulin s\}\in \F_{\ulin t\wedge \ulin s}\subset \F_{\ulin s}$.
So $A\in\F_{\ulin T^{\ulin t}}$. In particular, $\{\ulin T\le \ulin t\}\in\F_{\ulin T^{\ulin t}}$. On the other hand, suppose  $A\subset  \{\ulin T\le \ulin t\}$ and $A\in  \F_{\ulin T^{\ulin t}}$. Let $\ulin s\in\R_+^2$. If $\ulin t\le \ulin s$, then $A\cap \{\ulin T\le \ulin s\}=A=A\cap  \{\ulin T^{\ulin t}\le \ulin t\}\in\F_{\ulin t}\subset\F_{\ulin s}$. If $\ulin t\not\le \ulin s$, then $A\cap \{\ulin T\le \ulin s\}=A\cap \{\ulin T\le \ulin t\}\cap \{\ulin T\le \ulin s\}=A\cap\{\ulin T^{\ulin t}\le \ulin s\}\in\F_{\ulin s}$. Thus, $A\in \F_{\ulin T}$. So for $A\subset  \{\ulin T\le \ulin t\}$, $A\in \F_{\ulin T^{\ulin t}}$ if and only if  $A\in \F_{\ulin T}$.
\end{proof}

Now we fix a probability measure $\PP$,  and let $\EE$ denote the corresponding expectation.
	
	\begin{Definition}
	An  $\F$-adapted process $(X_{\ulin t})_{\ulin t\in\R_+^2}$ is called an  $\F$-martingale (w.r.t.\ $ \PP$) if  for any $\ulin s\le \ulin t\in\R_+^2$, a.s.\ $\EE[X_{\ulin t}|\F_{\ulin s}]=X_{\ulin s}$. If there is $\zeta\in L^1(\Omega,\F,\PP)$ such that $X_{\ulin t}=\EE[\zeta|\F_{ \ulin t}]$ for every $\ulin t\in\R_+^2$, then we say that $X$ is an $\F$-martingale closed by $\zeta$.
	\end{Definition}

The following proposition is \cite[Lemma 2.11]{Two-Green-interior}.
		
	\begin{Proposition} [Optional Stopping Theorem]
		Suppose $X$ is a continuous $\F$-martingale.
The following are true. (i) If $X$ is closed by $\zeta$, then for any finite $\F$-stopping time $\ulin T$, $X_{\ulin T}=\EE[\zeta|\F_{ \ulin T}]$. (ii) If $\ulin T\le \ulin S$ are two bounded $\F$-stopping times, then $\EE[X_{\ulin S}|\F_{\ulin T}]=X_{\ulin T}$.\label{OST}
	\end{Proposition}

\subsection{Jacobi polynomials}
For $\alpha,\beta>-1$,  Jacobi polynomials (\cite[Chapter 18]{NIST:DLMF}) $P^{(\alpha,\beta)}_n(x)$, $n=0,1,2,3,\dots$, are a class of classical orthogonal polynomials with respect to the weight $\Psi^{(\alpha,\beta)}(x):={\bf 1}_{(-1,1)}(1-x)^\alpha(1+x)^\beta$. This means that each $P^{(\alpha,\beta)}_n(x)$ is a polynomial of degree $n$, and for the inner product defined by $\langle f,g\rangle_{\Psi^{(\alpha,\beta)}}:=\int_{-1}^1 f(x)g(x)\Psi^{(\alpha,\beta)}(x)dx$, we have $\langle P^{(\alpha,\beta)}_n, P^{(\alpha,\beta)}_m\rangle_{\Psi^{(\alpha,\beta)}}=0$ when $n\ne m$. The normalization is that $P^{(\alpha,\beta)}_n(1)=\frac{\Gamma(\alpha+n+1)}{n!\Gamma(\alpha+1)}$,  $P^{(\alpha,\beta)}_n(-1)=(-1)^n\frac{\Gamma(\beta+n+1)}{n!\Gamma(\beta+1)}$, and
\BGE\|  P^{(\alpha,\beta)}_n\|_{\Psi^{(\alpha,\beta)}}^2 =\frac{2^{\alpha+\beta+1}}{2n+\alpha+\beta+1}\cdot \frac{\Gamma(n+\alpha+1)\Gamma(n+\beta+1)} {n!\Gamma(n+\alpha+\beta+1)}.\label{norm}\EDE
For each $n\ge 0$,   $P^{(\alpha,\beta)}_n(x)$ is a solution of the second order differential equation:
\BGE (1-x^2)y''-[(\alpha+\beta+2)x+(\alpha-\beta)]y'+n(n+\alpha+\beta+1)y=0.\label{Jacobi-ODE}\EDE
When $\max\{\alpha,\beta\}>-\frac 12$, we have an exact value of the supernorm of $P^{(\alpha,\beta)}_n$ over $[-1,1]$:
\BGE \|P^{(\alpha,\beta)}_n\|_\infty= \max\{|P^{(\alpha,\beta)}_n(1)|,|P^{(\alpha,\beta)}_n(-1)|\} =\frac{\Gamma(\max\{\alpha,\beta\}+n+1)}{n!\Gamma(\max\{\alpha,\beta\}+1)}.\label{super-exact}\EDE
For  general $\alpha,\beta>-1$, we get an upper bound {on} $\|P^{(\alpha,\beta)}_n\|_\infty$ using (\ref{super-exact}), the exact value of  $P^{(\alpha,\beta)}_n(1)$, and the derivative formula $\frac d{dx} P^{(\alpha,\beta)}_n(x)=\frac{ \alpha+\beta+n+1}{2} P^{(\alpha+1,\beta+1)}_{n-1}(x)$ for $n\ge 1$:
\BGE \|P^{(\alpha,\beta)}_n\|_\infty\le \frac{\Gamma(\alpha+n+1)}{n!\Gamma(\alpha+1)}+   ({ \alpha+\beta+n+1} )\cdot \frac{\Gamma(\max\{\alpha,\beta\}+n+1)}{\Gamma(n)\Gamma(\max\{\alpha,\beta\}+2)}.\label{super-upper}\EDE

\section{Deterministic Ensemble of Two Chordal Loewner Curves}  \label{section-deterministic}
In this section, we develop a framework about commuting pairs of deterministic chordal Loewner curves, which will be needed to study the commuting pairs of random chordal Loewner curves in the next two sections. The major length of this section is caused by the fact that we allow the two Loewner curves {to} have intersections. This is needed in order to handle the case $\kappa\in(4,8)$. The ensemble without intersections appeared earlier in \cite{reversibility,duality}.

\subsection{Ensemble with possible intersections}  \label{section-deterministic1}
Let $w_-<w_+\in\R$. Suppose for $\sigma\in\{+,-\}$, $\eta_\sigma(t)$, $0\le t<T_\sigma$, is a chordal Loewner curve (with speed $1$) driven by $\ha w_\sigma$   started from $w_\sigma$, such that $\eta_+$ does not hit $(-\infty, w_-]$, and $\eta_-$ does not hit $[w_+,\infty)$. Let   $K_\sigma(t_\sigma)=\Hull(\eta[0,t_\sigma])$, $0\le t_\sigma<T_\sigma$, $\sigma\in\{+,-\}$. Then $K_\sigma(\cdot)$ are chordal Loewner hulls driven by $\ha w_\sigma$, $\hcap_2(K_\sigma(t_\sigma))=t_\sigma$, and by Proposition \ref{Loewner-chain},
\BGE \{\ha w_\sigma(t_\sigma)\}=\bigcap_{\del>0} \lin{K_\sigma(t_\sigma+\del)/K_\sigma(t_\sigma)},\quad 0\le t_\sigma<T_\sigma.\label{haw=}\EDE
The corresponding chordal Loewner maps are $g_{K_\sigma(t)}$, $0\le t<T_\sigma$, $\sigma\in\{+,-\}$. From the assumption on $\eta_+$ and $\eta_-$ we get
\BGE a_{K_-(t_-)}\le w_-< a_{K_+(t_+)},\quad b_{K_-(t_-)}< w_+\le b_{K_+(t_+)},\quad\mbox{for } t_\sigma\in(0,T_\sigma),\quad \sigma\in\{+,-\}. \label{lem-aabb}\EDE
Since each $K_\sigma(t)$ is generated by a curve, $f_{K_\sigma(t)}$ is well defined.
Let ${\cal I}_\sigma =[0,T_\sigma)$, $\sigma\in\{+,-\}$, and for $\ulin t=(t_+,t_-)\in\I_+\times\I_-$, define
\BGE K(\ulin t)=\Hull(\eta_+[0,t_+]\cup \eta_-[0,t_-]),\quad \mA(\ulin t)=\hcap_2(K(\ulin t)),\quad H(\ulin t)=\HH\sem K(\ulin t).\label{KmA}\EDE It is obvious that $K(\cdot,\cdot)$ and $\mA(\cdot,\cdot)$ are increasing ({although not necessarily strictly}) in both variables. Since $\pa K(t_+,t_-)$ is locally connected, $f_{K(t_+,t_-)}$ is well defined.  For $\sigma\in\{+,-\}$, $t_{-\sigma}\in\I_{-\sigma}$ and $t_\sigma\in\I_\sigma$, define $K_{\sigma}^{t_{-\sigma}}(t_\sigma)=K(t_+,t_-)/K_{-\sigma}(t_{-\sigma})$. Then we have
\BGE g_{K(t_+,t_-)}=g_{K_{+}^{t_-}(t_+)}\circ g_{K_-(t_-)}=g_{K_{-}^{t_+}(t_-)}\circ g_{K_+(t_+)}.\label{circ-g}\EDE
By (\ref{lem-aabb}) and the assumption on $\eta_+,\eta_-$, we have $a_{K(t_+,t_-)}=a_{K_-(t_-)}$ if $t_->0$, and $b_{K(t_+,t_-)}=b_{K_+(t_+)}$ if $t_+>0$.
\begin{Lemma}
    For any $t_+\le t_+'\in\I_+$ and $t_-\le t_-'\in \I_-$, we have
    \BGE \mA(t_+',t_-')-\mA(t_+',t_-)-\mA(t_+,t_-')+\mA(t_+,t_-)\le 0.\label{mA-concave}\EDE
     {In particular}, $\mA$ is Lipschitz continuous with constant $1$ in any variable, and so is continuous on $\I_+\times \I_-$.
    \label{lem-lips}
\end{Lemma}
\begin{proof}
  Let $t_+\le t_+'\in\I_+$ and $t_-\le t_-'\in \I_-$. Since $K(t_+',t_-)$ and $K(t_+,t_-')$ together generate the $\HH$-hull $K(t_+',t_-')$, and they both contain $K(t_+,t_-)$, we obtain (\ref{mA-concave}) from Proposition \ref{hcap-concave}.  The rest of the statements follow easily from (\ref{mA-concave}), the monotonicity of $\mA$, and the fact that $\mA(t_\sigma \lin e_\sigma)=t_\sigma$ for any $t_\sigma\in\I_\sigma$, $\sigma\in\{+,-\}$.
\end{proof}

\begin{Definition}
Let $\cal D\subset {\cal I}_+\times {\cal I}_-$ be an HC region as in Definition \ref{Def-HC}. Suppose that there are dense subsets $\I_+^*$ and $\I_-^*$ of $\I_+$ and $\I_-$, respectively, such that for any $\sigma\in\{+,-\}$ and $t_{-\sigma}\in {\cal I}^*_{-\sigma}$, the following two conditions hold:
\begin{enumerate}
  \item[(I)] $K_\sigma^{t_{-\sigma}}(t_{{\sigma}})$, $0\le t_\sigma<T^{\cal D}_\sigma(t_{-\sigma})$, are chordal Loewner hulls generated by a chordal Loewner curve, denoted by $\eta_{\sigma}^{t_{-\sigma}}$, with some speed.
  \item [(II)] $ \eta_{\sigma}^{t_{-\sigma}}[0,T^{\cal D}_\sigma(t_{-\sigma}))\cap \R$ has Lebesgue measure zero.
\end{enumerate}
Then we call $(\eta_+,\eta_-;{\cal D})$ a commuting pair of chordal Loewner curves, and call $K(\cdot,\cdot)$ and $\mA(\cdot,\cdot)$ the hull function and the capacity function, respectively, for this pair.
  \label{commuting-Loewner}
\end{Definition}

\begin{Remark}
{The theory developed in this section will later be applied to the random setting. In Section \ref{section-commuting-SLE-kappa-rho}, we will study  a commuting pair of SLE$_\kappa(2,\ulin\rho)$ curves $(\eta_+,\eta_-)$. They have random lifespan $\I_+$ and $\I_+$, and satisfy the property that, for any $\sigma\in\{+,-\}$ and $t_{-\sigma}\in \R_+$, a.s.\ on the event $\{t_{-\sigma}\in {\cal I}_{-\sigma}\}$, Conditions (I) and (II) are satisfied. Letting $\I_\pm^*=I_\pm\cap\Q$, we see that $\eta_+$ and $\eta_-$ a.s.\ satisfy the condition of Definition \ref{commuting-Loewner}. So they are a.s.\ a (random) commuting pair of chordal Loewner curves.
}

 Later in Lemma \ref{Lebesgue} we will show that for the commuting pair in Definition \ref{commuting-Loewner}, Conditions (I) and (II) actually hold for all $t_{-\sigma}\in {\cal I}_{-\sigma}$, $\sigma\in\{+,-\}$.
\end{Remark}

From now on, let $(\eta_+,\eta_-;{\cal D})$ be a commuting pair of chordal Loewner curves, and let $\I_+^*$ and $\I_-^*$ be as in Definition \ref{commuting-Loewner}.

\begin{Lemma}
  $K(\cdot,\cdot)$ and $\mA(\cdot,\cdot)$ restricted to $\cal D$ are strictly increasing in both variables. \label{lem-strict}
\end{Lemma}
\begin{proof}
By Condition (I), for any $\sigma\in\{+,-\}$ and $t_{-\sigma}\in\I_{-\sigma}^*$, $t\mapsto K(t_{-\sigma}\ulin e_{-\sigma}+t\ulin e_\sigma)$ and $t\mapsto \mA(t_{-\sigma}\ulin e_{-\sigma}+t\ulin e_\sigma)$ are strictly increasing on $[0,T^{\cal D}_\sigma(t_{-\sigma}))$. By (\ref{mA-concave}) and the denseness of $\I_{-\sigma}^*$ in $\I_{-\sigma}$, this property extends to any $t_{-\sigma}\in\I_{-\sigma}$.
\end{proof}

In the rest of the section, when we talk about $K(t_+,t_-)$, $\mA(t_+,t_-)$, $K_{+}^{t_-}(t_+)$ and $K_{-}^{t_+}(t_-)$, it is always implicitly assumed that $(t_+,t_-)\in\cal D$. So we may now simply say that $K(\cdot,\cdot)$ and $\mA(\cdot,\cdot)$ are strictly increasing in both variables.

\begin{Lemma} We have the following facts.
  \begin{enumerate}
    \item [(i)]  Let $\ulin a=(a_+,a_-)\in\cal D$ and $L={\rad_0}(K(a_+,a_-))$. Let $\sigma\in\{+,-\}$. Suppose $ t_\sigma<t_\sigma'\in [0,a_\sigma]$ satisfy that $\diam(\eta_\sigma[t_\sigma,t_\sigma'])<r$ for some $r\in(0,L)$. Then for any $t_{-\sigma}\in[0,a_{-\sigma}]$,
  $\diam(K_{\sigma}^{t_{-\sigma}}({t_\sigma'})/K_{\sigma}^{t_{-\sigma}}(t_\sigma))\le 10\pi L \log(L/r)^{-1/2}$.
   \item [(ii)] For any  $ (a_+,a_-)\in \cal D$ and $\sigma\in\{+,-\}$,
$$\lim_{\del\downarrow 0}\, \sup_{0\le t_\sigma\le a_\sigma}\,\sup_{t_\sigma'\in(t_\sigma,t_\sigma+\delta)}\, \sup_{0\le t_{-\sigma}\le a_{-\sigma}}\, \sup_{z\in\C\sem K_{\sigma}^{t_{-\sigma}}(t_\sigma')^{\doub} }\,|g_{K_{\sigma}^{t_{-\sigma}}(t_\sigma')}(z)-g_{K_{\sigma}^{t_{-\sigma}}(t_\sigma)}(z)|=0.$$
    \item [(iii)] The map $(\ulin t,z)\mapsto g_{K(\ulin t)}(z)$ is continuous on $\{(\ulin t,z) :\ulin t\in{\cal D}, z\in\C\sem K(\ulin t)^{\doub}\}$.
  \end{enumerate} \label{lem-uniform}
\end{Lemma}
\begin{proof}
  (i) Suppose $\sigma=+$ by symmetry. We {first} assume that $a_\pm \in\I_\pm^*$.
  Let $\Delta \eta_+=\eta_+[t_+,t_+']$ and $S=\{|z-\eta_+(t_+)|=r\}$.  By assumption, $\Delta\eta_+\subset \{|z-\eta_+(t_+)|<r\}$. By Lemma \ref{lem-strict}, there is $z_*\in \Delta \eta_+\cap H(t_+,a_-)\subset H(t_+,t_-)$. Since $z_*\in \{|z-\eta_+(t_+)|<r\}$, the set $ S\cap H(t_+,t_-)$ has a connected component, denoted by $J$, which separates $z_*$ from $\infty$ in $H(t_+,t_-)$. Such $J$ is a crosscut of $H(t_+,t_-)$\footnote{{A crosscut of a domain $D$ is  the image of a simple curve $\gamma:(\alpha,\beta)\to D$ such that the limits $\lim_{t\to \alpha^+} \gamma(t)$ and $\lim_{t\to \beta^-}\gamma(t)$ both exist and lie on $\pa D$.}}, which divides $H(t_+,t_-)$ into two domains, where the bounded domain, denoted by $D_J$, contains $z_*$.

    Now $\Delta \eta_+\cap H(t_+,a_-)\subset H(t_+,a_-)\sem J$. We claim that $\Delta \eta_+\cap H(t_+,a_-)$ is contained in one connected component of $H(t_+,a_-)\sem J$. Note that $J\cap H(t_+,a_-)$ is a disjoint union of crosscuts, each of which divides $H(t_+,a_-)$ into two domains.  To prove the claim, it suffices to show that, for each connected component $J_0$ of $J\cap H(t_+,a_-)$, $\Delta \eta_+\cap H(t_+,a_-)$ is contained in one connected component of $H(t_+,a_-)\sem J_0$. Suppose that this is not true for some $J_0$. Let $J_e=g_{K(t_+,a_-)}(J_0)$. Then $J_e$ is a crosscut of $\HH$, which divides $\HH$ into two domains, both of which intersect $\Delta\ha\eta_+:=g_{K(t_+,a_-)}(\Delta \eta_+\cap H(t_+,a_-))$. Since  $\Delta\eta_+$ has positive distance from $S\supset J$, and $g_{K(t_+,a_-)}^{-1}|_{\HH}$ extends continuously to $\lin\HH$,  $\Delta\ha\eta_+$ has positive distance from $J_e$. Thus, there is another crosscut $J_i$ of $\HH$, which is disjoint from and surrounded by $J_e$, such that the subdomain  $\HH$ bounded by $J_i$ and $J_e$ is disjoint from $\Delta\ha\eta_+$. Label the three connected components of $\HH\sem (J_e\cup J_i)$  by $D_e,A,D_i$, respectively, from outside to inside. Then $\Delta\ha\eta_+$ intersects both $D_e$ and $D_i$, but is disjoint from $\lin A$. Let $K_i=D_i\cup J_i$ and $K_e=K_i\cup A\cup J_e$ be two $\HH$-hulls.

    Let $\eta_+^* =\eta_+(t_++\cdot)$ and $\ha\eta_+^* =g_{K(t_+,a_-)}\circ \eta_+^* $, whose domain is $S:=\{s\in [0,T_+-t_+): \eta_+^*(s)\in\HH\sem K(t_+,a_-)\}$. For each $s\in[0,\delta:=t_+'-t_+]$, $K(t_++s,a_-)=\Hull(K(t_+,a_-)\cup\Delta\eta_+^s)$, and so $K_+'(s):=K_{+}^{a_-}(t_++s)/K_{+}^{a_-}(t_+)= K(t_++s,a_-)/K(t_+,a_-)$ (by (\ref{K123})) is the $\HH$-hull generated by $\ha\eta_+^*([0,s]\cap S)$. For $0\le s\le \delta$, since $A$ is disjoint from $\ha \eta_+^*([0,\delta]\cap S)\subset \Delta\ha\eta_+$, it is either contained in or disjoint from $K_+'(s)$. If $K_+'(s)\supset A$, then $K_+'(s)\supset \Hull(A)=K_i$; if $K_+'(s)\cap A=\emptyset$, then by the connectedness of $\lin{K_+'(s)}$, $K_+'(s)$ is contained in either $K_i$ or   $\HH\sem(K_i\cup A)$.
Since $K_+'(\delta)\supset  \Delta\ha \eta_+$ intersects both $D_e$ and $D_i$, we get $K_+'(\delta)\supset  K_e$. Let $s_0=\inf \{s\in [0,\delta]: K_e\subset K_+'(s)\}$. By Proposition \ref{Loewner-chain}, we have $s_0\in(0,\delta]$ and $K_+'(s_0)\supset K_e$. By the increasingness of $K_+'(\cdot)$, $\bigcup_{0\le s<s_0} K_+'(s)$ is contained in either $K_i$ or $\HH\sem (K_i\cup A)$. Let $S_0=\{s\in(s_0,T_+-t_+):\eta_+^*(s)\in\HH\sem K(t_++s_0,a_-)\}$. Then $\ha \eta_+^*(S_0)\subset g_{K(t_+,a_-)}(\HH\sem K(t_++s_0,a_-))= \HH\sem K_+'(s_0)\subset \HH\sem K_e=D_e$. By Lemma \ref{lem-strict}, $S$ is dense in $[s_0,T_+-t_+]$. Thus, $\ha \eta_+^*([s_0,\delta]\cap S)\subset \lin{D_e}$. Since $\Delta\ha\eta_+=\ha \eta_+^*([0,\delta]\cap S)$ intersects both $D_e$ and $D_i$, we conclude that $\ha \eta_+^*([0,s_0)\cap S)$ intersects $D_i$. So $K_+'(s)\subset K_i$ for $0\le s<s_0$, which implies that $K_+'(s_0)=\Hull(\bigcup_{0\le s<s_0} K_+'(s))\subset  K_i$. This contradicts that $K_i\subsetneqq K_e\subset K_+'(s_0)$. So the claim is proved.

Let $N$ denote the connected component of $H(t_+,a_-)\sem J$ that contains $\Delta \eta_+\cap H(t_+,a_-)$.
Then $N$ is contained in one connected component of $H(t_+,t_-)\sem J$. Since $N\supset  \Delta \eta_+\cap H(t_+,a_-)\ni z_*$ and $z_*$ lies in the connected component $D_J$ of $H(t_+,t_-)\sem J$, we get $\Delta \eta_+\cap H(t_+,a_-)\subset N\subset D_J$. Since $\Delta \eta_+\cap H(t_+,a_-)$ is dense in $\Delta \eta_+\cap H(t_+,t_-)$ (Lemma \ref{lem-strict}), and $\Delta\eta_+$ has positive distance from $J$, we get  $\Delta \eta_+\cap H(t_+,t_-)\subset D_J$. Since $K(t_+',t_-)$ is the $\HH$-hull generated by $K(t_+,t_-)$ and $\Delta\eta_+\cap H(t_+,t_-)$, we get  $K(t_+',t_-)\sem K(t_+,t_-)\subset D_J$, and so $K_+'(\delta)=K(t_+',t_-)/ K(t_+,t_-)$ is enclosed by the crosscut $g (J)$, where $g:= g_{K(t_+,t_-)}$.  Thus,  $\diam(K_+'(\del))\le \diam(g ( J))$.

Let  $R=2L$. From $\eta_+(t_+)\in  \lin{K(\ulin a)}$, we get $|\eta_+(t_+)|\le L$.  Recall that $J\subset S=\{|z-\eta_+(t_+)|=r\}$. So the arc $J$ and the circle $\{|z-\eta_+(t_+)|=R\}$ are separated by the annulus centered at $\eta_+(t_+)$ with inner radius $r$ and outer radius $R-L=L$. Let $J'=\{|z-\eta_+(t_+)|=R\}\cap \HH$ and  $D_{J'}=(\HH\cap \{|z-\eta_+(t_+)|<R\})\sem K(t_+,t_-)$.
By comparison principle (\cite{Ahl}), the extremal length of the curves in $D_{J'}$ that separate $J$ from $J'$ is bounded above by $2\pi/\log(L/r)$. By conformal invariance, the extremal length of the curves in the subdomain of $\HH$ bounded by the crosscut $g ( J')$, denoted by $D_{g(J')}$, that separate $g ( J)$ from $g ( J')$ is also bounded above by $2\pi/\log(L/r)$. By Proposition \ref{g-z-sup}, $g ( J')\subset \{|z|\le R+3L=5L\}$. So the Euclidean area of $D_{g(J')}$ is bounded above by $25\pi L^2/2$. By the definition of extremal length, there exists a curve in $\Omega$ with Euclidean length less than $10\pi L (\log(L/r))^{-1/2}$,
which separates  $g( J)$ from $g( J'\})$. This implies that the $\diam(g( J))$ is bounded above by $10\pi L (\log(L/r))^{-1/2}$, and so is that of $K_+'(\delta)=K_{+}^{t_-}(t_+')/K_{+}^{t_-}(t_+)$.  This finishes the proof of (i){ in the case $a_\pm\in \I_\pm^*$}.

{Now we do not assume $a_\pm\in \I_\pm^*$. Let $L'>L$. By the denseness of $\I_\pm^*$ in $\I_\pm$ and the continuity of $\eta_\pm$, we can find $a_\pm'>a_\pm$ such that $a_\pm'\in\I_\pm^*$, $(a_+',a_-')\in \cal D$, and $\rad_0(K(a_+',a_-'))<L'$. Since $ t_\sigma<t_\sigma'\in [0,a_\sigma']$,  $\diam(\eta_\sigma[t_\sigma,t_\sigma'])<r<L$, and $t_{-\sigma}\in[0,a_{-\sigma}']$, by the above result, we get $\diam(K_{\sigma}^{t_{-\sigma}}({t_\sigma'})/K_{\sigma}^{t_{-\sigma}}(t_\sigma))\le 10\pi L' \log(L'/r)^{-1/2}$. Since the inequality holds for any $L'>L$, it also holds with $L$ in place of $L'$.
}

(ii) This follows from (i), Proposition \ref{g-z-sup}, the continuity of $\eta_\sigma$, the {fact that} $ L\log(L/r)^{-1/2}$ {tends to $0$ as $r\downarrow 0$}, and the equality $g_{K_{{\sigma}}^{t_{{-\sigma}}}(t_{{\sigma}}')} =g_{K_{{\sigma}}^{t_{{-\sigma}}}(t_{{\sigma}}')/K_{{\sigma}}^{t_{{-\sigma}}}(t_{{\sigma}})}\circ g_{K_{{\sigma}}^{t_{{-\sigma}}}(t_{{\sigma}})}$.

(iii) This follows from (ii), (\ref{circ-g}) and the fact that $g_{K(\ulin t)}$ is analytic on $\C\sem K(\ulin t)^{\doub}$.
\end{proof}

For a function $X$ defined on  $\cal D$, $\sigma \in\{+,-\}$ and $t_{-\sigma} \in\R_+$, we use $X|^{-\sigma}_{t_{-\sigma}} $  to denote the single-variable function $t_{\sigma}\mapsto X(t_\sigma \ulin e_\sigma +t_{-\sigma}\ulin e_{-\sigma})$, $0\le t_\sigma<T^{\cal D}_\sigma(t_{-\sigma})$, and use $\pa_{\sigma} X(t_+,t_-)$  to denote the derivative of this function at $t_{\sigma}$.

\begin{Lemma}
   There are two functions $W_+,W_-\in C({\cal D},\R)$ such that for any $\sigma\in\{+,-\}$ and  $t_{-\sigma}\in\I_{-\sigma}$,
$K_{\sigma}^{t_{-\sigma}}(t_\sigma)$, $0\le t_\sigma<T^{\cal D}_\sigma(t_{-\sigma})$, are chordal Loewner hulls  driven by $W_\sigma|^{-\sigma}_{t_{-\sigma}}$ with speed $\mA|^{-\sigma}_{t_{-\sigma}}$, and for any $\ulin t=(t_+,t_-)\in\cal D$, $\eta_\sigma(t_\sigma)=f_{K(\ulin t)}( W_\sigma(\ulin t))$.
  \label{Lem-W}
\end{Lemma}
\begin{proof}
	By symmetry, we only need to prove the case that $\sigma=+$. Since
 $$\hcap_2(K_{+}^{t_-}(t_++\del))-\hcap_2(K_{+}^{t_-}(t_+))=\mA(t_++\del,t_-)-\mA(t_+,t_-),$$
by Lemma \ref{lem-uniform} (i), the continuity of $\eta_\sigma$ and Proposition \ref{Loewner-chain}, for every $t_-\in\I_-$, $K_{+}^{t_-}(t_+)$, $0\le t_+<T^{\cal D}_+(t_-)$, are chordal Loewner hulls with  speed $\mA|^-_{t_-}$, and the driving function, denoted by  $W_+(\cdot,t_-)$, satisfies {the property} that
   \BGE \{W_+(t_+,t_-)\}= \bigcap_{\del>0} \lin{K_{+}^{t_-}(t_++\del)/K_{+}^{t_-}(t_+)}=\bigcap_{\del>0} \lin{K(t_++\delta,t_-)/K(t_+,t_-)}.\label{W-def}\EDE

Fix $\ulin t=(t_+,t_-)\in\cal D$.  We now show that $f_{K(\ulin t)}( W_+(\ulin t)) =\eta_+(t_+)$. By Lemma \ref{lem-strict}, there exists a  sequence  $t_+^n\downarrow t_+$ such that   $\eta_+(t_+^n)\in K(t_+^n,t_-)\sem K(t_+,t_-)$ for all $n$. Then $g_{K(\ulin t)}(\eta_+(t_+^n))\in K(t_+^n,t_-)/K(\ulin t)=K_{+}^{t_-}(t_+^n)/K_{+}^{t_-}(t_+)$. So we have $g_{K(\ulin t)}( \eta_+(t_+^n))\to W_+(\ulin t)$ by (\ref{W-def}). From the continuity of $f_{K(\ulin t)} $  and  $\eta_+$, we then get
$$\eta_+(t_+)=\lim_{n\to \infty} \eta_+(t_+^n)=\lim_{n\to \infty}  f_{K(\ulin t)}( g_{K(\ulin t)}( \eta_+(t_+^n)))=f_{K(\ulin t)}( W_+(\ulin t)).$$

It remains to show that $W_+$ is continuous on $\cal D$.
Let $t_+,t_-^1,t_-^2\in\R_+$ be such that $t_-^1<t_-^2$ and $(t_+,t_-^2)\in \cal D$.
By Lemma \ref{lem-strict}, there is a sequence $\delta_n\downarrow 0$ such that $z_n:=\eta_+(t_++\delta_n)\in H(t_+,t_-^2)$. Then  $g_{K(t_+,t_-^j)}(z_n)\in K(t_++\delta_n,t_-^j)/ K(t_+,t_-^j)=K_{+}^{t_-^j}(t_++\delta_n)/K_{+}^{t_-^j}(t_+)$, $j=1,2$. From (\ref{W-def}) we get
$$|W_+(t_+,t_-^j)-g_{K(t_+,t_-^j)}(z_n)|\le \diam(K_{+}^{t_-^j}(t_++\delta_n)/K_{+}^{t_-^j}(t_+)),\quad j=1,2.$$
Since $g_{K(t_+,t_-^2)}(z_n)=g_{K(t_+,t_-^2)/K(t_+,t_-^1)}\circ g_{K(t_+,t_-^1)}(z_n)$, by Proposition \ref{g-z-sup} we get
$$|g_{K(t_+,t_-^2)}(z_n)-g_{K(t_+,t_-^1)}(z_n)|\le 3\diam(K(t_+,t_-^2)/K(t_+,t_-^1))= 3\diam(K_{-}^{t_+}(t_-^2)/K_{-}^{t_+}(t_-^1)) .$$
Combining the above  displayed formulas and letting $n\to\infty$, we get
\BGE |W_+(t_+,t_-^2)-W_+(t_+,t_-^1)|\le 3\diam(K_{-}^{t_+}(t_-^2)/K_{-}^{t_+}(t_-^1)) ,\label{W+continuity}\EDE
which together with Lemma \ref{lem-uniform} (i) implies that, for any $(a_+,a_-)\in\cal D$, the family of functions $[0,a_-]\ni t_-\mapsto W_+(t_+,t_-)$, $0\le t_+\le a_+$, are equicontinuous. Since $W_+$ is continuous in $t_+$ as a driving function, we conclude that $W_+$ is continuous on $\cal D$.
\end{proof}

\begin{Definition}
	We call $W_+$ and $W_-$ the driving functions for the commuting pair $(\eta_+,\eta_-;{\cal D})$.
It is obvious that $W_\sigma|^{-\sigma}_0=\ha w_\sigma$, $\sigma\in\{+,-\}$.
\end{Definition}

\begin{Remark} By (\ref{W-def}) and Propositions \ref{prop-connected} and \ref{winK}, for $t_+^1<t_+^2\in \I_+$ and $t_-\in \I_-$ such that $(t_+^2,t_-)\in \cal D$,
$|W_+(t_+^2,t_-)-W_+(t_+^1,t_-)|\le 4\diam(K_{+}^{t_-}(t_+^2)/K_{+}^{t_-}(t_+^1)) $.
This combined with (\ref{W+continuity}) and Lemma \ref{lem-uniform} (i) implies that, if $\eta_\sigma$ extends continuously to $[0,T_\sigma]$ for $\sigma\in\{+,-\}$, then $W_+$ and $W_-$ are uniformly continuous on $\cal D$, and so extend continuously to $\lin{\cal D}$.
  \label{Remark-continuity-W}
\end{Remark}

\begin{Lemma}
For any $\sigma\in\{+,-\}$ and  $t_{-\sigma}\in\I_{-\sigma}$, the chordal Loewner hulls $K_{\sigma}^{t_{-\sigma}}(t_\sigma)=K(t_+,t_-)/K_{-\sigma}(t_{-\sigma})$, $0\le t_\sigma<T^{\cal D}_\sigma(t_{-\sigma})$,  are generated by a chordal Loewner curve, denoted by $\eta_{\sigma}^{t_{-\sigma}}$, which intersects $\R$ at a set with Lebesgue measure zero such that $\eta_\sigma|_{[0,T^{\cal D}_\sigma(t_{-\sigma}))}=f_{K_{-\sigma}(t_{-\sigma})}\circ \eta_{\sigma}^{t_{-\sigma}}$. Moreover, for $\sigma\in\{+,-\}$, $(t_+,t_-)\mapsto \eta_{\sigma}^{t_{-\sigma}}(t_\sigma)$ is continuous on $\cal D$.
  \label{Lebesgue}
\end{Lemma}
\begin{proof}
  It suffices to work {in} the case that $\sigma=+$. First  we   show that there exists a continuous function $(t_+,t_-)\mapsto \eta_{+}^{t_-}(t_+)$ from $\cal D$ into $\lin\HH$ such that
\BGE \eta_+(t_+)=f_{K_-(t_-)}(\eta_{+}^{t_-}(t_+)),\quad \forall (t_+,t_-)\in\cal D.\label{eta-ft}\EDE
Let $(t_+,t_-)\in\cal D$. By Lemma \ref{lem-strict}, there is a sequence $t_+^n\downarrow t_+$ such that for all $n$, $(t_+^n,t_-)\in\cal D$ and $\eta_+(t_+^n)\in\HH\sem K(t_+,t_-)$. Then we get $g_{K_-(t_-)}(\eta_+(t_+^n))\in g_{K_-(t_-)}(K(t_+^n,t_-)\sem K(t_+,t_-))=K_{+}^{t_-}(t_+^n)/ K_{+}^{t_-}(t_+)$. If $t_-\in\I_-^*$, then
by Condition (I),  $\bigcap_n \lin{K_{+}^{t_-}(t_+^n)/K_{+}^{t_-}(t_+)}=\{\eta_{+}^{t_-}(t_+)\}$, which implies that $g_{K_-(t_-)}(\eta_+(t_+^n))\to \eta_{+}^{t_-}(t_+)$. From the continuity of $f_{K_-(t_-)}$ and $\eta_+$, we find that (\ref{eta-ft}) holds if $t_-\in\I_-^*$. Thus,
\BGE \eta_{+}^{t_-}(t_+)=g_{K_-(t_-)}(\eta_+(t_+)),\quad \mbox{if } (t_+,t_-)\in{\cal D}, \,t_-\in\I_-^*\mbox{ and }\eta_+(t_+)\in\HH\sem K_-(t_-).\label{eta-gt}\EDE
Fix $a_-\in\I_-^*$. Let ${\cal R}=\{t_+\in\I_+:(t_+,a_-)\in{\cal D},\eta_+(t_+)\in \HH\sem K_-(a_-)\}$, which by Lemma \ref{lem-strict} is dense in $[0,T^{\cal D}_+(a_-))$.
By Propositions \ref{g-z-sup}  and \ref{Loewner-chain},
\BGE \lim_{\delta\downarrow 0}\,\sup_{ t_-\in [0,a_-]}\,\,\sup_{t_-'\in[0,a_-]\cap (t_--\delta,t_-+\delta)} \,\, \sup_{t_+\in{\cal R}}\, |g_{K_-(t_-)}( \eta_+(t_+))-g_{K_-(t_-')}(\eta_+(t_+))|=0.\label{unif-g-g}\EDE
This combined with (\ref{eta-gt}) implies that
\BGE \lim_{\delta\downarrow 0}\,\sup_{ {t_-\in [0,a_-]\cap \I_-^*}}\,\,\sup_{t_-'\in[0,a_-]\cap\I_-^*\cap (t_--\delta,t_-+\delta)} \,\, \sup_{t_+\in{\cal R}} \, |\eta_{+}^{t_-}(t_+)-\eta_{+}^{t_-'}(t_+)|=0.\label{unif-g-g'}\EDE
By the denseness of $\cal R$ in $[0,T^{\cal D}_+(a_-))$ and the continuity of each $\eta_{+}^{t_-}$, $t_-\in\I_-^*$, we know that (\ref{unif-g-g'}) still holds if $\sup_{t_+\in{\cal R}}$ is replaced by $\sup_{t_+\in[0,T^{\cal D}_+(a_-))}$. Since $\I_-^*$ is dense in $\I_-$, the continuity of each $\eta_{+}^{t_-}$, $t_-\in\I_-^*$, together with (\ref{unif-g-g'}) implies that   there exists a continuous function $[0,T^{\cal D}_+(a_-))\times [0,a_-]\ni (t_+,t_-)\mapsto \eta_{+}^{t_-}(t_+)\in\lin\HH$, which extends those $\eta_{+}^{t_-}(t_+)$ for $t_-\in\I_-^*\cap [0,a_-]$ and $t_+\in [0,T^{\cal D}_+(a_-))$. Running $a_-$ from $0$ to $T_-$, we get a continuous function ${\cal D}\ni (t_+,t_-)\mapsto \eta_{+}^{t_-}(t_+)\in\lin\HH$, which extends those $\eta_{+}^{t_-}(t_+)$ for $(t_+,t_-)\in\cal D$ and $t_-\in\I_-^*$.
Since $\eta_{+}^{t_-}(t_+)=g_{K_-(t_-)}(\eta_+(t_+))$ for all $t_+\in{\cal R}$ and $t_-\in[0,a_-]\cap \I_-^*$, from (\ref{eta-gt},\ref{unif-g-g}) we know that it is also true for any $t_-\in[0,a_-]$. Thus, $\eta_+(t_+)=f_{K_-(t_-)}(\eta_{+}^{t_-}(t_+))$ for all $t_+\in{\cal R}$ and $t_-\in[0,a]$. By the denseness of $\cal R$ in $[0,T^{\cal D}_+(a_-))$ and the continuity of $\eta_+$, $f_{K_-(t_-)} $ and $\eta_{+}^{t_-}$, we get (\ref{eta-ft}) for all $t_-\in[0,a_-]$ and $t_+\in [0,T^{\cal D}_+(a_-))$. So (\ref{eta-ft}) holds for all $(t_+,t_-)\in\cal D$.

For  $(t_+,t_-)\in\cal D$, since $K(t_+,t_-)=\Hull(K_-(t_-)\cup(\eta_+[0,t_+]\cap (\HH\sem K_-(t_-)))$, we see that $K_{+}^{t_-}(t_+)=g_{K_-(t_-)}(K(t_+,t_-)\sem K_-(t_-))$ is the $\HH$-hull generated by $g_{K_-(t_-)}(\eta_+[0,t_+]\cap (\HH\sem K_-(t_-)))=\eta_{+}^{t_-}[0,t_+]\cap \HH$. So $K_{+}^{t_-}(t_+)=\Hull(\eta_{+}^{t_-}[0,t_+])$. By Lemma \ref{Lem-W}, for any $t_-\in [0,T_-)$, $\eta_{+}^{t_-}(t_+)$, $0\le t_+<T^{\cal D}_+(t_-)$, is the chordal Loewner curve driven by $W_+(\cdot,t_-)$ with speed $\mA(\cdot,t_-)$. So we have $\eta_{+}^{t_-}(t_+)=f_{K_{+}^{t_-}(t_+)}(W_+(t_+,t_-))$, which together with $\eta_+(t_+)=f_{K(t_+,t_-)}( W_+(t_+,t_-))$ implies that $\eta_+(t_+)=f_{K_-(t_-)}(\eta_{+}^{t_-}(t_+))$.

Finally, we show that $\eta_{+}^{t_-}\cap\R$ has Lebesgue measure zero for all $t_-\in\I_-$. Fix $t_-\in \I_-$ and $\ha t_+\in\I_+$ such that $(\ha t_+,t_-)\in\cal D$. It suffices to show that $\eta_{+}^{t_-}[0,\ha t_+]\cap \R$ has Lebesgue measure zero. There exists a sequence $\I_-^*\ni t_-^n\downarrow t_-$ such that $(\ha t_+,t_-^n)\in\cal D$ for all $n$. Let $K_n=K_-(t^n_-)/K_-(t_-)$, $g_n=g_{K_n}$, and $f_n=g_n^{-1}$.
Then $f_{K_-(t_-)}=f_{K_-(t_-^n )}\circ g_n$ on $\HH\sem K_n$. Let $t_+\in[0,\ha t_+]$. From  $f_{K_-(t_-)}( \eta_{+}^{t_-}(t_+))=\eta_+(t_+)=f_{K_-(t_-^n )}( \eta_{+}^{t_-^n}(t_+))$ we get
$ \eta_{+}^{t_-^n}(t_+)=g_n(\eta_{+}^{t_-}(t_+))$ if $\eta_{+}^{t_-}(t_+)\in \HH\sem K_n$. By continuity we get $ \eta_{+}^{t_-}(t_+)=f_n(\eta_{+}^{t_-}(t_+^n))$  if $\eta_{+}^{t_-^n}(t_+)\in \R \sem [c_{K_n},d_{K_n}]$, $0\le t_+\le \ha t_+$.  Thus, $\eta_{+}^{t_-}[0,\ha t_+]\cap (\R\sem [a_{K_n},b_{K_n}])\subset f_n(\eta_{+}^{t_-^n}[0,\ha t_+]\cap (\R\sem [c_{K_n},d_{K_n}]))$. Since $t_-^n\in\I_-^*$, by Condition (II) and the analyticity of $f_n$ on $\R\sem [c_{K_n},d_{K_n}]$ we know that $\eta_{+}^{t_-}[0,\ha t_+]\cap (\R\sem [a_{K_n},b_{K_n}])$  has Lebesgue measure zero for each $n$. Sending $n\to \infty$ and using the fact that $[a_{K_n},b_{K_n}]\downarrow \{\ha w_-(t_-)\}$, we see that $\eta_{+}^{t_-}[0,\ha t_+]\cap\R$ also has Lebesgue measure zero.
\end{proof}

\begin{Lemma}
For any $\sigma\in\{+,-\}$ and $(t_+,t_-)\in\cal D$,  $\ha w_\sigma(t_\sigma)= f_{K_{-\sigma}^{t_\sigma}(t_{-\sigma})}(W_\sigma(t_+,t_-))\in \pa (\HH\sem K_{-\sigma}^{t_\sigma}(t_{-\sigma}))$. \label{lem-w-W}
\end{Lemma}
\begin{proof}
By symmetry, it suffices to work on the case $\sigma=+$. For any $(t_+,t_-)\in\cal D$, by Lemma \ref{lem-strict} there is a sequence $t^n_+\downarrow t_+$ such that $\eta_+(t^n_+)\in K(t^n_+,t_-)\sem K(t_+,t_-)$ for all $n$.
From (\ref{haw=}) and  Lemma \ref{Lem-W} we get $g_{K_+(t_+)}(\eta_+(t^n_+))\to \ha w_+(t_+)$ and $g_{K(t_+,t_-)}(\eta_+(t^n_+))\to W_+(t_+,t_-)$. From (\ref{circ-g}) we get $g_{K_+(t_+)} =f_{K_{-}^{t_+}(t_-)}\circ g_{K(t_+,t_-)}$. From the continuity of $f_{K_{-}^{t_+}(t_-)}$ on $\lin\HH$, we then get  $\ha w_+(t_+)= f_{K_{-}^{t_+}(t_-)}( W_+(t_+,t_-))$. Finally, $\ha w_+(t_+)\in \pa(\HH\sem K_{-}^{t_+}(t_-))$ because $W_+(t_+,t_-)\in\pa\HH$ and   $f_{K_{-}^{t_+}(t_-)}$ maps $\HH$ conformally onto $\HH\sem K_-^{t_+}(t_-)$.
\end{proof}

\subsection{Force point functions}\label{section-V}
For $\sigma\in\{+,-\}$, define $C_\sigma$ and $D_\sigma$ on $\cal D$ such that if $t_\sigma>0$, $C_\sigma(t_+,t_-)=c_{K_{\sigma}^{t_{-\sigma}}(t_\sigma)}$ and $D_\sigma(t_+,t_-)=d_{K_{\sigma}^{t_{-\sigma}}(t_\sigma)}$; and if $t_\sigma=0$, then $C_\sigma =D_\sigma =W_\sigma$ at $t_{-\sigma}\ulin e_{-\sigma}$. Since $K_{\sigma}^{t_{-\sigma}}(\cdot)$ are chordal Loewner hulls driven by $W_\sigma|^{-\sigma}_{t_{-\sigma}}$ with some speed, by Proposition \ref{winK} we get
\BGE C_\sigma\le W_\sigma\le D_\sigma\quad \mbox{on }\cal D,\quad \sigma\in\{+,-\}.\label{CWD}\EDE
Since $K_{\sigma}^{t_{-\sigma}}(t_\sigma)$ is the $\HH$-hull generated by $\eta_{\sigma}^{t_{-\sigma}}[0,t_\sigma]$, we get
\BGE f_{K_{\sigma}^{t_{-\sigma}}(t_\sigma)}[C_\sigma(t_+,t_-),D_\sigma(t_+,t_-)]\subset \eta_{\sigma}^{t_{-\sigma}}[0,t_\sigma].\label{f[C,D]}\EDE

Recall that $w_-<w_+\in\R$. We write $\ulin w$ for $(w_+,w_-)$. Define $R_{\ulin w}=(\R\sem \{w_+,w_-\})\cup\{w_+^+,w_+^-,w_-^+,w_-^-\}$ with the obvious order endowed from $\R$. Assign the topology to $\R_{\ulin w}$ such that $I_-:=(-\infty,w_-^-],I_0:=[w_-^+,w_+^-],I_+:=[w_+^+,\infty)$ are three connected components of $\R_{\ulin w}$, which are respectively homeomorphic to $(-\infty,w_-],[w_-,w_+],[w_+,\infty)$. Recall that for $\sigma\in\{+,-\}$ and $t\in\I_\sigma$, $g_{K_\sigma(t)}^{w_\sigma}$ (Definition \ref{Def-Rw}) is defined on $\R_{w_\sigma}$, and agrees with $g_{K_\sigma(t)}$ on $\R\sem ([a_{K_\sigma(t)},b_{K_\sigma(t)}]\cup\{w_\sigma\})$. By Lemma \ref{lem-w-W} and the fact that $w_{-\sigma}\not\in [a_{K_\sigma(t)},b_{K_\sigma(t)}]\cup\{w_\sigma\}$, we then know that $g_{K_\sigma(t)}^{w_\sigma}(w_{-\sigma})=W_{-\sigma}(t\ulin e_\sigma)$. So we define $g_{K_\sigma(t)}^{w_\sigma}(w_{-\sigma}^\pm)=W_{-\sigma}(t\ulin e_\sigma)^\pm$, and understand $g_{K_\sigma(t)}^{w_\sigma}$ as a continuous function from $\R_{\ulin w}$ to $\R_{W_{-\sigma}(t\ulin e_\sigma)}$.

\begin{Lemma}
For any $\ulin t=(t_+,t_-)\in\cal D$, $g_{K_{+}^{t_-}(t_+)}^{W_+(0,t_-)}\circ g_{K_-(t_-)}^{w_-}$ and $g_{K_{-}^{t_+}(t_-)}^{W_-(t_+,0)}\circ g_{K_+(t_+)}^{w_+} $ agree on $\R_{\ulin w}$, and the common function in the equality, denoted by $g_{K(\ulin t)}^{\ulin w}$, satisfies the following properties.
  \begin{enumerate}
    \item [(i)]  $g_{K(\ulin t)}^{\ulin w}$ is increasing and continuous on $\R_{\ulin w}$, and agrees with $g_{K(\ulin t)}$ on $\R\sem \lin{K(\ulin t)}$.
    \item [(ii)] $g_{K(\ulin t)}^{\ulin w}$  maps $I_+\cap (\lin{K(\ulin t)}\cup\{w_+^+\})$ and $I_-\cap (\lin{K(\ulin t)}\cup \{w_-^-\})$ to $ D_+(\ulin t) $ and $ C_-(\ulin t) $, respectively.
    \item [(iii)] If $\lin{K_+(t_+)}\cap \lin{K_-(t_-)}=\emptyset$, $g_{K(\ulin t)}^{\ulin w}$  maps   $I_0\cap (\lin{K_+(t_+)}\cup\{w_+^-\})$ and $I_0\cap (\lin{K_-(t_-)}\cup \{w_-^+\})$ to   $ C_+(\ulin t) $ and $ D_-(\ulin t) $, respectively.
    \item [(iv)] If $\lin{K_+(t_+)}\cap \lin{K_-(t_-)}\ne\emptyset$, $g_{K(\ulin t)}^{\ulin w}$  maps $I_0$ to   $ C_+(\ulin t) =D_-(\ulin t) $.
    \item [(v)] The map $(\ulin t,v)\mapsto g_{K(\ulin t)}^{\ulin w}(v)$ from ${\cal D}\times \R_{\ulin w}$ to $\R$ is jointly continuous.
  \end{enumerate}
  \label{common-function}
\end{Lemma}
\begin{proof}
 Fix $\ulin t=(t_+,t_-)\in\cal D$. For $\sigma\in\{+,-\}$, we write $K$ for $K(\ulin t)$, $K_\sigma$ for $K_\sigma(t_\sigma)$, $\til K_\sigma$ for $K_{\sigma}^{t_{-\sigma}}(t_\sigma)$, $\til w_\sigma$ for $W_\sigma(t_{-\sigma}\ulin e_{-\sigma})$,
 $C_\sigma$ for $C_\sigma(\ulin t)$, and $D_\sigma$ for $D_\sigma(\ulin t)$.
    The equality now reads $g_{\til K_+}^{\til w_+}\circ g_{K_-}^{w_-}=g_{\til K_-}^{\til w_-}\circ g_{K_+}^{w_+}$. Before proving the equality, we first show that both sides are well defined and satisfy (i-iii) and a weaker version of (iv) (see below). First consider $g_{\til K_-}^{\til w_-}\circ g_{K_+}^{w_+}$. Since $g_{K_+}^{w_+}:\R_{\ulin w}\to \R_{\til w_-}$, the composition is well defined on $\R_{\ulin w}$. We denote it by $g_{K(\ulin t)}^{\ulin w,+}$.

(i)
The continuity and monotonicity of  the composition follows from the continuity and monotonicity of both $g_{\til K_-}^{\til w_-}$ and $g_{K_+}^{w_+}$.
Let $v \in \R \sem \lin{K}$. Then $v\not\in \lin{K_+}$, and $g_{K_+}^{w_+}(v)=g_{K_+}(v)$.
Since $\til K_-=K/K_+$, $K\sem K_+=f_{K_+}(\til K_-)$. From $v=f_{K_+}(g_{K_+}(v))\not\in \lin{K\sem K_+}$ and the continuity of $f_{K_+}$ on $\lin\HH$, we know that $g_{K_+}(v)\not\in \lin{\til K_-}$, which implies that $g_{K(\ulin t)}^{\ulin w,+}(v)=g_{\til K_-} \circ g_{K_+} (v)=g_K(v)$.

In the proof of (ii,iii) below, we write $\eta_\sigma$ for $\eta_\sigma[0,t_\sigma]$ and $\til\eta_\sigma$ for $\eta_{\sigma}^{t_{-\sigma}}[0,t_\sigma]$;
when $t_\sigma=0$, i.e., $K_\sigma=\til K_\sigma=\emptyset$, we understand $a_{K_\sigma}=b_{K_\sigma}=c_{K_\sigma}=d_{K_\sigma}=w_\sigma$, and $a_{\til K_\sigma}=b_{\til K_\sigma}=c_{\til K_\sigma}=d_{\til K_\sigma}=\til w_\sigma$. Then it is always true that $a_{K_\sigma}=\min\{\eta_\sigma \cap \R\}$, $b_{K_\sigma}=\max\{\eta_\sigma \cap \R\}$,
$a_{\til K_\sigma}=\min\{\til \eta_\sigma \cap \R\}$, $b_{\til K_\sigma}=\max\{\til \eta_\sigma \cap \R\}$, $c_{\til K_\sigma}=C_\sigma$, and $d_{\til K_\sigma}=D_\sigma$. Since $\eta_\pm=f_{K_{\mp}}(\til \eta_\pm)$, we get $ b_{\til K_+}=g_{K_-}(b_{K_+})$, $a_{\til K_-}=g_{K_+}(a_{K_-})$. If $\lin{K_+}\cap \lin{K_-}=\emptyset$, then $a_{\til K_+}=g_{K_-}(a_{K_+})$, $b_{\til K_-}=g_{K_+}(b_{K_-})$.

(ii) Since  $I_+\cap (\lin K\cup\{w_+^+\})=\{w_+^+\}\cup (w_+,b_K]=\{w_+^+\}\cup (w_+^+,b_{K_+}]$ is mapped by $g_{K_+}^{w_+}$ to a single point, it is also mapped by $g_{K(\ulin t)}^{\ulin w,+} $ to a single point, which by (i) is equal to
$$\lim_{x\downarrow b_{K}} g_K(x)=\lim_{x\downarrow b_{K_+}} g_{\til K_+}\circ g_{K_-}(x)=\lim_{y\downarrow b_{\til K_+}} g_{\til K_+}(y)=d_{\til K_+}=D_+.$$

To show that $I_-\cap( \lin K\cup\{w_-^-\})=[a_K,w_-^-)\cup\{w_-^-\}$ is mapped by $g_{K(\ulin t)}^{\ulin w,+} $ to $C_-$, by (i)  it suffices to show that $\lim_{x\uparrow a_{K}}g_K(x)=g_{\til K_-}^{\til w_-}\circ g_{K_+}^{w_+}(w_-^-)=c_{\til K_-}$. This holds because $$g_{\til K_-}^{\til w_-}\circ g_{K_+}^{w_+}(w_-^-)= g_{\til K_-}^{\til w_-}(\til w_-^-)=c_{\til K_-}=\lim_{x\uparrow a_{\til K_-}} g_{\til K_-}(x)=\lim_{x\uparrow a_{ K_-}} g_{\til K_-}\circ g_{K_+}(x)=\lim_{x\uparrow a_{K}}g_K(x).$$

(iii) Suppose $\lin{K_+}\cap \lin{K_-}=\emptyset$. Then $I_0\cap (\lin{K_+}\cup\{w_+^-\})=[a_{K_+},w_+^-)\cup\{w_+^-\}$  is mapped by $g_{K_+}^{w_+}$ to a single point, so is also mapped by $g_{K(\ulin t)}^{\ulin w,+}$ to a single point. By (i) the latter point is
$$\lim_{x\uparrow a_{K_+}}g_K(x)=\lim_{x\uparrow a_{K_+}} g_{\til K_+}\circ g_{K_-}(x)=\lim_{y\uparrow a_{\til K_+}} g_{\til K_+}(y)=c_{\til K_+}=C_+.$$

Since $I_0\cap (\lin{K_-}\cup\{w_-^+\})=\{w_-^+\}\cup (w_-^+,b_{K_-}]$  is mapped by $g_{K_+}^{w_+}$ to $\{\til w_-^+\}\cup (\til w_-^+,b_{\til K_-}]$, which is further mapped by $g_{\til K_-}^{\til w_-}$ to   $d_{\til K_-}= D_-$, we see that
 $g_{K(\ulin t)}^{\ulin w,+}$ maps  $I_0\cap (\lin{K_-}\cup\{w_-^+\})$ to  $ D_-$.

(iv) Suppose $\lin{K_+}\cap \lin{K_-}\ne\emptyset$. For now, we only show that $I_0$ is mapped by $g_{K(\ulin t)}^{\ulin w,+}$ to $D_-$.  By the assumption we have $t_+,t_->0$ and  $[c_{K_+},d_{K_+}]\cap \lin{\til K_-}\ne \emptyset$, which implies that $c_{K_+}\le b_{\til K_-}$. Thus, $g_{K_+}^{w_+}(I_0)= [\til w_-^+, c_{K_+}]\subset [\til w_-^+, b_{\til K_-}]$, from which follows that  $g_{K(\ulin t)}^{\ulin w,+}(I_0)=\{d_{\til K_-}\}=\{D_-\}$.

Now $g_{\til K_-}^{\til w_-}\circ g_{K_+}^{w_+}$  satisfies (i-iii) and a weaker version of (iv). By symmetry, this is also true for $g_{\til K_+}^{\til w_+}\circ g_{K_-}^{w_-}$, where for (iv), $I_0$ is mapped to $\{C_+\}$. We now show that the two functions agree on $\R_{\ulin w}$.
By (i), $g_{\til K_+}^{\til w_+}\circ g_{K_-}^{w_-}$ and $g_{\til K_-}^{\til w_-}\circ g_{K_+}^{w_+}$ agree on $\R \sem \lin{K}$. By (ii),  the two functions also agree on $I_+\cap (\lin{K(\ulin t)}\cup\{w_+^+\})$ and $I_-\cap (\lin{K(\ulin t)}\cup \{w_-^-\})$. Thus they agree on both $I_+$ and $I_-$. By (i,iii)  they agree on $I_0$ when $\lin{K_+}\cap \lin{K_-}=\emptyset$. To prove that they agree on $I_0$ when $\lin{K_+}\cap \lin{K_-}\ne\emptyset$, by the weaker versions of (iv) we only need to show that $c_{\til K_+}=d_{\til K_-}$ in that case.

First, we show that $d_{\til K_-}\le c_{\til K_+}  $. Suppose $d_{\til K_-}> c_{\til K_+}$. Then $J:=(c_{\til K_+},d_{\til K_-})\subset [c_{\til K_-},d_{\til K_-}]\cap [c_{\til K_+},d_{\til K_+}]$. So $ f_{\til K_+}(J)\subset \pa (\HH\sem \til K_+)$. If $f_{\til K_+}(J)\subset \R$, then it is disjoint from $\lin{\til K_+}$, and so is disjoint from $[a_{\til K_+},b_{\til K_+}]$ since $\til K_+$ is generated by $\til\eta_+ $, which does not spend any nonempty interval of time on $\R$. That $f_{\til K_+}(J)\cap [a_{\til K_+},b_{\til K_+}]=\emptyset$ then implies that $J\cap [c_{\til K_+},d_{\til K_+}]=\emptyset$, a contradiction. So there is $x_0\in J$ such that $f_{\til K_+}(x_0)\subset \HH$, which implies that $f_K(x_0)=f_{K_-}\circ f_{\til K_+}(x_0)\in \HH\sem K_-$. On the other hand, since $x_0\in [c_{\til K_-},d_{\til K_-}]$,
 $f_K(x_0)=f_{K_+}\circ f_{\til K_-}(x_0) \subset   f_{K_+}(\til\eta_- )=\eta_- $, which contradicts that $f_K(x_0)\in \HH\sem K_-$. So $d_{\til K_-}\le c_{\til K_+}  $.

Second, we show that $d_{\til K_-}\ge  c_{\til K_+}  $. Suppose $d_{\til K_-}< c_{\til K_+}$. Let $J=(d_{\til K_-}, c_{\til K_+})$. Then $f_{\til K_+}(J)=(f_{\til K_+}(d_{\til K_-}),a_{\til K_+})$. From $\lin{K_+}\cap \lin{K_-}\ne\emptyset$ we know $a_{\til K_+}\le d_{K_-}$. From $a_{\til K_-}=g_{K_+}(a_{K_-})$
 we get $d_{\til K_-}\ge c_{\til K_-}=\lim_{x\uparrow a_{\til K_-}}g_{\til K_-}(x)
 =\lim_{y\uparrow a_{  K_-}}g_{\til K_-}\circ g_{K_+}(y)=\lim_{y\uparrow a_{  K_-}} g_{\til K_+}\circ g_{K_-}(y)$.
 Thus, $f_{\til K_+}(d_{\til K_-})\ge \lim_{y\uparrow a_{  K_-}} g_{K_-}(y)=c_{K_-}$. So we get $f_{\til K_+}(J)\subset  [c_{K_-},d_{K_-}]$, which is mapped into $\eta_-$ by $f_{K_-}$. Thus, $f_K(J)\subset \eta_-$. Symmetrically, $f_K(J)\subset \eta_+$. Since $\eta_-=f_{K_+}(\til \eta_-)$ and $f_K(J)\subset \pa(\HH\sem K)$, for every $x \in J$, there is $z_-\in\til\eta_-\cap \pa(\HH\sem \til K_-)$ such that $f_K(x )=f_{K_+}(z_-)$.
Then there is $y_-\in [c_{\til K_-},d_{\til K_-}]$ such that $z_-=f_{\til K_-}(y_-)$. So $f_K(x )=f_K(y_-)$. Similarly, for every $x \in J$, there is $y_+\in [c_{\til K_-},d_{\til K_-}]$ such that $f_K(x )=f_K(y_+)$. Pick $x^1<x^2\in J$ such that $f_K(x^1)\ne f_K(x^2)$. This is possible because $f_K(J)$ has positive harmonic measure in $\HH\sem K$.  Then there exist $y^1_+\in [c_{\til K_+},d_{\til K_+}]$ and $y^2_-\in [c_{\til K_-},d_{\til K_-}]$ such that $f_K(x^1)=f_K(y_1^+)$ and $f_K(x^2)=f_K(y^2_-)$. This contradicts that $y^1_+>x^2>x^1>y^2_-$.
 So $d_{\til K_-}\ge  c_{\til K_+}  $.

 Combining the last two paragraphs, we get $c_{\til K_+}=d_{\til K_-}$. So $g_{K_{+}^{t_-}(t_+)}^{W_+(0,t_-)}\circ g_{K_-(t_-)}^{w_-}$ and $g_{K_{-}^{t_+}(t_-)}^{W_-(t_+,0)}\circ g_{K_+(t_+)}^{w_+} $ agree on $I_+\cup I_-\cup I_0=\R_{\ulin w}$, and the original (iv) holds for both functions.

 (v) By (i), the composition $g_{K(\ulin t)}^{\ulin w}$ is continuous on $\R_{\ulin w}$ for any $\ulin t\in\cal D$. It suffices to show that, for any  $(a_+,a_-)\in\cal D$ and $\sigma\in\{+,-\}$, the family of maps $[0,a_\sigma]\ni t_\sigma\mapsto g_{K(\ulin t)}^{\ulin w}(v)$, $(t_{-\sigma},v)\in [0,a_{-\sigma}]\times \R_{\ulin w}$, are equicontinuous. This statement follows from the expression $g_{K(\ulin t)}^{\ulin w}=g_{K_{\sigma}^{t_{-\sigma}}(t_\sigma)}^{W_\sigma(t_{-\sigma}\ulin e_{-\sigma})}\circ g_{K_{-\sigma}(t_{-\sigma})}^{w_{-\sigma}}$, Proposition \ref{Prop-cd-continuity'} and Lemma \ref{lem-uniform} (i).
\end{proof}

\begin{Lemma}
For any $(t_+,t_-)\in\cal D$ and $\sigma\in\{+,-\}$, $W_\sigma(t_+,t_-)=g_{K_{-\sigma}^{t_\sigma}(t_{-\sigma})}^{W_{-\sigma}(t_\sigma \ulin e_\sigma)}(\ha w_\sigma(t_\sigma))$. \label{W=gw}
\end{Lemma}
\begin{proof}
Fix $\ulin t=(t_+,t_-)\in\cal D$. By symmetry, we may assume that $\sigma=+$.
If $t_-=0$, it is obvious since $W_+(\cdot,0)=\ha w_+$ and $K_{-}^{t_+}(0)=\emptyset$. Suppose $t_->0$. From (\ref{CWD}) and Lemma \ref{common-function} (i,iii,iv) we know that $W_+(\ulin t)\ge C_+(\ulin t)\ge D_-(\ulin t)=d_{K_{-}^{t_+}(t_-)}$. Since $\ha w_+(t_+)= f_{K_{-}^{t_+}(t_-)}(W_+(\ulin t))$ by Lemma \ref{lem-w-W}, we find that either $W_+(\ulin t) =d_{K_{-}^{t_+}(t_-)}$ and $\ha w_+(t_+)=b_{K_{-}^{t_+}(t_-)}$, or $W_+(\ulin t)>d_{K_{-}^{t_+}(t_-)}$ and $W_+(\ulin t)=g_{K_{-}^{t_+}(t_-)}(\ha w_+(t_+))$. In either case, we get $W_+(\ulin t)=g_{K_{-}^{t_+}(t_-)}^{W_-(t_+,0)}(\ha w_+(t_+))$.
\end{proof}

\begin{Definition}
For $v\in\R_{\ulin w}$, we call $V(\ulin t):=g_{K(\ulin t)}^{\ulin w}(v)$, $\ulin t\in \cal D$,  the force point function {started from $v$} (for the commuting pair $(\eta_+,\eta_-;\cal D)$), which is continuous by Lemma \ref{common-function} (v). {The $v$ is called the force point for this function $V(\ulin t)$. } \label{def-force-function}
\end{Definition}

{
\begin{Remark}
The name in Definition \ref{def-force-function} comes from the following fact.
  In Section \ref{section-commuting-SLE-kappa-rho}, we will study  a commuting pair of SLE$_\kappa(2, \rho,\rho_+,\rho_-)$ curves $(\eta_+,\eta_-)$, which is a.s.\ a commuting pair of chordal Loewner curves. For $\sigma\in\{+,-\}$, $\eta_\sigma$ starts from $w_\sigma$ with force points $w_{-\sigma},v_0,v_+,v_-$. Let $V_\nu$ be the force point function started from $v_\nu$, $\nu\in\{0,+,-\}$, for the commuting pair $(\eta_+,\eta_-)$. The two curves commute in the sense that, for $\sigma\in\{+,-\}$, if $\tau$ is an $\F^{-\sigma}$-stopping time, then conditionally on $\F^{-\sigma}_\tau$ and the event that $\eta_{-\sigma}$ is not complete by time $\tau$, $\eta_\sigma^\tau$ is an SLE$_\kappa(2, \rho_0,\rho_+,\rho_-)$ curve with some speed, whose driving function is $W_\sigma|^{-\sigma}_\tau$, and whose force point functions for $2,\rho_0,\rho_+,\rho_-$ are respectively $W_{-\sigma}|^{-\sigma}_\tau$, $V_0|^{-\sigma}_\tau$, $V_+|^{-\sigma}_\tau$, and $V_-|^{-\sigma}_\tau$.
\end{Remark}
}

\begin{Definition} Let $(\eta_+,\eta_-;{\cal D})$ be a commuting pair of chordal Loewner curves started from $(w_+,w_-)$ with hull function $K(\cdot,\cdot)$.
  \begin{enumerate}
    \item [(i)] For $\sigma\in\{+,-\}$, let $\phi_\sigma$ be a continuous and strictly increasing function defined on the lifespan of $\eta_\sigma$ with $\phi_\sigma(0)=0$, and let $\phi_\oplus (t_+,t_-)=(\phi_+(t_+),\phi_-(t_-))$. Let $\til\eta_\sigma=\eta\circ \phi_\sigma^{-1}$, $\sigma\in\{+,-\}$, and $\til{\cal D}=\phi_\oplus({\cal D})$.
        Then we call $(\til\eta_+,\til\eta_-;\til{\cal D})$ a commuting pair of chordal Loewner curves with speeds $(\phi_+,\phi_-)$, and call $(\eta_+,\eta_-;{\cal D})$ its normalization.
    \item [(ii)] Let $\ulin\tau\in\cal D$. Suppose there is a commuting pair of chordal Loewner curves $(\til\eta_+,\til\eta_-;\til{\cal D})$ with some speeds such that $\til{\cal D}=\{\ulin t\in\R_+^2: \ulin\tau+\ulin t\in{\cal D}\}$, and $\eta_{\sigma}(\tau_\sigma+\cdot)=f_{K(\ulin\tau)}\circ \eta_\sigma$, $\sigma\in\{+,-\}$. Then we call  $(\til\eta_+,\til\eta_-;\til{\cal D})$ the part of $(\eta_+,\eta_-;{\cal D})$ after $\ulin\tau$ up to a conformal map.
\end{enumerate}
\label{Def-speeds}
\end{Definition}

For a commuting pair $(\til\eta_+,\til\eta_-;\til{\cal D})$ with some speeds, we still define the hull function $\til K(\cdot,\cdot)$ and the capacity function $\til\mA(\cdot,\cdot)$ using (\ref{KmA}), define the driving functions $\til W_+$ and $\til W_-$ using Lemma \ref{Lem-W}, and define the force point functions by  $\til V (\ulin t)=g_{\til K(\ulin t)}^{\ulin w}(v )$ started from $v$ for any $v\in\R_{\ulin w}$. {All} lemmas in this section still hold except that {in Lemma \ref{lem-lips},} $\mA$ may not be Lipschitz continuous.

\begin{Lemma}
\begin{enumerate}
  \item [(i)] For the $(\eta_+,\eta_-;{\cal D})$, $(\til\eta_+,\til\eta_-;\til{\cal D})$ and $\phi_\oplus$ in Definition \ref{Def-speeds} (i), we have $\til X=X\circ \phi_\oplus^{-1}$ for $X\in\{K,\mA,W_\pm,V\}$, where $V$ and $\til V$ are force point functions respectively for $(\eta_+,\eta_-;{\cal D})$ and $(\til\eta_+,\til\eta_-;\til{\cal D})$ started from the same $v\in\R_{\ulin w}$.
  \item [(ii)] For the $(\eta_+,\eta_-;{\cal D})$, $(\til\eta_+,\til\eta_-;\til{\cal D})$ and $\ulin\tau$ in Definition \ref{Def-speeds} (ii), we have $\til K =K(\ulin\tau+\cdot)/K(\ulin\tau)$, $\til \mA =\mA(\ulin\tau+\cdot)-\mA(\ulin\tau)$, and $\til W_\sigma =W_\sigma(\ulin\tau+\cdot)$, $\sigma\in\{+,-\}$. Let $v\in\R_{(w_+,w_-)}$, and let $V$ be the force point function for $(\eta_+,\eta_-;{\cal D})$ started from $v$. Let $\til w_\pm=\til W_\pm(\ulin 0)=W_\pm(\ulin\tau)$.  Define $\til v\in\R_{(\til w_+,\til w_-)}$ such that if $V(\ulin\tau)\not\in \{\til w_+,\til w_-\}$, then $\til v=V(\ulin\tau)$; and if $V(\ulin\tau)=\til w_\sigma$, then $\til v=\til w_\sigma^{\sign(v-w_\sigma)}$, $\sigma\in\{+,-\}$. Let $\til V$ be the  force point function for $(\til\eta_+,\til\eta_-;\til{\cal D})$ started from $\til v$. Then $\til V  =V (\ulin\tau+\cdot)$ on $\til {\cal D}$.
\end{enumerate}
\label{DMP-determin-1}
\end{Lemma}
\begin{proof}
Part (i) is obvious. We now work on (ii). Let $\ulin t=(t_+,t_-)\in\til{\cal D}$. From $K(\ulin\tau+\ulin t)=\Hull(\bigcup_\sigma \eta_\sigma[0,\tau_\sigma+t_\sigma])$, we get
$$K(\ulin\tau+\ulin t)=\Hull(K(\ulin\tau)\cup \bigcup_\sigma \eta_\sigma[\tau_\sigma,\tau_\sigma+t_\sigma])=\Hull(K(\ulin\tau)\cup f_{K(\ulin\tau)}( \bigcup_\sigma \til \eta_\sigma[0,t_\sigma])).$$
This implies that $\til K(\ulin t)=\Hull( \bigcup_\sigma \til \eta_\sigma[0,t_\sigma])=K(\ulin\tau+\ulin t)/K(\ulin\tau)$, which then implies that $\til \mA(\ulin t)=\mA(\ulin\tau+\ulin t)-\mA(\ulin\tau)$. It together with  (\ref{K123},\ref{W-def})  implies that  $\til W_\sigma(\ulin t)=W_\sigma(\ulin\tau+\ulin t)$.

By (i), Proposition \ref{prop-comp-g} and Lemma \ref{common-function}, if $V(\ulin\tau)\not\in \{\til w_+,\til w_-\}$,
$$\til V(\ulin t)=g_{\til K(\ulin t)/\til K_-(t_-)}^{\til W_+(0,t_-)}\circ g_{\til K_-(t_-)}^{\til w_-}(\til v)=g_{K(\ulin\tau+\ulin t)/ K(\tau_+,\tau_-+t_-)}^{W_+(\tau_+,\tau_-+t_-)}\circ g_{K(\tau_+,\tau_-+t_-)/K(\ulin\tau)}^{W_-(\ulin \tau)}(\til v)$$
$$=g_{ K(\ulin \tau+\ulin t)/K(\tau_+,\tau_-+t_-)}^{W_+(\tau_+,\tau_-+t_-)}\circ g_{K(\tau_+,\tau_-+t_-)/K(\ulin\tau)}^{W_-(\ulin \tau)}\circ g^{W_-(\tau_+,0)}_{K(\ulin\tau)/K(\tau_+,0)} \circ g^{w_+}_{K(\tau_+,0)}(v)$$
$$=g_{ K(\ulin \tau+\ulin t)/K(\tau_+,\tau_-+t_-)}^{W_+(\tau_+,\tau_-+t_-)}\circ g_{K(\tau_+,\tau_-+t_-)/K(\tau_+,0)}^{W_-( \tau_+,0)} \circ g^{w_+}_{K(\tau_+,0)}(v)$$
$$=g_{ K(\ulin \tau+\ulin t)/K(\tau_+,\tau_-+t_-)}^{W_+(\tau_+,\tau_-+t_-)}\circ g_{K(\tau_+,\tau_-+t_-)/K(0,\tau_-+t_-)}^{W_+(0,\tau_-+t_-)} \circ g^{w_-}_{K(0,\tau_-+t_-)}(v)$$
$$=g_{ K(\ulin \tau+\ulin t) }^{W_+(0,\tau_-+t_-)}\circ g^{w_-}_{K(0,\tau_-+t_-)}(v) =g_{K(\ulin\tau+\ulin t)}^{(w_+,w_-)}(v)=V(\ulin\tau+\ulin t).$$
Here the ``$=$'' in the $1^\text{st}$ line follow from the definition of $\til V(\ulin t)$ and that $\til K =K(\ulin\tau+\cdot)/K(\ulin\tau)$, the ``$=$'' in the $2^\text{nd}$ line follows from the definition of $\til v$,
the ``$=$'' in the $3^\text{rd}$  line and the first ``$=$'' in the $5^\text{th}$  line follow from Proposition \ref{prop-comp-g}, and the ``$=$''  in the $4^\text{th}$ line follows from Lemma \ref{common-function}.

Now suppose $V(\ulin\tau)\in \{\til w_+,\til w_-\}$. By symmetry,  assume that $V(\ulin\tau)=\til w_-=W_-(\ulin\tau)$.  If $v\ge w_-^+$,  we understand $g^{W_-(\tau_+,0)}_{K(\ulin\tau)/K(\tau_+,0)}$ as a map from $[W_-(\tau_+,0)^+,\infty)$ into $[\til w_-^+,\infty)$, i.e., when $g^{W_-(\tau_+,0)}_{K(\ulin\tau)/K(\tau_+,0)}$ takes value $\til w_-$ at some point in $[W_-(\tau_+,0)^+,\infty)$,   we redefine the value as $\til w_-^+$. Then the above displayed formula still holds. The case that $v\le w_-^-$ is similar, in which we understand $g^{W_-(\tau_+,0)}_{K(\ulin\tau)/K(\tau_+,0)}$ as a map from $(-\infty, W_-(\tau_+,0)^-]$ into $(-\infty,\til w_-^-]$.
\end{proof}

From now on, we fix $v_0\in I_0= [w_-^+,w_+^-]$, $v_+\in I_+=[w_+^+,\infty)$, and $v_-\in I_-=(-\infty,w_-^-]$, and let $V_\nu(\ulin t)$, $\ulin t\in\cal D$, be the force point function started from $v_\nu$, $\nu\in\{0,+,-\}$.   By Lemma \ref{common-function}, $V_-\le C_-\le D_-\le V_0\le C_+\le D_+\le V_+$, which combined with (\ref{CWD}) implies
\BGE V_-\le C_-\le W_-\le D_-\le V_0\le C_+\le W_+\le D_+\le V_+. \label{VWVWV}\EDE

\begin{Lemma}
	For any $\ulin t=(t_+,t_-)\in\cal D$, we have
	\BGE |V_+(\ulin t)-V_-(\ulin t)|/4\le \diam(K(\ulin t)\cup [v_-,v_+])\le |V_+(\ulin t)-V_-(\ulin t)|.\label{V-V}\EDE
	\BGE f_{K(\ulin t)}[V_0(\ulin t),V_\nu(\ulin t)]\subset \eta_\nu[0,t_\nu]\cup [v_0,v_\nu],\quad \nu\in\{+,-\}\label{fV1}\EDE
Here for $x,y\in\R$, the $[x,y]$ in (\ref{fV1}) is the line segment connecting $x$ with $y$, which is the same as $[y,x]$; and if any $v_\nu$, $\nu\in\{0,+,-\}$, takes value $w_\sigma^\pm$ for some $\sigma\in\{+,-\}$, then its appearance in (\ref{V-V},\ref{fV1}) is understood as $w_\sigma$. {See Figure \ref{V-figure}.} \label{lem-V0}
\end{Lemma}

\begin{figure}
	\begin{center}
		\includegraphics[width=3in]{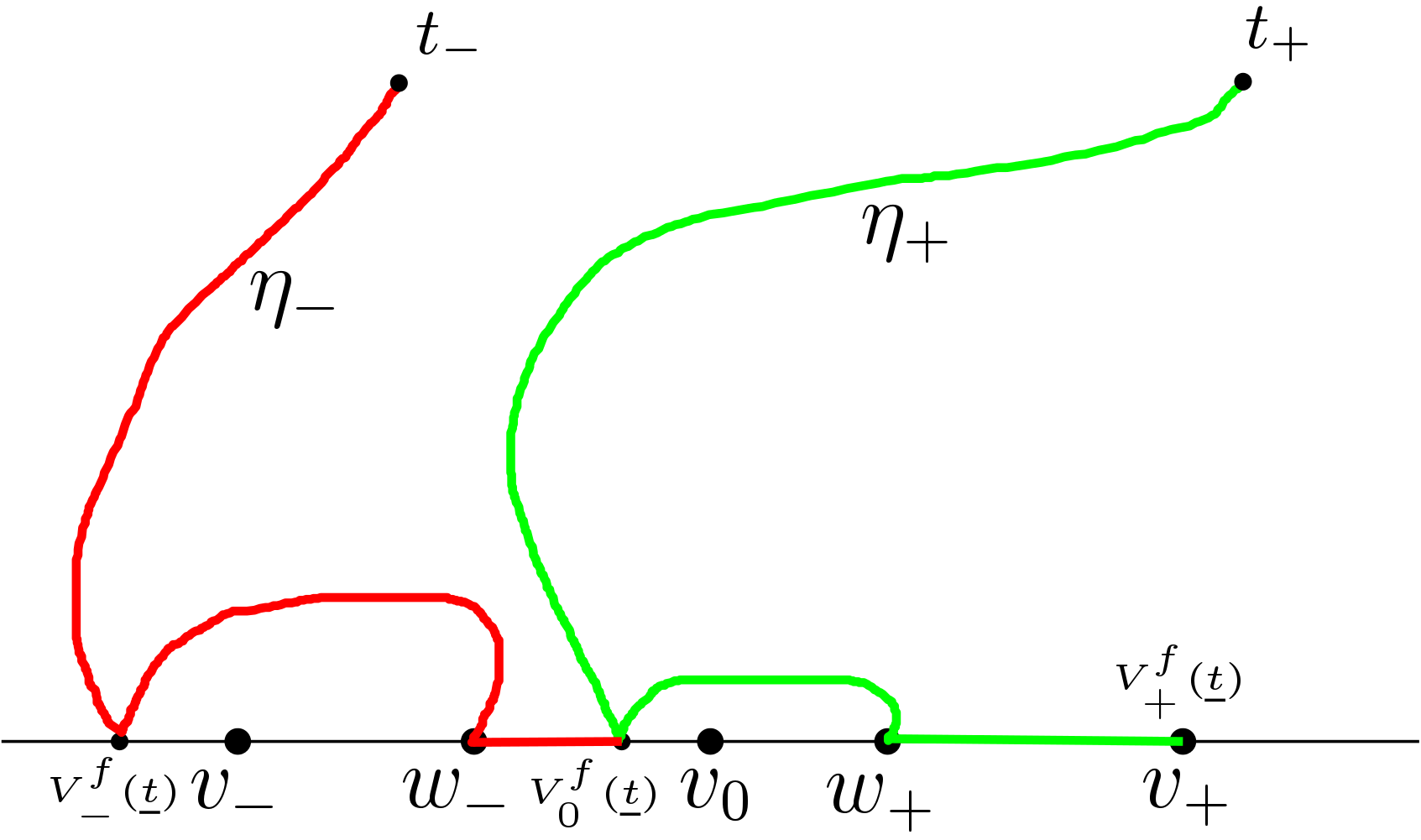}\quad
		\includegraphics[width=3in]{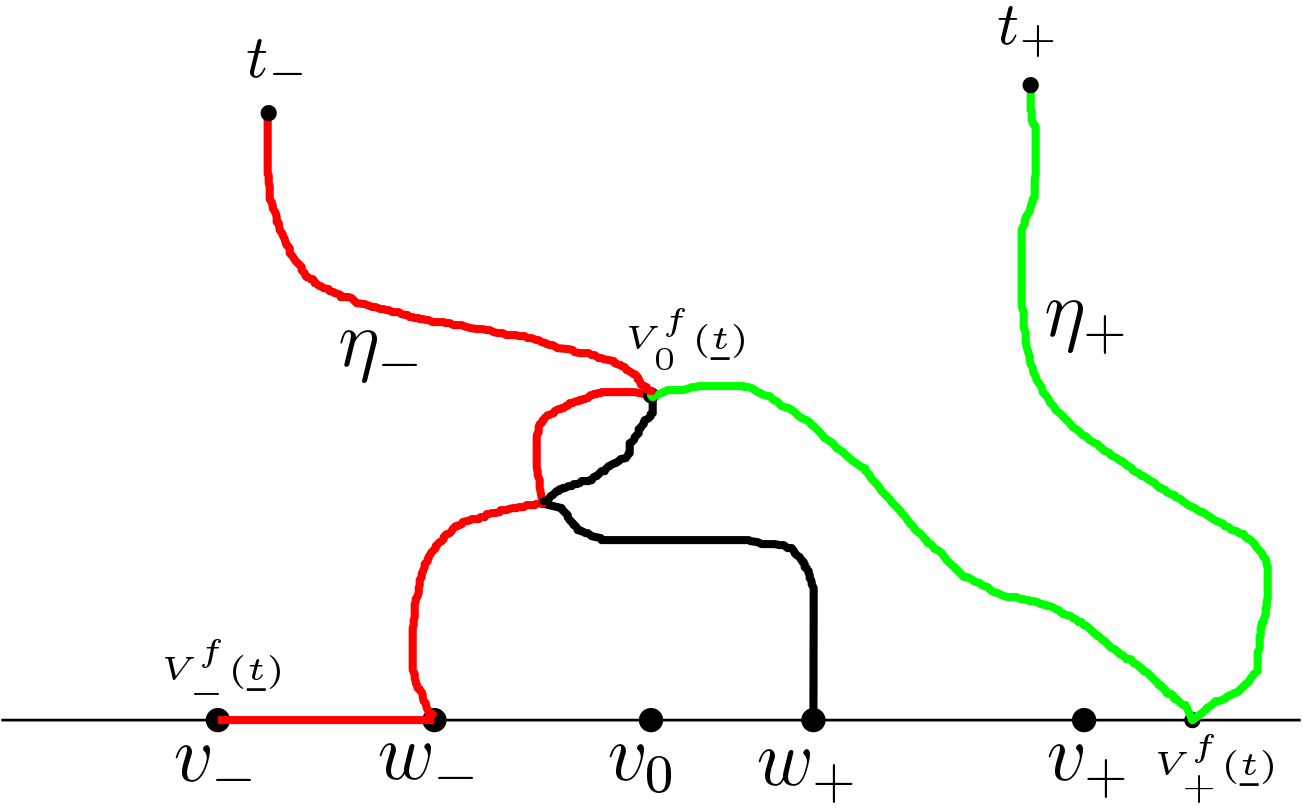}
	\end{center}
	\caption{{The above figures illustrate two situations. In each situation, the curves $\eta_+$ and $\eta_-$ are respectively stopped at the time $t_+$ and $t_-$. Let $\ulin t=(t_+,t_-)$. For $\nu\in\{0,+,-\}$, the point $f_{K(\ulin t)}(V_\nu(\ulin t))$ is labeled by $V_\nu^f(\ulin t)$, which  agrees with $v_\nu$ in the case $v_\nu\not\in \lin{K(\ulin t)}$. The sets $f_{K(\ulin t)}[V_0(\ulin t),V_+(\ulin t)]$ and $f_{K(\ulin t)}[V_-(\ulin t),V_0(\ulin t)]$ are respectively colored green and red.}} \label{V-figure}
\end{figure}

\begin{proof}
	Fix $\ulin t=(t_+,t_-)\in\cal D$.  We write $K$ for $K(\ulin t)$, $K_\pm$ for $K_\pm(t_\pm)$, $\til K_\pm$ for $K_{\pm}^{t_\mp}(t_\pm)$, $\eta_\pm$ for $\eta_\pm[0,t_\pm]$, $\til\eta_\pm$ for $\eta_{\pm}^{t_\mp}[0,t_\pm]$, and $X$ for $X(\ulin t)$, $X\in\{V_0,V_+,V_-,C_+,C_-,D_+,D_-\}$.
	
	Since $g_{K }$ maps $\C\sem (K^{\doub}\cup [v_-,v_+])$ conformally onto $\C\sem [V_- ,V_+ ]$, fixes $\infty$, and has derivative $1$ at $\infty$, by Koebe's $1/4$ theorem, we get (\ref{V-V}).
	For (\ref{fV1}) by symmetry we only need to prove the case $\nu=+$. By (\ref{VWVWV}), $V_0\le C_+\le D_+\le V_+$. By (\ref{f[C,D]}),
$f_K[C_+,D_+]\subset f_{K_-}(\til\eta_+)=\eta_+$. It remains to show that  $f_K(D_+,V_+]\subset [w_0,v_+]$ and $f_K[V_0,C_+)\subset [v_0,w_0]$. If $V_+=D_+$, then $(D_+,V_+]=\emptyset$, and $f_K(D_+,V_+]=\emptyset \subset [w_+,v_+]$. Suppose $V_+>D_+$. By Lemma \ref{common-function}, $D_+=\lim_{x\downarrow \max((\lin K\cap \R)\cup\{w_+\})} g_K(x)$, and $V_+=g_K(v_+)$. So $f_K(D_+,V_+]=( \max((\lin K\cap \R)\cup\{w_+\}),v_+]\subset [w_+,v_+]$.
	If $V_0=C_+$, then $[V_0,C_+)=\emptyset$, and $f_K[V_0,C_+)=\emptyset \subset [v_0,w_+]$. If $V_0<C_+$, by Lemma \ref{common-function} (iii,iv),   $\lin{K_+}\cap\lin{K_-}=\emptyset$, $v_0\ne \lin{K_+}$, and $C_+=\lim_{x\uparrow \min((\lin{K_+}\cap\R)\cup\{w_+\})} g_K(x)$. Now either $v_0\not\in \lin K\cup\{w_-^+\}$ and $V_0=g_K(v_0)$, or $v_0\in \lin {K_-}\cup\{w_-^+\}$ and $V_0=D_-$. In the former case, $f_K[V_0,C_+)\subset [v_0,\min((\lin{K_+}\cap\R)\cup\{w_+\}))\subset [v_0,w_+]$. In the latter case, $f_K[V_0,C_+)= [\max((\lin{K_-}\cap\R)\cup\{w_-\}),\min((\lin{K_+}\cap\R)\cup\{w_+\}))\subset [v_0,w_+]$. \end{proof}

\begin{Lemma}
  Suppose for some $\ulin t=(t_+,t_-)\in\cal D$ and $\sigma\in\{+,-\}$,  $\eta_\sigma(t_\sigma)\in \eta_{-\sigma}[0,t_{-\sigma}]\cup [v_{-\sigma},v_0]$. Then $W_\sigma(\ulin t)=V_0(\ulin t)$.
  \label{W=V}
\end{Lemma}
\begin{proof}
  We assume $\sigma=+$ by symmetry. By Lemma \ref{W=gw}, $W_+(\ulin t)= g_{K_-^{t_+}(t_-)}^{W_-(t_+,0)}(\ha w_+(t_+))$. By Lemma \ref{common-function},  $V_0(\ulin t)=g_{K_-^{t_+}(t_-)}^{W_-(t_+,0)}\circ g_{K_+(t_+)}^{w_+}(v_0)$. If $\eta_+(t_+)\in[v_-,v_0]$, then $\ha w_+(t_+)=c_{K_+(t_+)}=g_{K_+(t_+)}^{w_+}(v_0)$, and so we get $W_\sigma(\ulin t)=V_0(\ulin t)$. Now suppose $\eta_+(t_+)\in \eta_-[0,t_-]$. Then $\ha w_+(t_+)\in \lin{K_-^{t_+}(t_-)}$, which together with $\ha w_+(t_+)=W_+(t_+,0)\ge V_0(t_+,0)=g_{K_+(t_+)}^{w_+}(v_0)\ge W_-(t_+,0)$ implies that $g_{K_-^{t_+}(t_-)}^{W_-(t_+,0)}\circ g_{K_+(t_+)}^{w_+}(v_0)=g_{K_-^{t_+}(t_-)}^{W_-(t_+,0)}(\ha w_+(t_+))$, as desired.
\end{proof}

\subsection{Ensemble without intersections}  \label{section-deterministic-2}
We say that $(\eta_+,\eta_-;{\cal D})$ is disjoint if $\lin {K_+(t_+)}\cap \lin{K_-(t_-)}=\emptyset$ for any $(t_+,t_-)\in\cal D$. Given a  commuting pair $(\eta_+,\eta_-;{\cal D})$, we get  a disjoint commuting par $(\eta_+,\eta_-;{\cal D}_{\disj})$ by defining
\BGE {\cal D}_{\disj}=\{(t_+,t_-)\in{\cal D}:\lin {K_+(t_+)}\cap \lin{K_-(t_-)}=\emptyset\}.\label{D-disj}\EDE

In this subsection, we assume that $(\eta_+,\eta_-;{\cal D})$ is disjoint. {In addition to the $W_\pm,V_0,V_\pm$ defined in the previous subsection, we are going to define on $\cal D$ the functions $W_{\pm,j}$, $j=1,2,3$, $W_{\pm,S}$, $Q$, $W_{\pm,N}$, $V_{\nu,N}$, $\nu\in\{0,+,-\}$, and $E_{X,Y}$ for $X\ne Y\in\{W_+,W_-,V_0,V_+,V_-\}$. We will also derive differential equations for these functions, which will be applied to the random setting in Section \ref{section-indep} to construct some two-time-parameter local martingale.} 

We now write $g_\sigma^{t_{-\sigma}}(t_\sigma,\cdot)$ for $g_{K_\sigma^{t_{-\sigma}}(t_\sigma)}$, $\sigma\in\{+,-\}$.
For $(t_+,t_-)\in\cal D$ and  $\sigma\in\{+,-\}$,  {since} $\lin {K_+(t_+)}\cap \lin{K_-(t_-)}=\emptyset$, $ [c_{K_\sigma(t_\sigma)},d_{K_\sigma(t_\sigma)}]$ has positive distance from $K_{-\sigma}^{t_\sigma}(t_{{-}\sigma}))$. So $g_{-\sigma}^{t_\sigma}(t_{-\sigma},\cdot)$ is analytic at $\ha w_\sigma(t_\sigma)\in [c_{K_\sigma(t_\sigma)},d_{K_\sigma(t_\sigma)}]$. By Lemma \ref{W=gw}, $W_{\sigma }(t_+,t_-)=g_{-\sigma}^{t_\sigma}(t_{-\sigma},  \ha w_\sigma(t_\sigma))$.
We further define  $W_{\sigma,j}$, $j=1,2,3$, and $W_{\sigma,S}$ on ${\cal D}$ by \BGE W_{\sigma,j}(t_+,t_-)=(g_{-\sigma}^{t_\sigma})^{(j)}(t_{-\sigma},\ha w_\sigma(t_\sigma)),\quad W_{\sigma,S}=\frac{W_{\sigma,3}}{W_{\sigma,1}}-\frac 32\Big(\frac{W_{\sigma,2}}{W_{\sigma,1}}\Big)^2,\quad \sigma\in\{+,-\}.\label{WS}\EDE
Here the superscript $(j)$ means the $j$-th complex derivative w.r.t.\ the space variable.
The functions are all continuous on $\cal D$ because $(t_+,t_-,z)\mapsto (g_{-\sigma}^{t_\sigma})^{(j)}(t_{-\sigma},z)$ is continuous by Lemma \ref{lem-uniform}.
Note that $W_{\sigma,S}(t_+,t_-)$ is the Schwarzian derivative of $g_{-\sigma}^{t_\sigma}(t_{-\sigma},\cdot)$ at $\ha w_\sigma(t_\sigma)$.

\begin{Lemma}
$\mA $ is continuously differentiable  with   $\pa_\sigma \mA= W_{\sigma, 1} ^2$, $\sigma\in\{+,-\}$.
\label{lem-positive}
\end{Lemma}
\begin{proof}
This follows from a standard argument, which first appeared in \cite[Lemma 2.8]{LSW1}. The statement for ensemble of chordal Loewner curves appeared in \cite[Formula (3.7)]{reversibility}.
\end{proof}

So  for any $\sigma\in\{+,-\}$ and  $t_{-\sigma}\in{\cal I}_{-\sigma}$, $K_{\sigma}^{t_{-\sigma}}(t_\sigma)$, $0\le t_\sigma<T_\sigma^{\cal D}(t_{-\sigma})$, are chordal Loewner hulls  driven by $W_{\sigma}|^{-\sigma}_{t_{-\sigma}}$ with speed $(W_{\sigma,1}|^{-\sigma}_{t_{-\sigma}})^2$, and we get the differential equation:
\BGE \pa_{t_\sigma} g_{\sigma}^{t_{-\sigma}}(t_\sigma,z)= \frac{2 (W_{\sigma,1}(t_+,t_-)^2}{ g_{\sigma}^{t_{-\sigma}}(t_\sigma,z)-W_\sigma(t_+,t_-)},\label{pag}\EDE
which together with Lemmas \ref{W=gw} and \ref{common-function}
implies the differential equations for $V_0,V_+,V_-$:
\BGE \pa_\sigma V_\nu \aeq \frac{2W_{\sigma,1}^2}{V_\nu-W_\sigma},\quad \nu\in\{0,+,-\},\label{pa-X}\EDE
and the differential equations for $W_{-\sigma}$, $W_{-\sigma,1}$ and $W_{-\sigma,S}$:
\BGE  \pa_{\sigma} W_{-\sigma} = \frac{2W_{\sigma,1}^2}{W_{-\sigma}-W_{\sigma}},\quad \frac{\pa_{\sigma} W_{-\sigma,1}}{W_{-\sigma,1}}=\frac{-2W_{\sigma,1}^2}{(W_+-W_-)^2},\quad \pa_{\sigma} W_{-\sigma,S}=-\frac{12 W_{+,1}^2 W_{-,1}^2}{(W_+-W_-)^4}.\label{pajW}\EDE

Define $Q$ on ${\cal D}$ by
\BGE Q(\ulin t)=\exp\Big(\int_{[\ulin 0,\ulin t]} -\frac{12 W_{+,1}(\ulin s) ^2W_{-,1}(\ulin s) ^2}{(W_+(\ulin s) -W_-(\ulin s) )^4}d^2\ulin s\Big).\label{F}\EDE
Then $Q$ is continuous and positive with $Q(t_+,t_-)=1$ when {$t_+=0$ or $t_-=0$}. From (\ref{pajW}) we get
\BGE \frac{\pa_\sigma Q}{Q}=W_{\sigma,S},\quad \sigma\in\{+,-\}.\label{paF}\EDE
{
\begin{Remark} The function $Q$ implicitly appeared earlier in \cite{reversibility,duality}: $Q^{-\frac{\cc}6}$   agrees with the second factor on the RHS of \cite[Formula (4.3)]{reversibility}. It is related to the Brownian loop measure (\cite{loop}) by the fact that $-\frac 16 \log Q(\ulin t)$ equals the Brownian loop measure of the set of loops in $\HH$ that intersect both $K_+(t_+)$ and $K_-(t_-)$. The ODE (\ref{paF}) reflects the facts that  the Brownian loop measure can be decomposed into Brownian bubble measures along a curve, and the Schwarzian derivatives are related to the Brownian bubble measures.
\end{Remark}
}

By (\ref{VWVWV}), $V_+\ge W_+\ge C_+\ge V_0\ge D_- \ge W_-\ge V_-$ on $\cal D$. Because of disjointness, we further have  $C_+>D_-$ by Lemma \ref{common-function}.
By the same lemma,
\BGE V_\sigma(t_+,t_-)=g_{-\sigma}^{t_\sigma}(t_{-\sigma},V_\sigma(t_\sigma \ulin e_\sigma));\label{WV+g}\EDE
\BGE V_0(t_+,t_-)=g_{-\sigma}^{t_\sigma}(t_{-\sigma},V_0(t_\sigma \ulin e_\sigma)),\quad \mbox{if }v_0\not\in\lin{K_{-\sigma}(t_{-\sigma})}.\label{V0-g}\EDE

Let $\ulin t=(t_+,t_-)\in {\cal D}$. For $\sigma\in\{+,-\}$, differentiating (\ref{circ-g}) w.r.t.\ $t_\sigma$, letting $\ha z=g_{K_\sigma(t_\sigma)}(z)$, and using Lemma \ref{W=gw} and (\ref{pag},\ref{WS}) we get
\BGE \pa_{t_\sigma} g_{-\sigma}^{t_\sigma}(t_{-\sigma},\ha z)=\frac{2  (g_{-\sigma}^{t_\sigma})'(t_{-\sigma},\ha w_\sigma(t_\sigma))^2}{g_{-\sigma}^{t_\sigma}(t_{-\sigma},\ha z)-g_{-\sigma}^{t_\sigma}(t_{-\sigma},\ha w_\sigma(t_\sigma))}-\frac{2(g_{-\sigma}^{t_\sigma})'(t_{-\sigma},\ha z)}{\ha z-\ha w_\sigma(t_\sigma)}.\label{diff}\EDE
Letting $\HH\sem {K_{-\sigma}^{t_\sigma}(t_{-\sigma})}\ni \ha z\to \ha w_\sigma(t_\sigma)$ and using the power series expansion of $g_{-\sigma}^{t_\sigma}(t_{-\sigma},\cdot)$ at $\ha w_\sigma(t_\sigma)$, we get
\BGE \pa_{t_\sigma} g_{-\sigma}^{t_\sigma}(t_{-\sigma},\ha z)|_{\ha z=\ha w_\sigma(t_\sigma)}=-3 W_{\sigma,2}(\ulin t),\quad \sigma\in\{+,-\}.\label{-3}\EDE
Differentiating (\ref{diff}) w.r.t.\ $\ha z$ and letting $ \ha z\to \ha w_\sigma(t_\sigma)$, we get
\BGE \frac{ \pa_{t_\sigma} (g_{-\sigma}^{t_\sigma})'(t_{-\sigma},\ha z)}{   (g_{-\sigma}^{t_\sigma})'(t_{-\sigma},\ha z)}\bigg|_{\ha z=\ha w_\sigma(t_\sigma)}=\frac 12\Big(\frac{W_{\sigma,2}(\ulin t)}{W_{\sigma,1}(\ulin t)}\Big)^2-\frac 43\frac{W_{\sigma,3}(\ulin t)}{W_{\sigma,1}(\ulin t)},\quad \sigma\in\{+,-\}.\label{1/2-4/3}\EDE

For $\sigma\in\{+,-\}$, define $ W_{\sigma,N}$ on ${\cal D}$ by  $W_{\sigma,N} =\frac{ W_{\sigma,1} }{ W_{\sigma,1}|^{\sigma}_0 }$. Since $W_{\sigma,1}|^{-\sigma}_0\equiv 1$, we get $  W_{\sigma,N}(t_+,t_-)=1$ when $t_+t_-=0$. From (\ref{pajW}) we get
\BGE \frac{\pa_{\sigma}  W_{-\sigma,N}}{ W_{-\sigma,N}}=\frac{-2 W_{\sigma,1}^2}{(W_{-\sigma}-W_{\sigma})^2}\,\pa t_\sigma-\frac{-2 W_{\sigma,1}^2}{(W_{-\sigma}-W_{\sigma})^2}\bigg|^{-\sigma}_0 \,\pa t_\sigma,\quad \sigma\in\{+,-\}.
\label{pajhaW}\EDE

We now define $  V_{0,N},V_{+,N},  V_{-,N}$ on ${\cal D} $ by
$$  V_{\nu,N}(\ulin t)=(g_{-\nu}^{t_\nu})'(t_{-\nu}, V_\nu(t_\nu \ulin e_\nu))/(g_{-\nu}^{0})'(t_{-\nu}, v_\nu),\quad \nu\in\{+,-\};$$
\BGE   V_{0,N}(\ulin t)=(g_{-\sigma}^{t_\sigma})'(t_{-\sigma},V_0(t_\sigma \ulin e_\sigma))/(g_{-\sigma}^{t_\sigma})'(t_{-\sigma},v_0),\quad  \mbox{if }  v_0\not\in  \lin{K_{-\sigma}(t_{-\sigma})},\quad \sigma\in\{+,-\}.\label{V1V2}\EDE
The functions are well defined because  either $v_0\not\in \lin{K_+(t_+)}$ or $v_0\not\in \lin{K_-(t_-)}$, and when they both hold, the RHS of (\ref{V1V2}) equals $g_{K(t_+,t_-)}'(v_0)/(g_{K_+(t_+)}'(v_0)g_{K_-(t_-)}'(v_0))$ by (\ref{circ-g}).

Note that $ V_{\nu,N}(t_+,t_-)=1$ if $t_+t_-=0$ for $\nu\in\{0,+,-\}$. From (\ref{WV+g}-\ref{V0-g}) and (\ref{circ-g},\ref{pag}) we find that these functions satisfy the following differential equations on ${\cal D}$:
\BGE  \frac{\pa_{\sigma}   V_{\nu,N}} { V_{\nu,N}}= \frac{-2W_{\sigma,1}^2}{(V_{\nu}-W_\sigma)^2}\,\pa t_\sigma-\frac{-2W_{\sigma,1}^2}{(V_{\nu}-W_\sigma)^2}\bigg|^{-\sigma}_0\,\pa t_\sigma,\quad \sigma\in\{+,-\},\quad \nu\in\{0,-\sigma\},\quad \mbox{if }v_\nu\not\in \lin{K_\sigma(t_\sigma)}. \label{paV1}\EDE

We now define $E_{X,Y}$ on ${\cal D}$ for $X\ne Y\in\{W_+,W_-,V_0,V_+,V_-\}$ as follows. First, let
\BGE E_{X,Y}(t_+,t_-)=\frac{(X(t_+,t_-)-Y(t_+,t_-))(X(0,0)-Y(0,0))} {(X(t_+,0)-Y(t_+,0))(X(0,t_-)-Y(0,t_-))},\label{RXY}\EDE
if the denominator is not $0$. If the denominator is $0$, then since $V_+\ge W_+\ge V_0\ge W_-\ge V_-$ and $W_+>W_-$, there is $\sigma\in\{+,-\}$ such that $\{X,Y\}\subset\{W_\sigma,V_\sigma,V_0\}$.  By symmetry, we will only describe the definition of $E_{X,Y}$ in the case that $\sigma=+$. If $X(t_+,0)=Y(t_+,0)$, by Lemmas \ref{common-function} and \ref{W=gw}, $X(t_+,\cdot)\equiv Y(t_+,\cdot)$. If $X(0,t_-)=Y(0,t_-)$, then we must have $X(\ulin 0)=Y(\ulin 0)$, and so $X(0,\cdot)\equiv Y(0,\cdot)$. For the definition of $E_{X,Y}$, we modify (\ref{RXY}) by writing the RHS as $\frac{X(t_+,t_-)-Y(t_+,t_-)}{X(t_+,0)-Y(t_+,0) }:\frac{ X(0,t_-)-Y(0,t_-)}{ X(0,0)-Y(0,0)}$,  replacing the first factor (before ``$:$'') by $(g_{-}^{t_+})'(t_-,X(t_+,0))$ when $X(t_+,0)=Y(t_+,0)$,   replacing the second factor (after ``$:$'')  by $g_{K_-(t_-)}'(X(0,0))$ when $X(0,t_-)=Y(0,t_-)$; and do both replacements when two equalities both hold. Then in all cases, $E_{X,Y}$ is continuous and positive on ${\cal D}$, and $E_{X,Y}(t_+,t_-)=1$ if {$t_+=0$ or $t_-=0$}. By (\ref{pa-X},\ref{pajW}), for $\sigma\in\{+,-\}$, if $X,Y\ne W_\sigma$, then
\BGE \frac{\pa_\sigma E_{X,Y}}{E_{X,Y}}\aeq \frac{-2W_{\sigma,1}^2}{(X-W_\sigma)(Y-W_\sigma)}\,\pa t_\sigma- \frac{-2W_{\sigma,1}^2}{(X-W_\sigma)(Y-W_\sigma)}\Big|^{-\sigma}_0\,\pa t_\sigma.\label{paEXY}\EDE

\subsection{A time curve in the time region} \label{time curve}
{We call the set $\cal D$ a time region, which is composed of two-dimensional time variables, whose two components are one-dimensional time variables for $\eta_+$ and $\eta_-$ respectively. A time curve in $\cal D$ is a continuous and strictly increasing function $\ulin u=(u_+,u_-):[0,T^u)\to \cal D$ with $\ulin u(0)=\ulin 0$. Such a time curve can be used to grow $\eta_+$ and $\eta_-$ simultaneously. This means that we construct two curves $\eta^u_\sigma:=\eta_\sigma\circ u_\sigma$,   $\sigma\in\{+,-\}$, on the same interval $[0,T^u)$, which are time-changes of some initial segments of $\eta_+$ and $\eta_-$. In this subsection, we are going to construct a special time curve in $\cal D$ such that if we grow $\eta_+$ and $\eta_-$ simultaneously using $\ulin u$, then the factors (F1) and (F2) in Section \ref{Strategy} are satisfied. Later in the next section, for a commuting pair of SLE$_\kappa(2,\rho_0,\rho_+,\rho_-)$ curves, we will use the driving functions $W_\sigma$, $\sigma\in\{+,-\}$, and the force point functions $V_\nu$, $\nu\in\{0,+,-\}$, and the time-curve $\ulin u$ to construct a diffusion process $(\ulin R(t))_{0\le t<T^u}$, whose transition density then leads to the proof of the main theorem of the paper.
}

We {use the settings and results in the previous subsections except that we }
do not assume that $(\eta_+,\eta_-;\cal D)$ is disjoint. 
We {made an additional assumption} in this subsection that
\BGE v_+-v_0=v_0-v_-.\label{add-assump}\EDE
Define  $\Lambda$ and $\Upsilon$ on $\cal D$ by $\Lambda=\frac 12\log\frac{V_+-V_0}{V_0-V_-}$ and $\Upsilon=\frac 12\log\frac{V_+-V_-}{v_+-v_-}$. By {the additional} assumption (\ref{add-assump}), we have $\Lambda(\ulin 0)=\Upsilon(\ulin 0)=0$.

\begin{Lemma} There exists a unique  continuous and strictly increasing function $\ulin u:[0,T^u)\to {\cal D}$, for some $T^u\in(0,\infty]$,  with $\ulin u(0)=\ulin 0$, such that for any $0\le t<T^u$ and $\sigma\in\{+,-\}$, $|V_\sigma(\ulin u(t))-V_0(\ulin u(t))|=e^{2t}|v_\sigma-v_0|$; and $\ulin u$ {cannot} be extended beyond $T^u$ {while still satisfying this property}. \label{time curve-lem}
\end{Lemma}

\begin{proof}
The proof resembles the argument in \cite[Section 4]{Two-Green-interior}. It is clear that the property of $\ulin u$ is equivalent to that $\Lambda(\ulin u(t))=0$ and $\Upsilon (\ulin u(t))=t$ for $0\le t<T^u$.
By (\ref{VWVWV}), the definition of $V_\nu$ ($V_\nu(\ulin t)=g_{K(\ulin t)}^{\ulin w}(v_\nu)$), Lemma \ref{common-function} (the definition of $g_{K(\ulin t)}^{\ulin w}$) and Proposition \ref{Prop-cd-continuity'}, we see that, for $\sigma\in\{+,-\}$, $|V_\sigma-V_0|$ and $|V_\sigma-V_{-\sigma}|$ are strictly increasing in $t_\sigma$, and $|V_0-V_{-\sigma}|$ is strictly decreasing in $t_\sigma$. Thus, $\Lambda$ is strictly increasing in $t_+$ and strictly decreasing in $t_-$, and $\Upsilon$ is strictly increasing in both $t_+$ and $t_-$. Since $\Lambda(\ulin 0)=0$, we see that $\Lambda>0$ on $[0,T_+)\times \{0\}$ and $<0$ on $\{0\}\times [0,T_-)$.
For $\sigma\in \{+,-\}$, let
$$T^u_\sigma=\sup\{t_\sigma\in[0,T_\sigma): \exists t_{-\sigma}\in [0,T_{-\sigma}^{\cal D}(t_\sigma))\mbox{ such that } \sigma \Lambda(t_\sigma \ulin e_\sigma+t_{-\sigma}\ulin e_{-\sigma})\le 0\}.$$

By the definition of $T^u_+$, $\Lambda>0$ on ${\cal D}\cap (T^u_+,\infty)\times \R_+$. By the continuity of $\Lambda$, we have $\Lambda\ge 0$ on ${\cal D}\cap [T^u_+,\infty)\times \R_+$. If $\Lambda(T^u_+,t_-)=0$ for some $t_-\in\R_+$ such that $(T^u_+,t_-)\in\cal D$, then by the strictly deceasingness of $\Lambda$ in $t_-$, there is $t_-'>t_-$ such that $(T^u_+,t_-')\in\cal D$ and $\Lambda(T^u_+,t_-)<0$, which is a contradiction. So $\Lambda>0$ on ${\cal D}\cap [T^u_+,\infty)\times \R_+$. For $t_+\in[0,T^u_+)$, by intermediate value theorem and the strictly decreasingness of $\Lambda$ in $t_-$,  there exists a unique $t_-\in [0,T^{\cal D}_-(t_+))$ such that $\Lambda(t_+,t_-)=0$. We then define a function $u_{+\to -}:\to [0,T^u_+)\to [0,T_-)$ such that $\Lambda(t_+,u_{+\to -}(t_+))=0$, $0\le t_+<T^u_+$. Note that $u_{+\to -}(0)=0$. Since $\Lambda$ is strictly increasing (resp.\ decreasing) in $t_+$ (resp.\ $t_-$), $u_{+\to -}$ is strictly increasing.
Similarly, $\Lambda <0$ on ${\cal D}\cap \R_+\times [T^u_-,\infty)$, and there is a strictly increasing function $u_{-\to +}:[0,T^u_-)\to [0,T_+)$ with $u_{-\to +}(0)=0$ such that $\Lambda(u_{-\to +}(t_-),t_-)=0$, $0\le t_-<T^u_-$. Since $\Lambda>0$ on ${\cal D}\cap [T^u_+,\infty)\times \R_+$. We see that $u_{-\to +}$ takes values in $[0,T^u_+)$. Similarly, $u_{+\to -}$ takes values in $[0,T^u_-)$. From  $\Lambda(t_+,u_{+\to -}(t_+))=\Lambda(u_{-\to +}(t_-),t_-)=0$ and the monotonicity of $\Lambda$, we see that $u_{+\to -}$ and $u_{-\to +}$ are inverse of each other, and are both continuous

By the continuity and strictly increasingness of $u_{+\to -}$ and $\Upsilon$, the map $[0,T^u_+)\ni t\mapsto \Upsilon(t,u_{+\to -}(t))$ is continuous and strictly increasing, and so its range is $[0,T^u)$ for some $T^u\in (0,\infty]$. Let $u_+$ denote the inverse of this map, and let $u_-=u_{+\to -}\circ u_+$. Then for $\sigma\in\{+,-\}$, $u_\sigma$ is a continuous and strictly increasing function from $[0,T^u)$ onto $[0,T^u_\sigma)$. Let $\ulin u=(u_+,u_-)$. Then for $0\le t<T^u$, $\Lambda(\ulin u(t))=0$ and $\Upsilon(\ulin u(t))=t$. So $\ulin u$ satisfies the desired property on $[0,T^u)$. It cannot be extended beyond $T^u$ while keeping this property because $\sup u_\sigma[0,T^u)=T^u_\sigma$, and  $\Lambda>0$ on ${\cal D}\cap [T^u_+,\infty)\times \R_+$, and $\Lambda<0$ on ${\cal D}\cap \R_+\times [T^u_-,\infty)$.
\end{proof}

\begin{Lemma}
	For any $t\in [0,T^u)$ and $\sigma\in\{+,-\}$,
\BGE e^{2t}|v_+-v_-|/128\le   \rad_{v_0} (\eta_\sigma[0,u_\sigma(t)]\cup[v_0,v_\sigma])\le  e^{2t} |v_+-v_-| .\label{V-V'}\EDE
If $T^u<\infty$, then $\lim_{t\uparrow T^u} \ulin u(t)$ {is} a point in $\pa{\cal D}\cap (0,\infty)^2$. If ${\cal D}=\R_+^2$, then $T^u=\infty$. If $T^u=\infty$, then $\diam(\eta_+)=\diam(\eta_-)=\infty$.
	\label{Beurling}
\end{Lemma}
\begin{proof}
Fix $t\in[0,T^u)$. For $\sigma\in\{+,.-\}$, let $S_\sigma= [v_0,v_\sigma]\cup\eta_\sigma[0,u_\sigma(t)]\cup \lin{\eta_\sigma[0,u_\sigma(t)]}$, where the bar stands for complex conjugation, and  $L_\sigma =\rad_{v_0} (S_\sigma)$.  From (\ref{V-V}) and that $|V_+(\ulin u(t))-V_-(\ulin u(t))|=e^{2t}|v_+-v_-|$, we get $ e^{2t}|v_+-v_-|/8\le L_+\vee L_-\le e^{2t}|v_+-v_-|$. Since $V_+(\ulin u(t))-V_0(\ulin u(t))=V_0(\ulin u(t))-V_-(\ulin u(t))$, by Lemma \ref{lem-V0}, $S_+$ and $S_-$ have the same harmonic measure viewed from $\infty$. By Beurling's estimate,  $ L_+\vee L_- \le 16 (L_+\wedge L_-)$. So we get (\ref{V-V'}).
For any $\sigma\in\{+,-\}$, $u_\sigma(t)=\hcap_2(\eta_\sigma[0,u_\sigma(t)])\le L_\sigma^2\le e^{4t} |v_+-v_-|^2$. Suppose $T^u<\infty$. Then $u_+$ and $u_-$ are bounded on $[0,T^u)$. Since $\ulin u$ is increasing, $\lim_{t\uparrow T^u} \ulin u(t)$ {is} a point in $(0,\infty)^2$, which must lie on $\pa \cal D$ because   $\ulin u$ cannot be extended beyond $T^u$. If ${\cal D}=\R_+^2$, then $\pa{\cal D}\cap(0,\infty)^2=\emptyset$, and so $T^u=\infty$. If $T^u=\infty$, then by (\ref{V-V'}),  $\diam(\eta_\pm)=\infty$.
\end{proof}

{For any function $X$ on $\cal D$,} define
 $X^u=X\circ \ulin u$ {on $[0,T^u)$}. Let $I=|v_+-v_0|=|v_--v_0|$.
 From the definition of $\ulin u$, we have $|V^u_\pm(t)-V^u_0(t)|=e^{2t}I$ for any $t\in[0,T^u)$. Let
{ $$R_\sigma =\frac{W_\sigma^u -V_0^u }{V_\sigma^u -V_0^u }\in [0,1],\quad \sigma\in\{+,-\};\quad \ulin R=(R_+,R_-).$$ } Let $e^{c\cdot}$ denote the function $t\mapsto e^{ct}$ for $c\in\R$.

\begin{Lemma}
Let ${\cal D}_{\disj}$ be defined by (\ref{D-disj}). Let $T^u_{\disj}\in(0,T^u]$ be such that $\ulin u(t)\in{\cal D}_{\disj}$ for $0\le t<T^u_{\disj}$. Then $\ulin u$ is continuously differentiable on $[0,T^u_{\disj})$, and
 \BGE (W^u_{\sigma,1})^2 u_\sigma'= \frac{R_\sigma(1-R_\sigma^2)}{R_++R_-}\,e^{4\cdot } I^2 \mbox{ on }[0,T^u_{\disj}),\quad \sigma\in\{+,-\}.  \label{uj''}\EDE
  \label{lem-uj}
\end{Lemma}
\begin{proof}
  By   (\ref{pa-X}), $\Lambda$ and $\Upsilon$  satisfy the following differential equations on ${\cal D}_{\disj}$:
$$  \pa_{\sigma} \Lambda \aeq\frac{(V_+-V_-) W_{\sigma,1}^2}{\prod_{\nu\in\{0,+,-\}} (V_\nu^u -W_\sigma^u )}\,\mbox{ and }\,
 \pa_{\sigma}\Upsilon\aeq \frac{-  W_{\sigma,1}^2}{\prod_{\nu\in\{+,-\}} (V_\nu^u -W_\sigma^u )},\quad \sigma\in\{+,-\}.$$
From  $\Lambda^u(t)= 0$ and $\Upsilon^u(t)=t$,  we get
$$\sum_{\sigma\in\{+,-\}}\frac{ (W_{\sigma,1}^u) ^2u_\sigma '}{\prod_{\nu\in\{0,+,-\}} (V_\nu^u -W_\sigma^u )}\aeq 0\,\mbox{ and }\,
\sum_{\sigma\in\{+,-\}} \frac{-  (W_{\sigma,1}^u)^2 u_\sigma '}{\prod_{\nu\in\{+,-\}} (V_\nu^u -W_\sigma^u )}\aeq 1.$$
Solving the system of equations, we get $(W^u_{\sigma,1})^2 u_\sigma'\aeq (\prod_{\nu\in\{0,+,-\}} (V_\nu^u -W_\sigma^u ))/(W_\sigma-W_{-\sigma})$, $\sigma\in \{+,-\}$. Using $V_\sigma^u-V_0^u=\sigma e^{2\cdot} I$ and $W^u_\sigma-V^u_0=R_\sigma (V_\sigma^u-V_0^u)$, we find that (\ref{uj''}) holds with ``$\aeq$'' in place of ``$=$''. Since $W_+>W_-$ on ${\cal D}_{\disj}$, we get $R_++R_->0$ on $[0,T^u_{\disj})$. So the original (\ref{uj''}) holds by the continuity of its RHS.
\end{proof}

Now suppose that $\eta_+$ and $\eta_-$ are random curves, and $\cal D$ is a random region. Then $\ulin u$ and $T^u$ are also random. Suppose that there is an $\R_+^2$-indexed filtration $\F$ such that $\cal D$ is an $\F$-stopping region, and $V_0,V_+,V_-$ are all $\F$-adapted. Now we extend $\ulin u$ to $\R_+$ such that if $T^u<\infty$, then $\ulin u(s)=\lim_{t\uparrow T^u} \ulin u(t)$ for $s\in [T^u,\infty)$. The following proposition is similar to \cite[Lemma 4.1]{Two-Green-interior}.

\begin{Proposition}
	For every $t\in\R_+$,  $\ulin u(t)$ is an $\F$-stopping time.\label{Prop-u(t)}
\end{Proposition}

Since $\ulin u$ is non-decreasing, we get an $\R_+$-indexed filtration $\F^u$: $\F^u_t= \F _{\ulin u(t)}$, ${t\ge 0}$, by Propositions \ref{T<S} and \ref{Prop-u(t)}.

\section{Commuting Pair{s} of SLE$_\kappa(2,{\protect\ulin{\rho}})$ Curves}\label{section-commuting-SLE-kappa-rho}
In this section, we apply the results from the previous section to study a pair of commuting SLE$_\kappa(2, \ulin\rho)$ curves, which arise as flow lines {($\kappa\ne 4$) or level lines (for $\kappa=4$)} of a GFF with piecewise constant boundary data (cf.\ {\cite{MS1,WW-level}}).
The results of this section will be used in the next section to study {the} commuting pair{s} of hSLE$_\kappa$ curves that we are mostly interested in.

\subsection{Martingale and domain Markov property} \label{subsection-commuting-SLE-kappa-rho}
Throughout this section, we fix $\kappa,\rho_0,\rho_+,\rho_-$ such that $\kappa\in(0,8)$, $\rho_+,\rho_->\max\{-2,\frac \kappa 2-4\}$, $\rho_0\ge \frac{\kappa}{4}-2$, and $\rho_0+\rho_\sigma\ge \frac{\kappa}{2}-4$, $\sigma\in\{+,-\}$. Let $w_-<w_+\in\R$. Let $v_+\in[w_+^+,\infty)$, $v_-\in(-\infty,w_-^-]$, and $v_0\in[w_-^+,w_+^-]$. Write $\ulin\rho$ for $(\rho_0,\rho_+,\rho_-)$.
 From \cite{MS1} {(for $\kappa\ne 4$) and \cite{WW-level} (for $\kappa=4$)}, we know that there is a coupling of two chordal Loewner curves $\eta_+(t_+)$, $0\le t_+<\infty$, and $\eta_-(t_-)$, $0\le t_-<\infty$, driven by $\ha w_+$ and $\ha w_-$ (with speed $1$), respectively, with the following properties.
\begin{enumerate}
  \item [(A)] For $\sigma\in\{+,-\}$, $\eta_\sigma$ is a chordal SLE$_\kappa(2,\ulin\rho)$ curve in $\HH$ started from $w_\sigma$ with force points at $w_{-\sigma}$ and $v_\nu$, $\nu\in\{0,+,-\}$. 
      Let $\ha w_{-\sigma}^\sigma,\ha v_\nu^\sigma $ denote the force point functions for $\eta_\sigma$ started from $w_{-\sigma},v_\nu$, $\nu\in\{0,+,-\}$, respectively.
  \item [(B)]  Let $\sigma\in\{+,-\}$. If $\tau_{-\sigma}$ is a finite stopping time w.r.t.\ the
  filtration $\F^{-\sigma}$ generated by $\eta_{-\sigma}$,  then a.s.\ there is a chordal Loewner curve $\eta_{\sigma}^{t_{-\sigma}}(t)$, $0\le t<\infty$, with some speed such that  $\eta_\sigma =f_{K_{-\sigma}(\tau_{-\sigma})}\circ \eta_{\sigma}^{\tau_{-\sigma}} $. Moreover, the conditional law of the normalization of $\eta_{\sigma}^{\tau_{-\sigma}}$ given $\F^{-\sigma}_{\tau_{-\sigma}}$ is that of a  chordal SLE$_\kappa(2,\ulin\rho)$ curve   in $\HH$ started from $\ha w_\sigma^{-\sigma}(\tau_{-\sigma})$   with force points at $\ha w_{-\sigma}(\tau_{-\sigma}),\ha v_\nu^{-\sigma}(\tau_{-\sigma})$, $\nu\in\{0,+,-\}$.
\end{enumerate}

{There are some tiny flaws in the above two properties, which will be described and fixed as follows. First, since $\eta_\sigma$ starts from $w_\sigma$, its force points must take values in $\R_{w_\sigma}$. However, some $v_\nu$ may take value $w_{-\sigma}^+$ or $w_{-\sigma}^-$, which does not belong to $\R_{w_\sigma}$. When this happens, as a force point for $\eta_\sigma$, $v_\nu$ is treated as $w_{-\sigma}$. Second, it may happen that $\ha v_\nu^{-\sigma}(\tau_{-\sigma})=\ha w_\sigma^{-\sigma}(\tau_{-\sigma})$ for some $\nu$. When this happens, as a force point for the $\eta_{\sigma}^{\tau_{-\sigma}}$ (started from $\ha w_\sigma^{-\sigma}(\tau_{-\sigma})$), $\ha v_\nu^{-\sigma}(\tau_{-\sigma})$ is treated as $\ha w_\sigma^{-\sigma}(\tau_{-\sigma})^\mu$ for some $\mu\in\{+,-\}$, which is chosen such that, if $\nu\in\{+,-\}$, then $\mu=\nu$, and if $\nu=0$, then $\mu=-\sigma$. We choose $\mu$ in this way because $\ha v_-^{-\sigma}\le \ha w_-^{-\sigma}\le \ha v_0^{-\sigma}\le \ha w_+^{-\sigma}\le \ha v_+^{-\sigma}$.
}

One may construct $\eta_+$ and $\eta_-$ as two flow lines of the same GFF on $\HH$ with some piecewise boundary conditions (cf.\ \cite{MS1}). The conditions that $\kappa\in(0,8)$, $\rho_0,\rho_+,\rho_->\max\{-2,\frac \kappa 2-4\}$ and $\rho_0+\rho_\sigma\ge \frac{\kappa}{2}-4$, $\sigma\in\{+,-\}$, ensure that (i) there is no continuation threshold for either $\eta_+$ or $\eta_-$, and so $\eta_+$ and $\eta_-$ both have lifetime $\infty$ and  $\eta_\pm(t)\to\infty$ as $t\to\infty$; (ii) $\eta_+$ does not hit $(-\infty,w_-]$, and $\eta_-$ does not hit $[w_+,\infty)$; and (iii) $\eta_\pm\cap\R$ has Lebesgue measure zero. The stronger condition that $\rho_0\ge \frac{\kappa}{4}-2$ (which implies that $\rho_0>\max\{-2,\frac \kappa 2-4\}$) will be needed later (see Remark \ref{Remark-rho0}).
We call the above $(\eta_+,\eta_-)$ a commuting pair of chordal SLE$_\kappa(2,\ulin\rho)$ curves in $\HH$ started from $(w_+,w_-;v_0,v_+,v_-)$. If $\rho_0=0$, which satisfies $\rho_0>\frac \kappa 4-2$ since $\kappa<8$, the $v_0$ does not play a role, and we omit $\rho_0$ and $v_0$ in the name.

We may take $\tau_{-\sigma}$ in (B) to be a deterministic time. So for each $t_{-\sigma}\in\Q_+$, a.s.\ there is an SLE$_\kappa$-type curve $\eta_{\sigma}^{t_{-\sigma}}$ defined on $\R_+$ such that $\eta_\sigma=f_{K_{{-\sigma}}(t_{-\sigma})}\circ \eta_{\sigma}^{t_{-\sigma}}$. The conditions on $\kappa$ and $\ulin\rho$ implies that a.s.\ the Lebesgue measure of  $\eta_{\sigma}^{t_{-\sigma}}\cap \R$ is $0$. This implies that a.s.\ $\eta_+$ and $\eta_-$ satisfy the conditions in Definition \ref{commuting-Loewner} with $\I_+=\I_-=\R_+$, $\I_+^*=\I_-^*=\Q_+$, and ${\cal D}=\R_+^2$. So $(\eta_+,\eta_-)$ is a.s.\ a commuting pair of chordal Loewner curves. Here we omit $\cal D$ when it is $\R_+^2$. Let $K$ and $\mA$ be the hull function and the capacity function, $W_+,W_-$ be the driving functions, and $V_0,V_+,V_-$ be the force point functions started from $v_0,v_+,v_-$, respectively. Then $\ha w_\sigma=W_\sigma|^{-\sigma}_0$, $\ha w_{-\sigma}^\sigma=W_{-\sigma}|^{-\sigma}_0$, and $\ha v_\nu^\sigma=V_\nu|^{-\sigma}_0$, $\nu\in\{0,+,-\}$. For each $\F^{-\sigma}$-stopping time $\tau_{-\sigma}$, $\eta_{\sigma}^{\tau_{-\sigma}}$ is the chordal Loewner curve driven by $W_\sigma|^{-\sigma}_{\tau_{-\sigma}}$ with speed $\mA|^{-\sigma}_{\tau_{-\sigma}}$, and the force point functions are $W_{-\sigma}|^{-\sigma}_{\tau_{-\sigma}}$ and $V_\nu|^{-\sigma}_{\tau_{-\sigma}}$, $\nu\in\{0,+,-\}$.

Let $\F^\pm$ be the $\R_+$-indexed filtration as in (B). Let $\F$ be the separable ${\R_+^2}$-indexed filtration generated by $\F^+$ and $\F^-$. From (A) we know that, for $\sigma\in\{+,-\}$, there exist standard $\F^\sigma$-Brownian motions $B_\sigma$  such that the driving functions $\ha w_\sigma$ satisfies the SDE
\BGE d\ha w_\sigma\aeq \sqrt{\kappa}dB_\sigma +\Big[\frac 2{\ha w_\sigma-\ha w_{-\sigma}^\sigma}+\sum_{\nu\in\{0,+,-\}} \frac{\rho_{\nu}}{\ha w_\sigma-\ha v_\nu^\sigma}\Big]dt_\sigma.\label{dhaw}\EDE

\begin{Lemma}[Two-curve DMP] Let $\cal G$ be a $\sigma$-algebra. Let ${\cal D}=\R_+^2$. Suppose that, conditionally on $\cal G$, $(\eta_+,\eta_-;{\cal D})$ is a commuting pair of  chordal SLE$_\kappa(2,\ulin\rho)$ curves started from $(w_+,w_-;v_0,v_+,v_-)$, which are ${\cal G}$-measurable random points. Let $K,W_\sigma,V_\nu$, $\sigma\in\{+,-\}$, $\nu\in\{0,+,-\}$, be respectively the hull function, driving functions, and force point functions. {For $\sigma\in\{+,-\}$, l}et $\F^\sigma{=(\F^\sigma_t)_{t\ge 0}}$ be the $\R_+$-indexed filtration {such that for $t\ge 0$,} $\F^\sigma_t$ {is the $\sigma$-algebra generated by} ${\cal G}$ {and} $\eta_\sigma(s))$, $0\le s\le t$. Let $\lin\F$ be the right-continuous augmentation of the separable $\R_+^2$-indexed filtration generated by $\F^+$ and $\F^-$. Let $\ulin\tau$ be an $\lin\F$-stopping time.  Then  on the event $E_{\ulin\tau}:=\{\ulin\tau\in\R_+^2,\eta_\sigma(\tau_\sigma)\not\in \eta_{-\sigma}[0,\tau_{-\sigma}],\sigma\in\{+,-\}\}$, there is a random commuting pair of chordal Loewner curves $(\til\eta_1,\til\eta_2;\til{\cal D})$  with some speeds, which is the part of $(\eta_+,\eta_-;{\cal D})$ after $\ulin\tau$ up to a conformal map (Definition \ref{Def-speeds}), and whose normalization conditionally on $\lin\F_{\ulin\tau}\cap E_{\ulin\tau}$ has the law of a commuting pair of  chordal SLE$_\kappa(2,\ulin\rho)$ curves started from $(W_+ ,W_- ;V_0 ,V_+ ,V_-)|_{\ulin\tau}$. Here if $V_\nu(\ulin\tau)=W_\sigma(\ulin\tau)$ for some $\sigma\in \{+,-\}$ and $\nu\in\{0,+,-\}$, then as a force point $V_\nu(\ulin\tau)$ is treated as $W_\sigma(\ulin\tau)^{\sign(v_\nu-w_\sigma)}$.
\label{DMP}
\end{Lemma}
\begin{proof}
  This lemma is similar to \cite[Lemma A.5]{Green-cut}, which is about the two-directional DMP of chordal SLE$_\kappa$ for $\kappa\le 8$. The argument also works here. See \cite[Remark A.4]{Green-cut}.
\end{proof}

\begin{Remark}
Here is an intuition why Lemma \ref{DMP} is true. By Properties (A) and (B), it is easy to see that Lemma \ref{DMP} holds if the stopping time $\ulin\tau$ has the form $(\tau_+,0)$ or $(0,\tau_-)$, which means that we only grow one curve up to a stopping time. Applying this argument for a second time, we see that the lemma holds if $\ulin\tau=(\tau_+,\tau_-)$ is such that $\tau_+$ is an $({\cal G},\F^+)$-stopping time, and $\tau_-$ is a stopping time w.r.t.\ the filtration generated by $\cal G$, $\F^+_{\tau_+}$, and $\F^-$, which means that we first grow $\eta_+$ up to some stopping time, and then grow $\eta_-$ up to some stopping time depending on the part of $\eta_+$ that has grown. We may further alternatively grow $\eta_+$ and $\eta_-$ up to stopping times depending on previously grown parts, and conclude that the lemma holds for the $\ulin\tau$ constructed in this way. However, not all $\lin\F$-stopping time can be constructed by this way. To deal with the general case, an approximation argument was used in \cite{Green-cut}.
\end{Remark}

\subsection{Relation with the independent coupling}\label{section-indep}
Write $\ulin w$ and $\ulin v$ respectively for $(w_+,w_-)$ and $(v_0,v_+,v_-)$.
Let $\PP^{\ulin\rho}_{\ulin w;\ulin v}$ denote the joint law of the driving functions $(\ha w_+,\ha w_-)$ of a commuting pair of chordal SLE$_\kappa(2,\ulin\rho)$ curves in $\HH$ started from $(\ulin w;\ulin v)$. Now we fix $\ulin w$ and $\ulin v$, and omit the subscript in the joint law.

The  $\PP^{\ulin\rho}_{{\ulin w;\ulin v}}$ is a probability measure on $\Sigma^2$, where $\Sigma:=\bigcup_{0<T\le\infty} C([0,T),\R)$ was defined in \cite[Section 2]{decomposition}.
A random element in $\Sigma$ is   a continuous stochastic process with random lifetime. The space $\Sigma^2$ is  equipped with an  ${\R_+^2}$-indexed filtration $\F$ defined by $\F_{(t_+,t_-)}=\F^+_{t_+}\vee \F^-_{t_-}$, where  $\F^+$ and $\F^-$ are $\R_+$-indexed filtrations generated by the first function and the second function, respectively.

Let $\PP^{\ulin\rho}_+$ and $\PP^{\ulin\rho}_-$ respectively denote the first and second marginal laws of $\PP^{\ulin\rho}$ on $\Sigma$. Then $\PP^{\ulin\rho}$  is different from the product measure $\PP^{\ulin\rho}_i:=\PP^{\ulin\rho}_+\times \PP^{\ulin\rho}_-$. We will derive some relation between $\PP^{\ulin\rho}$ and $\PP^{\ulin\rho}_i$. Suppose now that $(\ha w_+,\ha w_-)$ follows the law $\PP^{\ulin\rho}_i$ instead of $\PP^{\ulin\rho}$. Then  (\ref{dhaw}) holds for two independent Brownian motions $B_+$ and $B_-$, and $\eta_+$ and $\eta_-$ are independent. Define ${\cal D}_{\disj}$ by (\ref{D-disj}). Then $(\eta_+,\eta_-;{\cal D}_{\disj})$ is a disjoint commuting pair of chordal Loewner curves.  Since $B_+$ and $B_-$ are independent, for any $\sigma\in\{+,-\}$, $B_\sigma$ is a Brownian motion w.r.t.\ the filtration $(\F^\sigma_{t}\vee \F^{-\sigma}_{\infty})_{t\ge 0}$, and we may view (\ref{dhaw}) as an $(\F^\sigma_{t_\sigma}\vee \F^{-\sigma}_{\infty})_{t_\sigma\ge 0}$-adapted SDE. We will repeatedly apply It\^o's formula (cf.\ \cite{RY})  in this subsection, where $\sigma\in\{+,-\}$, the variable $t_{-\sigma}$ of every function is a fixed finite  $\F^{-\sigma}$-stopping time $t_{-\sigma}$ unless it is set to be zero using $|^{-\sigma}_0$, and all SDE are  $(\F^\sigma_{t_\sigma}\vee \F^{-\sigma}_{\infty})_{t_\sigma\ge 0}$-adapted in $t_\sigma$.

By (\ref{-3}) we get the SDE for $W_\sigma$:
\BGE \pa_\sigma W_\sigma=W_{\sigma,1} \pa \ha w_\sigma+\Big(\frac{\kappa}{2}-3\Big) W_{\sigma,2}\pa t_\sigma .\label{paWsigma}\EDE
We will use the boundary scaling exponent $\bb$ and central charge $\cc$ in the literature defined by 
$ \bb=\frac{6-\kappa}{2\kappa}$ and $\cc=\frac{(3\kappa-8)(6-\kappa)}{2\kappa}$.
By (\ref{1/2-4/3}) we get the SDE for $W_{\sigma,N}^{\bb}$:
\BGE \frac{\pa_\sigma W_{\sigma,N}^{\bb}}{W_{\sigma,N}^{\bb}}=\bb \frac{W_{\sigma,2}}{W_{\sigma,1}}  \pa \ha w_\sigma+ \frac{\cc}6 W_{\sigma,S}\pa t_\sigma . \label{paWsigmaN}\EDE

Recall the $E_{X,Y}$ defined in (\ref{RXY}). For $Y\in \{W_{-\sigma},V_0,V_+,V_-\}$, $E_{W_\sigma,Y}(t_+,t_-)$ equals a function in $t_{-\sigma}$ times $f(\ulin t,W_\sigma(t_\sigma \ulin e_\sigma), Y(t_\sigma \ulin e_\sigma))$, where
\BGE f(\ulin t,w,y):=\left\{
\begin{array}{ll}
(g_{K_{-\sigma}^{t_\sigma}(t_{-\sigma})}(w)-g_{K_{-\sigma}^{t_\sigma}(t_{-\sigma})}(y))/(w-y), &w\ne y;\\
g_{K_{-\sigma}^{t_\sigma}(t_{-\sigma})}'(w), & w=y.
\end{array}\right.
\label{f(t,w,y)}\EDE
 Using  (\ref{pa-X},\ref{paWsigma}) and (\ref{WV+g}-\ref{V0-g}) we see that $E_{W_\sigma,Y}$  satisfies the SDE
$$\frac{\pa_\sigma E_{W_\sigma,Y}}{E_{W_\sigma,Y}}\aeq \Big[\frac{W_{\sigma,1}}{W_{\sigma}-Y}-\frac{W_{\sigma,1}}{W_{\sigma}-Y}\Big|^{-\sigma}_0 \Big]d\ha w_\sigma +\Big[\frac{2W_{\sigma,1}^2}{(W_\sigma-Y)^2}-\frac{2W_{\sigma,1}^2}{(W_\sigma-Y)^2} \Big|^{-\sigma}_0\Big]\pa t_\sigma$$
\BGE -\frac{\kappa }{W_\sigma-Y}\Big|^{-\sigma}_0 \cdot  \Big[\frac{W_{\sigma,1}}{W_{\sigma }-Y}- \frac{W_{\sigma,1}}{W_{\sigma }-Y}\Big|^{-\sigma}_0 \Big]\pa t_\sigma+\Big(\frac\kappa 2-3\Big)\frac{W_{\sigma,2}}{W_\sigma-Y}\pa t_\sigma.\label{paEWY}\EDE

Recall the $Q$ defined in (\ref{F}).
Define a positive continuous function $M$ on ${\cal D}_{\disj}$ by
\BGE  M=Q^{-\frac{\cc}6}  E_{W_+,W_-}^{\frac 2\kappa}  \prod_{\nu_1<\nu_2\in\{0,+,-\}} E_{V_{\nu_1},V_{\nu_2}}^{\frac{\rho_{\nu_1}\rho_{\nu_2}}{2\kappa}}  \prod_{\sigma\in\{+,-\}} \Big[ W_{\sigma,N}^{\bb}  \prod_{\nu\in\{0,+,-\}}  E_{W_\sigma,V_\nu}^{\frac{\rho_\nu}\kappa} V_{\nu,N}^{\frac{\rho_\nu(\rho_\nu+4-\kappa)}{4\kappa}}\Big]. \label{Mirho-crho}\EDE
Then $M(t_+,t_-)=1$ if {$t_+=0$ or $t_-=0$}. Combining (\ref{paF},\ref{pajhaW},\ref{paV1},\ref{paEXY},\ref{dhaw},\ref{paWsigmaN},\ref{paEWY}) and using the facts that $\ha w_\sigma=W_\sigma|^{-\sigma}_0$, $\ha w_{-\sigma}^\sigma=W_{-\sigma}|^{-\sigma}_0$ and $\ha v_\nu^\sigma=V_\nu|^{-\sigma}_0$, we get the SDE for $M$ in $t_\sigma$ 
:
$$\frac{\pa_\sigma M}{M}=\bb \frac{W_{\sigma,2}}{W_{\sigma,1}}\pa B_\sigma-\Big[\frac 2{\ha w_\sigma-\ha w^\sigma_{-\sigma}}+\sum_{\nu\in\{0,+,-\}} \frac{\rho_{\nu}}{\ha w_\sigma-\ha v_\nu^\sigma}\Big] \frac{\pa B_\sigma}{\sqrt\kappa}+$$
\BGE+ \Big[\frac {2W_{\sigma,1} }{W_\sigma-W_{-\sigma}}+\sum_{\nu\in\{0,+,-\}} \frac{\rho_{\nu}W_{\sigma,1} }{W_\sigma-V_\nu}\Big]\frac{\pa B_\sigma}{\sqrt\kappa} .\label{paM}\EDE
This means that $M|^{-\sigma}_{t_{-\sigma}}$ is a local martingale in $t_\sigma$.

For $\sigma\in\{+,-\}$, let $\Xi_\sigma$ denote the space of simple crosscuts of $\HH$ that separate $w_\sigma$ from $w_{-\sigma} $ and $\infty$. Note that the crosscuts also separate $w_\sigma$ from $v_{-\sigma}$ since $v_{-\sigma}$ is further away from $w_\sigma$ than $w_{-\sigma}$.
But  the crosscuts {may not} separate $w_\sigma$ from $v_\sigma$ or $v_0$. For $\sigma\in\{+,-\}$ and $\xi_\sigma\in\Xi_\sigma$, let $\tau^\sigma_{\xi_\sigma}$ be the first time that $\eta_\sigma$ hits the closure of $\xi_\sigma$; or $\infty$ if such time does not exist. We see that $\tau^\sigma_{\xi_\sigma}\le \hcap_2(\xi_j)<\infty$.
Let $\Xi=\{(\xi_+,\xi_-)\in\Xi_+\times\Xi_-,\dist(\xi_+,\xi_-)>0\}$. For $\ulin\xi=(\xi_+,\xi_-)\in\Xi$, let $\tau_{\ulin\xi}=(\tau^+_{\xi_+},\tau^-_{\xi_-})$.  {Let $\Xi^*$ be the set of $(\xi_+,\xi_-)\in\Xi$ such that $\xi_+$ and $\xi_-$ are polygonal curves whose vertices have rational coordinates.}
{Then $\Xi^*$ is a countable subset of} $\Xi$ such that for every $\ulin\xi=(\xi_+,\xi_-)\in\Xi$ there is $(\xi_+^*,\xi_-^*)\in\Xi^*$ such that $\xi_\sigma$ is enclosed by $\xi_\sigma^*$, $\sigma\in\{+,-\}$. See Figure \ref{fig-crosscut}.


\begin{figure}
	\begin{center}
		\includegraphics[width=5in]{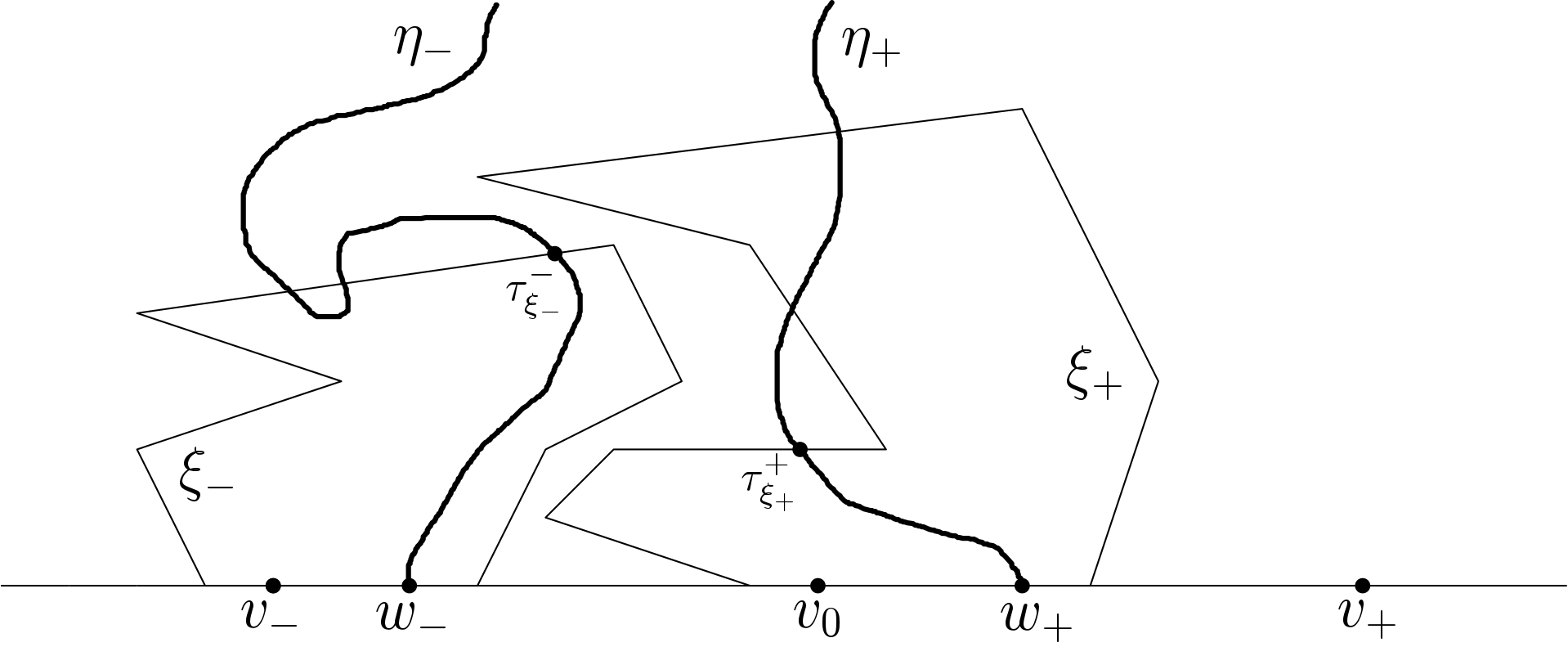}
	\end{center}
	\caption{{ The figure above illustrates an element $(\xi_+,\xi_-)\in \Xi^*$ and the corresponding stopping times $\tau^+_{\xi_+}$ and $\tau^-_{\xi_-}$ for the curves $\eta_+$ and $\eta_-$.}} \label{fig-crosscut}
\end{figure}

\begin{Lemma}
For any $\ulin\xi\in\Xi$ and $R>0$, there is a constant $C>0$ depending only on $\kappa,\ulin\rho, \ulin\xi, R$, such that if $|v_+-v_-|\le R$, then $|\log M|\le C$ on $[\ulin 0, \tau_{\ulin\xi}]$. \label{uniform}
\end{Lemma}
\begin{proof}
	Fix $\ulin\xi=(\xi_+,\xi_-)\in\Xi$ and $R>0$. Suppose $|v_+-v_-|\le R$. Throughout the proof, a constant is a number that depends only on $\ulin\xi, R$; and  a function defined on $[\ulin 0, \tau_{\ulin\xi}]$  is said to be uniformly bounded if its absolute value on $[\ulin 0, \tau_{\ulin\xi}]$ is bounded above by a constant. By the definition of $M$, it suffices to prove that $|\log Q|$, $|\log E_{Y_1,Y_2}|$, $Y_1\ne Y_2\in\{W_+,W_-,V_0,V_+,V_-\}$, $|\log W_{\sigma,N}|$, $\sigma\in\{+,-\}$, and $|\log V_{\nu,N}|$, $\nu\in\{0,+,-\}$, are all uniformly bounded.

 Let $K_{\xi_\sigma}=\Hull(\xi_\sigma)$, $\sigma\in\{+,-\}$ and	 $K_{\ulin \xi}=K_{\xi_+}\cup K_{\xi_-}$.  Let $I=(\max(\lin\xi_-\cap\R),\min(\lin\xi_+\cap\R))$. Then $|g_{K_{\ulin\xi}}(I)|$ is a positive constant. By symmetry we assume that either $v_0\in\lin{K_{\xi_-}}$ or $v_0\in I$ and $g_{K_{\ulin\xi}}(v_0)$ is no more than the middle of $g_{K_{\ulin\xi}}(I)$. So we may pick $v_0^1<v_0^2\in I$ with $v_0\le v_0^1$ such that $|g_{K_{\ulin\xi}}(v_0^2)-g_{K_{\ulin\xi}}(v_0^1)|\ge |g_{K_{\ulin\xi}}(I)|/3$.  Let $V_0^j$ be the force point function started from $v_0^j$, $j=1,2$. By (\ref{VWVWV}), $V_+\ge W_+\ge V_0^2>V_0^1\ge V_0\ge W_-\ge V_-$ on $[\ulin 0,\tau_{\ulin\xi}]$.

By Proposition \ref{Prop-contraction}, $W_{+,1},W_{-,1}$  are uniformly bounded   by $1$.
For $\sigma\in\{+,-\}$, the function $(t_+,t_-)\mapsto t_\sigma$ is bounded on $[\ulin 0, \tau_{\ulin\xi}]$   by $\hcap_2(K_{\ulin\xi})$.
For any $\ulin t\in [\ulin 0, \tau_{\ulin\xi}]$,  since  $g_{K_{\ulin\xi}}=g_{K_{\ulin\xi}/K(\ulin t)}\circ g_{K(\ulin t)}$, by Proposition \ref{Prop-contraction} we get $0<g_{K_{\ulin\xi}}'\le g_{K(\ulin t)}'\le 1$ on $[v_0^1,v_0^2]$.
Since $V_0^j(\ulin t)=g_{K(\ulin t)}(v_0^j)$, $j=1,2$, we have
$|V_0^2(\ulin t)-V_0^1(\ulin t)|\ge |g_{K_{\ulin\xi}}(v_0^2)-g_{K_{\ulin\xi}}(v_0^1)|\ge |g_{K_{\ulin\xi}}(I)|/3$.
So  $\frac 1{V_0^2-V_0^1}$ is uniformly bounded, which then implies that $\frac1{|W_\sigma-W_{-\sigma}|}$ and $\frac1{|W_\sigma-V_{-\sigma}|}$ are uniformly bounded, $\sigma\in\{+,-\}$. From (\ref{F}) we see that $|\log Q|$ is uniformly bounded.  From (\ref{pajhaW},\ref{paV1}) and the fact that $W_{-\sigma,N}|^\sigma_0=V_{-\sigma,N}|^\sigma_0=1$, we see that $|\log W_{-\sigma,N}|$ and $|\log V_{-\sigma,N}|$, $\sigma\in\{+,-\}$,  are uniformly bounded. We also know that $\frac1{|W_+-V_0|}\le \frac 1{|V_0^2-V_0^{1}|}$ is uniformly bounded. From (\ref{paV1}) with $\nu=0$ and $\sigma=+$ and the fact that $ V_{0,N}|^+_0\equiv 1$ we find that $|\log V_{0,N}|$ is uniformly bounded.
	
	Now we estimate $|\log E_{Y_1,Y_2}|$. By (\ref{V-V}), $|V_+-V_-|$ is uniformly bounded. Thus, for any $Y_1\ne Y_2\in \{W_+,W_-,V_0,V_+,V_-\}$, $|Y_1-Y_2|\le |V_+-V_-|$ is uniformly bounded. If $Y_1\in\{W_+,V_+\}$ and $Y_2\in\{W_-,V_-\}$, then $\frac1{|Y_1-Y_2|}\le \frac 1{ |V_0^1-V_0^2|}$ is uniformly bounded. From (\ref{RXY}) we see that $|\log  E_{Y_1,Y_2}|$ is uniformly bounded. If $Y_1,Y_2\in \{W_{-\sigma},V_{-\sigma}\}$ for some $\sigma\in\{+,-\}$, then $\frac 1{|Y_j-W_\sigma|}$, $j=1,2$, are uniformly bounded, and then the uniformly boundedness of $|\log E_{Y_1,Y_2}|$ follows from (\ref{paEXY}) and the fact that $E_{Y_1,Y_2}|^\sigma_0\equiv 1$. Finally, we consider the case that $Y_1=V_0$.   If $Y_2\in\{W_+,V_+\}$, then $\frac 1{|Y_2-V_0|}\le \frac 1{|V_0^2-V_0^1|}$, which is uniformly bounded.  We can again use (\ref{RXY}) to get the uniformly boundedness of  $|\log E_{V_0,Y_2}|$. If $Y_2\in\{W_-,V_-\}$, then $\frac 1{|V_0-W_+|}$ and $\frac 1{|Y_2-W_+|}$  are uniformly bounded. The uniformly boundedness of $|\log E_{V_0,Y_2}|$ then follows from (\ref{paEXY}) with $\sigma=+$, $X=V_0$, $Y=Y_2$, and the fact that $E_{V_0,Y_2}|^+_0\equiv 1$.
\end{proof}

\begin{Corollary}
	For any $\ulin\xi\in\Xi$, $M(\cdot \wedge \tau_{\ulin\xi})$ is an $\F$-martingale closed by $M(\tau_{\ulin\xi})$ w.r.t.\ $\PP^{\ulin\rho}_i$. \label{Doob}
\end{Corollary}
\begin{proof}
	This follows from (\ref{paM}), Lemma \ref{uniform}, and the same argument {used to prove} \cite[Corollary 3.2]{Two-Green-interior}.
\end{proof}

\begin{Lemma}
  For any $\ulin\xi=(\xi_+,\xi_-)\in\Xi$, $\PP^{\ulin\rho}$ is absolutely continuous w.r.t.\ $\PP^{\ulin\rho}_i$ on $\F^1_{\tau^1_{\xi_1}}\vee\F^2_{\tau^2_{\xi_2}}$, and the RN derivative is $M(\tau_{\ulin\xi})$. \label{RN-M-ic}
\end{Lemma}
\begin{proof}
  Let $\ulin\xi=(\xi_+,\xi_-)\in\Xi$. The above corollary implies that  $\EE^{\ulin\rho}_i[M(\tau_{\ulin\xi})]=M(\ulin 0)=1$. So we may define a new probability measure $\PP^{\ulin\rho}_{\ulin \xi}$ by $ {d \PP^{\ulin\rho}_{\ulin \xi}}= {M(\tau_{\ulin\xi})} {d\PP^{\ulin\rho}_i}$.

  Since $M(t_+,t_-)=1$ when $t_+t_-=0$, from the above corollary we know that the marginal laws of $\PP^{\ulin\rho}_{\ulin \xi}$ agree with that of $\PP^{\ulin\rho}_i$, which are $\PP^{\ulin\rho}_+$ and $\PP^{\ulin\rho}_-$. Suppose $(\ha w_+,\ha w_-)$ follows the law $\PP^{\ulin\rho}_{\ulin \xi}$. Then they satisfy Condition (A) in Section \ref{subsection-commuting-SLE-kappa-rho}. Now we write $\tau_\pm$ for $\tau^\pm_{\xi_\pm}$, and $\ulin\tau$ for $\tau_{\ulin\xi}$. Let $\sigma_-\le \tau_-$ be an $\F^-$-stopping time. From Lemma \ref{OST} and Corollary \ref{Doob},
$\frac{d \PP^{\ulin\rho}_{\ulin \xi}|\F _{(t_+,\sigma_-)}}{d\PP^{\ulin\rho}_i|\F _{(t_+,\sigma_-)}}= M(t_+\wedge {\tau_+},{\sigma_-})$,  $0\le t_+<\infty$.
From Girsanov{'s} Theorem and (\ref{dhaw},\ref{paM}), we see that, under $\PP^{\ulin\rho}_{\ulin \xi}$,  $\ha w_+$ satisfies  the following  SDE up to $\tau_+$:
\begin{align*}
d\ha w_+=&\sqrt\kappa d B^{\tau_-}_{+}+\kappa \bb \frac{W_{+,2} }{W_{+,1} }\Big|^-_{\tau_-} d t_++  \frac{2 W_{+,1} }{W_+ -W_{-} }\Big|^-_{\tau_-}\,d t_+   + \sum_{\nu\in\{0,+,-\}} \frac{\rho_\nu W_{+,1} }{W_+ -V_\nu }\Big|^-_{\tau_-}\,d t_+ ,
\end{align*}
where $B^{\tau_-}_{+}$  is a standard $(\F _{(t_+,\sigma_-)})_{t_+\ge 0}$-Brownian motion under $\PP^{\ulin\rho}_{\ulin \xi}$. Using Lemma \ref{W=gw} and (\ref{-3}) we find that $W_+(\cdot, \sigma_-)$ under $\PP^{\ulin\rho}_{\ulin \xi}$ satisfies the following SDE up to $\tau_+$:
\BGE d W_+|^{-}_{\sigma_{-}}   \aeq \sqrt\kappa  W_{+,1} |^{-}_{\sigma_{-}}d B_{+}^{\sigma_{-}}  +  \frac {2W_{+,1}^2}{W_{+}-W_{-}}\bigg |^{-}_{\sigma_{-}}  d t_+  +\sum_{\nu\in \{0,+,-\}} \frac {\rho_\nu W_{+,1} ^2}{W_{+}-V_{\nu} }\bigg|^{-}_{\sigma_{-}} d t_+ .\label{SDE-paW}\EDE
Note that the SDE (\ref{SDE-paW}) agrees with the SDE for $W_+(\cdot,\sigma_-)$ if $(\eta_+,\eta_-)$ is a commuting pair of chordal SLE$_\kappa(2,\ulin\rho)$ curves started from $(\ulin w;\ulin v)$, where the speed is $W_{+,1}(\cdot,\sigma_-)^2$. There is a similar SDE for $W_-(\sigma_+,\cdot)$ if $\sigma_+\le \tau_+$ is an $\F^+$-stopping time. Thus, $ { \PP^{\ulin\rho}_{\ulin \xi}}$ agrees with $\PP^{\ulin\rho}$ on $\F^1_{\tau^1_{\xi_1}}\vee\F^2_{\tau^2_{\xi_2}}$, which implies the conclusion of the lemma.
\end{proof}

\begin{Corollary}
  If $\ulin T$ is an $\F$-stopping time, then $\PP^{\ulin\rho} $ is absolutely continuous w.r.t.\ $\PP^{\ulin\rho}_i $ on $\F_{\ulin T}\cap \{\ulin T\in{\cal D}_{\disj}\}$, and the RN derivative is $M(\ulin T)$. In other words, if $A\in \F_{\ulin T}$ and $A\subset \{\ulin T\in{\cal D}_{\disj}\}$, then $\PP^{\ulin\rho} [A]=\EE^{\ulin\rho}_i [{\bf 1}_A M(\ulin T)]$. \label{RN-M-ic-cor}
\end{Corollary}
\begin{proof}
  Since $\{\ulin T\in{\cal D}_{\disj}\}= \bigcup_{\ulin\xi\in \Xi^*} \{\ulin T <\tau_{\ulin\xi}\}$ and $\Xi^*$ is countable, it suffices to prove the statement with $ \{\ulin T <\tau_{\ulin\xi}\}$ in place of $ \{\ulin T\in{\cal D}_{\disj}\}$ for every $\ulin\xi\in\Xi^*$. Fix   $\ulin\xi=(\xi_+,\xi_-)\in\Xi^*$. We write $\F_{\ulin\xi}$ for $\F^+_{\tau^+_{\xi^+}}\vee \F^-_{\tau^-_{\xi^-}}$. Let $A\in\F_{\ulin T}\cap \{\ulin T <\tau_{\ulin\xi}\}$.   Fix $\ulin t=(t_+,t_-)\in\Q_+^2$. Let $A_{\ulin t}=A\cap \{\ulin T\le \ulin t <\tau_{\ulin\xi}\}$.  For every  $B_+\in\F^+_{t_+}$ and $B_-\in\F^-_{t_-}$, $B_+\cap B_-\cap \{\ulin t<\tau_{\ulin\xi}\}\in \F^+_{\tau^+_{\xi^+}}\vee \F^-_{\tau^-_{\xi^-}}=\F_{\ulin\xi}$. Using a monotone class argument, we conclude that  $\F_{\ulin t}\cap \{\ulin t<\tau_{\ulin\xi}\}\in \F_{\ulin\xi}$. Thus, $A_{\ulin t}\in\F_{\ulin t}\cap \{\ulin t<\tau_{\ulin\xi}\}\subset \F_{\ulin\xi}$. Since $A=\bigcup_{\ulin t\in\Q_+^2} A_{\ulin t}$, we get $A\in\F_{\ulin\xi}$.  By Lemma \ref{RN-M-ic}, Proposition \ref{OST}, and the martingale property of $M(\cdot\wedge \tau_{\ulin\xi})$, we get
  $\EE^{\ulin\rho}[A]=\EE^{\ulin\rho}_i [{\bf 1}_A M(\tau_{\ulin\xi})]=\EE^{\ulin\rho}_i [{\bf 1}_A M(\ulin T\wedge \tau_{\ulin\xi})]=\EE^{\ulin\rho}_i [{\bf 1}_A M(\ulin T)]$.
\end{proof}

{
\begin{Remark}
  We call the $M$ a two-time-parameter local martingale. It plays the same role as the $M$ defined in \cite[Formula (4.3)]{reversibility} and the $M$ defined in \cite[Formula (4.34)]{duality}.
\end{Remark}
}

\subsection{SDE along a time curve up to intersection} \label{section-diffusion}
{Recall that $(\eta_+,\eta_-)$ is a.s.\ a commuting pair of chordal Loewner curves with the time region $\R_+^2$.} Now assume that $v_+-v_0=v_0-v_-$.
 Let $\ulin u=(u_+,u_-):[0,T^u)\to \R_+^2$ be {the function $\ulin u$ developed} in Section \ref{time curve} {for this random pair $(\eta_+,\eta_-)$}. By Lemma \ref{Beurling}, a.s.\ $T^u=\infty$. 
By Proposition \ref{Prop-u(t)}, $\ulin u(t)$ is an $(\F_{ \ulin t})$-stopping time for each $t\ge 0$. We then get an $\R_+$-indexed filtration $\F^u$ by $\F^u_t:=\F_{ \ulin u(t)}$, $t\ge 0$. For $\ulin \xi=(\xi_+,\xi_-)\in\Xi$,  let $\tau^u_{\ulin\xi}$ denote the first $t\ge 0$ such that $u_1(t)=\tau^1_{\xi_1}$ or $u_2(t)=\tau^2_{\xi_2}$, whichever comes first.
Note that such time exists and is finite because $(\tau^1_{\xi_1},\tau^2_{\xi_2})\in\cal D$. The following proposition has the same form as \cite[Lemma 4.2]{Two-Green-interior}{, and its proof is also the same as the proof there}.

\begin{Proposition}
	For every $\ulin\xi\in\Xi$, $\tau^u_{\ulin \xi}$ is an $\F^u$-stopping time, and  $\ulin u(\tau^u_{\ulin \xi})$ and $\ulin u(t\wedge\tau^u_{\ulin \xi})$, $t\ge 0$, are all $\F$-stopping times.   \label{u-st}
\end{Proposition}

Assume that $(\ha w_+,\ha w_-)$ follows the law $\PP^{\ulin\rho}_i$. {This assumption will be used up to Lemma \ref{Lem-uB}.} Let $\eta_\pm$ be the chordal Loewner curve driven by $\ha w_\pm$. Let ${\cal D}_{\disj}$ be as before. Let $\ha w_{-\sigma}^\sigma(t)$ and $\ha v_\nu^\sigma(t)$ be the force point functions for $\eta_\sigma$ started from $w_{-\sigma}$ and $v_\nu$ respectively, $\nu\in\{0,+,-\}$, $\sigma\in\{+,-\}$. Define $\ha B_\sigma$, $\sigma\in\{+,-\}$, by
\BGE \sqrt{\kappa }\ha B_\sigma(t)=\ha w_\sigma(t)-w_\sigma-\int_0^t \frac{2ds}{\ha w_\sigma(s)-\ha w_{-\sigma}^\sigma(s)}-\sum_{\nu\in \{0,+,-\}} \int_0^t \frac{\rho_{\nu}ds}{\ha w_\sigma(s)-\ha v_\nu^\sigma(s)}.\label{hawhaB}\EDE
Then $\ha B_+$ and $\ha B_-$ are independent standard Brownian motions. So we get five $\F$-martingales on ${\cal D}_{\disj}$: $\ha B_+(t_+)$, $\ha B_-(t_-)$, $\ha B_+(t_+)^2-t_+$, $\ha B_-(t_-)^2-t_-$, and $\ha B_+(t_+)\ha B_-(t_-)$. Fix $\ulin\xi\in\Xi$. Using  Propositions \ref{OST}  and \ref{Prop-u(t)} and the fact that $u_\pm$ is uniformly bounded above on $[0,\tau_{\ulin\xi}]$, we conclude that $\ha B^u_\sigma(t \wedge \tau^u_{\ulin \xi})$,  $\ha B^u_\sigma(t \wedge \tau^u_{\ulin \xi})^2- u_\sigma(t\wedge \tau^u_{\ulin \xi})$, $\sigma\in\{+,-\}$, and $\ha B^u_+(t \wedge \tau^u_{\ulin \xi}) \ha B^u_-(t\wedge \tau^u_{\ulin \xi} )$ are all $\F^u$-martingales under $\PP^{\ulin\rho}_i$. Recall that for a function $X$ defined on $\cal D$, we use $X^u$ to denote the function $X\circ \ulin u$ defined on $[0,T^u)$. This rule applies even if $X$ depends only on $t_+$ or $t_-$ (for example, $\ha B^u_\sigma(t)=\ha B_\sigma(u_\sigma(t))$); but does not apply to $\F^u,T^u,T^u_{\disj},\tau^u_{\ulin\xi}$.
Thus, the quadratic variation and covariation of $\ha B^u_+$ and $\ha B^u_-$ satisfy
\BGE \langle \ha B^u_+\rangle_t\aeq u_+(t),\quad \langle \ha B^u_-\rangle_t\aeq u_-(t),\quad \langle \ha B^u_+,\ha B^u_-\rangle_t=0,\label{quadratic}\EDE
up to $\tau^u_{\ulin\xi}$. By Corollary \ref{Doob} and Proposition \ref{OST},  $M^u(\cdot\wedge \tau^u_{\ulin\xi})$  is an $\F^u$-martingale. Let $T^u_{\disj}$ denote the first $t$ such that $\ulin u(t)\not\in{\cal D}_{\disj}$.  Since $T^u_{\disj}=\sup_{\ulin \xi\in\Xi}\tau^u_{\ulin\xi}=\sup_{\ulin \xi\in\Xi^*}\tau^u_{\ulin\xi}$, and $\Xi^*$ is countable,   we see that, $T^u_{\disj}$ is an $\F^u$-stopping time.
We now compute the SDE for $M^u$ up to $T^u_{\disj}$ in terms of $\ha B^u_+$ and $\ha B^u_-$.
Using (\ref{Mirho-crho}) we may express $M^u$ as a product of several factors, among which $E_{W_+,W_-}^u$, $(W_{\sigma,N}^u)^{\bb}$, $(E_{W_\sigma,V_\nu}^u)^{\rho_\nu/\kappa}$, $\sigma\in\{+,-\}$, $\nu\in\{0,+,-\}$,   contribute the local martingale part, and other factors are differentiable in $t$. For $\sigma\in\{+,-\}$, since $W_\sigma(t_+,t_-)=g_{K_{-\sigma}^{t_\sigma}(t_{-\sigma})}(\ha w_\sigma(t_\sigma))$, using (\ref{-3}) 
we get the $\F^u$-adapted SDEs:
\BGE dW_\sigma^u = W_{\sigma,1}^ud\ha w^u_\sigma+\Big(\frac \kappa 2-3\Big)W_{\sigma,2}^u u_\sigma'dt+\frac{2(W_{-\sigma,1}^u)^2u_{-\sigma}'}{W_\sigma^u-W_{-\sigma}^u} \,dt,\label{dWju}\EDE
Since $W_{\sigma,N}=\frac{W_{\sigma,1}}{W_{\sigma,1}|^\sigma_0}$, $W_\sigma^1(t_+,t_-)=g_{K_{-\sigma}^{t_\sigma}(t_{-\sigma})}'(\ha w_\sigma(t_\sigma))$, $W_\sigma^1|^\sigma_0$ is differentiable in $t_{-\sigma}$, and $g_{K_{-\sigma}^{t_\sigma}(t_{-\sigma})}'$ is differentiable in both $t_\sigma$ and $t_{-\sigma}$, we get
the SDE for $(W_{\sigma,N}^u)^{\bb}$:
$$\frac{d(W_{\sigma,N}^u)^{\bb}}{(W_{\sigma,N}^u)^{\bb}}\aeq\bb\frac{W_{\sigma,2}^u}{W_{\sigma,1}^u}\,\sqrt\kappa d \ha B^u_\sigma+\mbox{drift terms}.$$
For the SDE for $(E^u_{W_+,W_-})^{\frac 2\kappa}$, note that when $X=W_+$ and $Y=W_-$, the numerators and denominators in (\ref{RXY}) never vanish. So using (\ref{dWju}) we get
$$\frac{d(E_{W_+,W_-}^u)^{\frac2\kappa}}{ (E_{W_+,W_-}^u)^{\frac{2}\kappa}}=\frac 2\kappa \sum_{\sigma\in\{+,-\}}\Big[\frac{W_{\sigma,1}^u}{W_\sigma^u-W_{-\sigma}^u}-\frac 1{\ha w_\sigma^u-(\ha w_{-\sigma}^\sigma)^u}\Big]\sqrt\kappa d\ha B^u_\sigma +\mbox{drift terms}.$$
Note that  $E_{W_\sigma,V_\nu}^u(t)$ equals $f(\ulin u(t), \ha w^u_\sigma(t),(\ha v_\nu^\sigma)^u(t))$ times a differential function in $u_{-\sigma}(t)$,  where $f(\cdot,\cdot,\cdot)$ is given by (\ref{f(t,w,y)}). Using (\ref{dWju}) we get the SDE for $(E_{W_\sigma,V_\nu}^u)^{\frac{\rho_\nu}\kappa}$:
$$\frac{d(E_{W_\sigma,V_\nu}^u)^{\frac{\rho_\nu}\kappa}}{ (E_{W_\sigma,V_\nu}^u)^{\frac{\rho_\nu}\kappa}}\aeq \frac{\rho_\nu}{\kappa}\Big[\frac{W_{\sigma,1}^u }{W_\sigma^u-V_\nu^u}-\frac 1{\ha w_\sigma^u- (\ha v_\nu^\sigma)^u}\Big] \,\sqrt\kappa d\ha B^u_\sigma+\mbox{drift terms}.$$
Here if at any time $t$, $(\ha v_\nu^\sigma)^u(t)=\ha w_\sigma^u(t)$, then the function inside the square brackets is understood as $\frac 12\frac{W^u_{\sigma,2}(t)}{W^u_{\sigma,1}(t)}$, which is the limit of the function as $(\ha v_\nu^\sigma)^u(t)\to \ha w_\sigma^u(t)$.

Combining the above displayed formulas and using the fact that $M^u$ and $\ha B^u_\pm$ are all $\F^u$-local martingales under $\PP^{\ulin\rho}_i$, we get
\begin{align}
\frac{d M^u}{M^u}\aeq & \sum_{\sigma\in\{+,-\}}\bigg [ \kappa \bb\frac{W_{\sigma,2}^u}{W_{\sigma,1}^u}+2 \Big[\frac{ W_{\sigma,1}^u}{W_\sigma^u-W_{-\sigma}^u}-\frac  1{\ha w_\sigma^u-(\ha w_{-\sigma}^\sigma)^u}\Big]+ \nonumber\\
 &+\sum_{\nu\in\{0,+,-\}} \rho_{\nu} \Big[\frac{ W_{\sigma,1}^u }{W_\sigma^u-V_\nu^u}-\frac {1}{\ha w_\sigma^u- (\ha v_\nu^\sigma)^u}\Big] \bigg ]\,  \frac{d\ha B^u_\sigma}{\sqrt\kappa}.\label{dMitocu}
\end{align}
From Corollary \ref{RN-M-ic-cor} and Proposition \ref{u-st} we know that, for any $\ulin\xi\in\Xi$ and $t\ge 0$,
\BGE \frac{d \PP^{\ulin\rho}|\F_{ \ulin u(t\wedge \tau^u_{\ulin\xi})}}{d \PP^{\ulin\rho}_i|\F_{ \ulin u(t\wedge \tau^u_{\ulin\xi})}}= {M^u(t\wedge \tau^u_{\ulin\xi} )} .\label{RNito4xi}\EDE
We will use  a Girsanov{'s} argument to derive the SDEs for $\ha w_+^u$ and $\ha w_-^u$ up to $T^u_{\disj}$ under $\PP^{\ulin\rho}$.

	For $\sigma\in\{+,-\}$, define a process $\til B^u_\sigma(t)$ such that $\til B^u(t)=0$ and
\begin{align}
d\til B_\sigma^u=&d\ha B_\sigma^u-\bigg[\kappa \bb \frac{W_{\sigma,2}^u}{W_{\sigma,1}^u}+  \Big[\frac{2W_{\sigma,1}^u}{W_\sigma^u-W_{-\sigma}^u}-\frac 2{\ha w_\sigma^u-(\ha w_{-\sigma}^\sigma)^u}\Big] \nonumber \\
& +\sum_{\nu\in\{0,+,-\}} \Big[\frac{\rho_\nu W_{\sigma,1}^u }{W_\sigma^u-V_\nu^u}-\frac {\rho_\nu}{\ha w_\sigma^u- (\ha v_\nu^\sigma)^u}\Big] \bigg] \, \frac{u_\sigma'(t)}{\sqrt\kappa}\, dt.\label{tilwju}
\end{align}

\begin{Lemma}
	For any $\sigma\in\{+,-\}$ and $\ulin \xi\in\Xi$, $|\til B_\sigma^u|$ is bounded on $[0,\tau^u_{\ulin \xi}]$ by a constant depending only on $\kappa,\ulin\rho,\ulin w,\ulin v,\ulin\xi$. \label{uniform2}
\end{Lemma}
\begin{proof}
	Throughout the proof, a positive number that depends only on $\kappa,\ulin\rho,\ulin w,\ulin v,\ulin\xi$ is called a constant.  By Proposition \ref{g-z-sup},   $ V_+$ and $V_-$ are bounded in absolute value by a constant on $[0,\tau _{\ulin\xi}]$, and so are $W_+,V_0,W_-$ because $V_+\ge W_+\ge V_0\ge W_-\ge V_-$. It is clear that $\ha B_\sigma^u(t)=U(u_\sigma(t) \ulin e_\sigma)-U(\ulin 0)$, $\sigma\in\{+,-\}$, where $U:=W_++W_-+\sum_{\nu\in\{0,+,-\}} \frac{\rho_\nu}2 V_\nu$. Thus, $\ha B_\sigma^u$, $\sigma\in\{+,-\}$, are bounded in absolute value by a constant on $[0,\tau^u_{\ulin\xi}]$. By (\ref{V-V}) and that $V_+^u(t)-V_-^u(t)=e^{2t}(v_+-v_-)$ for $0\le t<T^u$, we know that $e^{2 \tau^u_{\ulin\xi}}\le 4\diam(\xi_+\cup\xi_-\cup[v_-,v_+])/|v_+-v_-|$. This means that $\tau^u_{\ulin\xi}$ is bounded above by a constant. Since $\ulin u[0,\tau^u_{\ulin \xi}]\subset [\ulin 0,\tau_{\ulin\xi}]$,  it remains to show that, for $\sigma\in\{+,-\}$, $$\frac{W_{\sigma,2} }{W_{\sigma,1} },\quad \frac{ W_{\sigma,1} }{W_\sigma -W_{-\sigma} }-\frac  1{\ha w_\sigma - \ha w_{-\sigma}^\sigma  },\quad \frac{  W_{\sigma,1}  }{W_\sigma -V_\nu }-\frac 1{\ha w_\sigma -  \ha v_\nu^\sigma },\quad \nu\in\{0,+,-\},$$ are all bounded in absolute value on $[\ulin 0,\tau_{\ulin\xi}]$ by a constant.
	
Because $\frac  1{\ha w_\sigma - \ha w_{-\sigma}^\sigma  }=\frac{ W_{\sigma,1} }{W_\sigma -W_{-\sigma} }\Big|^{-\sigma}_0$, the boundedness of $ \frac{ W_{\sigma,1} }{W_\sigma -W_{-\sigma} }-\frac  1{\ha w_\sigma - \ha w_{-\sigma}^\sigma  } $ on $[\ulin 0, \tau_{\ulin\xi}]$ simply follows from the boundedness of $ \frac{ W_{\sigma,1} }{W_\sigma -W_{-\sigma} } $, which in turn follows from $0\le W_{\sigma,1}\le 1$ and that $|W_\sigma-W_{-\sigma}|$ is bounded from below  on $[\ulin 0, \tau_{\ulin\xi}]$  by a positive constant, where the latter bound was given in the proof of Lemma \ref{uniform}.
		
For the boundedness of $ \frac{W_{\sigma,2} }{W_{\sigma,1} } $ on $[\ulin 0, \tau_{\ulin\xi}]$, we   assume $\sigma=+$ by symmetry. Since $W_{+,j}(t_+,t_-)=g_{K_{-}^{t_+}(t_-)}^{(j)}(\ha w_+(t_+))$, $j=1,2$, and $K_{-}^{t_+}(\cdot)$ are chordal Loewner hulls driven by $W_-(t_+,\cdot)$ with speed $W_{-,1}(t_+,\cdot)^2$, by differentiating  $ {g''_{K_{-}^{t_+}(t_-)}}/{g'_{K_{-}^{t_+}(t_-)}}$ at $\ha w_+(t_+)$ w.r.t.\ $t_-$,  we get
$$\frac{W_{+,2}(t_+,t_-)}{W_{+,1}(t_+,t_-)} =\int_0^{t_-}\frac{4 W_{-,1}^2 W_{+,1}}{(W_+-W_-)^3}\bigg|_{(t_+,s_-)}\,ds.$$
From the facts that $0\le W_{+,1},W_{-,1}\le 1$ and that $|W_+-W_-|$ is bounded from below by a constant on $[\ulin 0, \tau_{\ulin\xi}]$, we get the boundedness of $\frac{W_{+,2} }{W_{+,1} }$.
	
For the  boundedness of $\frac{  W_{\sigma,1}  }{W_\sigma -V_\nu }-\frac 1{\ha w_\sigma -  \ha v_\nu^\sigma }$, we assume by symmetry that $\sigma=+$. 
By differentiating  w.r.t.\ $t_-$ and using (\ref{pa-X},\ref{pajW}), we get
$$\frac{  W_{+,1}(t_+,t_-) }{W_+(t_+,t_-) -V_\nu(t_+,t_-) }-\frac 1{\ha w_+(t_+) -  \ha v_\nu^+(t_+) }=\int_0^{t_-}\frac{2 W_{-,1}^2 W_{+,1}}{(W_+-W_-)^2(V_\nu-W_-)}\bigg|_{(t_+,s_-)}\, ds.$$
Since $0\le W_{+,1} \le 1$, $|W_+-W_-|$ is bounded from below by a constant on $[\ulin 0, \tau_{\ulin\xi}]$, and $V_\nu-W_-$ does not change sign (but could be $0$), it suffices to show that $\big|\int_0^{t_-} \frac{2W_{-,1}^2}{V_\nu-W_-}|_{(t_+,s_-)}\,ds\big|$ is bounded by a constant on $[\ulin 0, \tau_{\ulin\xi}]$. This holds because the integral equals $V_\nu(t_+,t_-)-V_\nu(t_+,0)$ by (\ref{pa-X}), and $|V_\nu|$  is bounded by a constant on $[\ulin 0, \tau_{\ulin\xi}]$.
\end{proof}

\begin{Lemma}
	Under $\PP^{\ulin\rho}$,  there is a stopped planar Brownian motion $\ulin B(t)=(B_+(t),B_-(t))$, $0\le t<T^u_{\disj}$,  such that,
for $\sigma\in\{+,-\}$, $\ha w_\sigma^u$   satisfies the SDE
	$$d\ha w_\sigma^u\aeq \sqrt{\kappa u_\sigma' }dB_\sigma  +\Big[\kappa \bb\frac{W_{\sigma,2}^u}{W_{\sigma,1}^u} +\frac{2W^u_{\sigma,1}}{W^u_\sigma-W^u_{-\sigma}}+ \sum_{\nu\in\{0,+,-\}}  \frac{\rho_\nu W^u_{\sigma,1}}{W^u_\sigma-V^u_\nu}   \Big]u_\sigma'dt,\quad 0\le t<T^u_{\disj}.$$
Here by saying that $(B_+(t),B_-(t))$, $0\le t<T^u_{\disj}$, is a stopped planar Brownian motion, we mean that $B_+(t)$ and $B_-(t)$, $0\le t<T^u_{\disj}$, are local martingales with $d\langle B_\sigma\rangle_t=t$, $\sigma\in\{+,-\}$, $d\langle B_+,B_-\rangle_t=0$, $0\le t<T^u_{\disj}$.
\label{Lem-uB}
\end{Lemma}
\begin{proof} For $\sigma\in\{+,-\}$,  define $\til B_\sigma^u$ using (\ref{tilwju}). By (\ref{dMitocu}), $\til B_\sigma^u(t)M^u(t)$, $0\le t<T^u_{\disj}$, is an $\F^u$-local martingale under $\PP^{\ulin\rho}_i$. By Lemmas \ref{uniform} and \ref{uniform2},   for any $\ulin \xi\in\Xi$,  $\til B_\sigma^u(\cdot\wedge \tau^u_{\ulin \xi})M^u(\cdot \wedge \tau^u_{\ulin \xi})$  is an $\F^u$-martingale under $\PP^{\ulin\rho}_i$. Since this process is $(\F_{ \ulin u(\cdot\wedge \tau^u_{\ulin \xi})})$-adapted, and $\F_{ \ulin u(t\wedge \tau^u_{\ulin \xi})}\subset \F_{ \ulin u(t)}=\F^u_t$, it is also an $(\F_{ \ulin u(\cdot\wedge \tau^u_{\ulin \xi})})$-martingale. From (\ref{RNito4xi}) we see that $\til B_\sigma^u(\cdot\wedge \tau^u_{\ulin \xi})$, is an $(\F_{ \ulin u(\cdot\wedge \tau^u_{\ulin \xi})})$-martingale under $\PP^{\ulin\rho}$.
Then we see that $\til B_\sigma^u(\cdot\wedge \tau^u_{\ulin \xi})$ is an $\F^u$-martingale under $\PP^{\ulin\rho}$ since for any $t_2\ge t_1\ge 0$ and $A\in\F^u_{t_1}=\F_{\ulin u(t_1)}$, $A\cap \{t_1\le \tau^u_{\ulin \xi}\}\subset \F_{ \ulin u(t_1\wedge \tau^u_{\ulin \xi})}$, and on the event $\{t_1> \tau^u_{\ulin \xi}\}$, $\til B_\sigma^u(t_1\wedge \tau^u_{\ulin \xi})=\til B_\sigma^u(t_2\wedge \tau^u_{\ulin \xi})$. Since	$T^u_{\disj}= \sup_{\ulin \xi\in\Xi^*}\tau^u_{\ulin\xi}$, we see that, for $\sigma\in\{+,-\}$, $\til B_\sigma^u(t)$, $0\le t<T^u_{\disj}$, is an $\F^u$-local martingale under $\PP^{\ulin\rho}$.

From (\ref{quadratic}) and that for any $\ulin\xi\in\Xi^*$ and $t\ge 0$, $\PP^{\ulin\rho}\ll \PP^{\ulin\rho}_i$ on $\F_{ \ulin u(t\wedge \tau^u_{\ulin \xi})}$,  we know that, under $\PP^{\ulin\rho}$, (\ref{quadratic}) holds up to $\tau^u_{\ulin\xi}$. Since 	$T^u_{\disj}= \sup_{\ulin \xi\in\Xi^*}\tau^u_{\ulin\xi}$, and $\til B_\pm^u-\ha B_\pm^u$ are differentiable, (\ref{quadratic}) holds for $\til B^u_+$ and $\til B^u_-$ under $\PP^{\ulin\rho}$ up to $T^u_{\disj}$.
Since $\til B_+^u$ and $\til B_-^u$ up to $T^u_{\disj}$ are local martingales under $\PP^{\ulin\rho}$, the (\ref{quadratic}) for $\til B_\pm^u$ implies that there exists a stopped planar Brownian motion $(B_+(t),B_-(t))$, $0\le t<T^u_{\disj}$, under $\PP^{\ulin\rho}$, such that $d\til B_\sigma^u(t)=\sqrt{u_\sigma'(t)}dB_\sigma(t)$, $\sigma\in\{+,-\}$. Combining this fact with (\ref{hawhaB}) and (\ref{tilwju}), we then complete the proof.
\end{proof}

From now on, we work under the probability measure $\PP^{\ulin\rho}$.
Combining Lemma \ref{Lem-uB} with (\ref{dWju}) and (\ref{pa-X}), we get an SDE for $W_\sigma^u-V_0^u$ up to $T^u_{\disj}$:
\begin{align*}
  d(W_\sigma^u-V_0^u)\aeq & \,W_{\sigma,1}^u \sqrt{\kappa u_\sigma'} dB_\sigma +\sum_{\nu\in\{0,+,-\}} \frac{\rho_\nu (W_{\sigma,1}^u)^2 u_\sigma'}{W_\sigma^u-V_\nu^u}\,dt +\frac{2  (W_{\sigma,1}^u)^2 u_\sigma'}{W_\sigma^u-W_{-\sigma}^u}\,dt \\
  &+ \frac{2  (W_{-\sigma,1}^u)^2 u_{-\sigma}'}{W_\sigma^u-W_{-\sigma}^u}\,dt+ \frac{2  (W_{\sigma,1}^u)^2 u_\sigma'}{W_\sigma^u-V_{0}^u}\,dt + \frac{2  (W_{-\sigma,1}^u)^2 u_{-\sigma}'}{W_{-\sigma}^u-V_{0}^u}\,dt.
\end{align*}
Recall that $R_\sigma =\frac{W_\sigma^u -V_0^u }{V_\sigma^u -V_0^u }\in [0,1]$, $\sigma\in\{+,-\}$, and $V_\sigma^u -V_0^u=\sigma e^{2\cdot} I${, where $I=|v_+-v_0|=|v_--v_0|$}.
Combining the above SDE with (\ref{uj''}) and using the continuity of $R_\sigma$ and the positiveness of $R_++R_-$ (because $W_+^u>W_-^u$), we find that $R_\sigma$, $\sigma\in\{+,-\}$, satisfy the following SDE up to $T^u_{\disj}$:
\BGE dR_\sigma =\sigma \sqrt{\frac{\kappa R_\sigma(1-R_\sigma ^2)}{R_++R_-}}dB_\sigma +\frac{(2+\rho_0)-(\rho_\sigma-\rho_{-\sigma})R_\sigma-(\rho_++\rho_-+\rho_0+6)R_\sigma^2}{R_++R_-}\,dt. \label{SDE-R}\EDE

\subsection{SDE in the whole lifespan}
We are going to prove the following theorem in this subsection.

\begin{Theorem}
	Under $\PP^{\ulin\rho}$,  $R_+$ and $R_-$ satisfy (\ref{SDE-R}) throughout $\R_+$ for a pair of independent Brownian motions $B_+$ and $B_-$. \label{Thm-SDE-whole-R}
\end{Theorem}

\begin{Remark}
  Theorem \ref{Thm-SDE-whole-R} gives the   SDE for the two-dimensional diffusion process $(\ulin R)$ in its whole lifespan. In the next subsection, we will use this SDE to derive the transition density of $(\ulin R)$.
\end{Remark}

\begin{Lemma}
	Suppose that $R_+$ and $R_-$ are $[0,1]$-valued semimartingales satisfying (\ref{SDE-R}) for a stopped planar Brownian motion $(B_+,B_-)$ up to some stopping time $\tau$. Then on the event $\{\tau<\infty\}$, a.s.\ $\lim_{t\uparrow \tau} R_\pm(t)$ {exists and equals a finite number other than} $0$. \label{lemma-0-1}
\end{Lemma}
\begin{proof}
Let $X=R_+-R_-$ and $Y=1-R_+R_-$. Then $|X|\le Y\le 1$ because $Y\pm X=(1\pm R_+)(1\mp R_-)\ge 0$.   By (\ref{SDE-R}), $X$ and $Y$ satisfy the following SDEs up to $\tau$:
\BGE dX=
dM_X-[(\rho_++\rho_-+\rho_0+6)X+(\rho_+-\rho_-)]dt,\label{dX}\EDE
\BGE dY=
dM_Y-[(\rho_++\rho_-+\rho_0+6)Y-(\rho_++\rho_-+4)]dt,\label{dY}\EDE
where $M_X$ and $M_Y$ are local martingales whose quadratic variation and covariation satisfy the following equations up to $\tau$:
\BGE d\langle M_X\rangle =\kappa(Y-X^2)dt,\quad d\langle M_X,M_Y\rangle =\kappa (X-XY)dt,\quad d\langle M_Y\rangle =\kappa (Y-Y^2)dt.\label{d<X,Y>}\EDE

By (\ref{d<X,Y>}),  $\langle M_X\rangle_\tau,\langle M_Y\rangle_\tau\le \kappa\tau$, which implies that $\lim_{t\uparrow \tau} M_X(t)$ and  $\lim_{t\uparrow \tau} M_Y(t)$ a.s.\ converge on {the event} $\{\tau<\infty\}$. By (\ref{dX},\ref{dY}),
$\lim_{t\uparrow \tau} (X(t)-M_X(t))$ and $\lim_{t\uparrow \tau} (Y(t)-M_Y(t))$ a.s.\ converge on {the event} $\{\tau<\infty\}$. Combining these results, we see that, on the event $\{\tau<\infty\}$,   $\lim_{t\uparrow \tau}X(t)$ and $\lim_{t\uparrow \tau} Y(t)$ a.s.\ converge, which implies the a.s.\ convergence of $\lim_{t\uparrow \tau} R_\pm(t)$ .

Since $(R_+(t),R_-(t))\to (0,0)$ iff $(X(t),Y(t))\to (0,1)$. It suffices to show that, $ (X(t),Y(t))$ does not converge to $(0,1)$ as $t\uparrow \tau$. Since $(X,Y)$ is Markov, it suffices to show that, if $Y(0)\ne 0$, and if $T=\tau\wedge \inf\{t:Y(t)=0\}$, then $(X(t),Y(t))$ does not converge to $(0,1)$ as $t\uparrow T$. Since $Y\ne 0$ on $[0,T)$, we may define a process $Z=X/Y\in [-1,1]$ on $[0,T)$. Now it suffices to show that $(Z(t),Y(t))$ does not converge to $(0,1)$ as $t\uparrow T$.

From (\ref{dX}-\ref{d<X,Y>}) and It\^o's calculation, we see that there is a stopped planar Brownian motion $(B_{Z}(t),B_Y(t))$, $0\le t<T$, such that $Z$ and $Y$ satisfy the following SDEs on $[0,T)$:
	$$ dZ=\sqrt{\frac{\kappa (1-Z^2)}Y}dB_{Z}-\frac{(\rho_++\rho_-+4)Z+(\rho_+-\rho_-)}{Y}\,dt;$$
	$$ dY=\sqrt{\kappa Y(1-Y)}dB_Y -[(\rho_++\rho_-+\rho_0+6)Y-(\rho_++\rho_-+4)]dt.$$
	Let $v(t)=\int_0^t \kappa/Y(s)ds$, $0\le t<T$, and $\til T=\sup v[0,T)$. Let $\til Z(t)=Z(v^{-1}(t))$ and $\til Y(t)=Y(v^{-1}(t))$, $0\le t<\til T$.   Then there is a stopped planar Brownian motion $(\til B_{Z}(t),\til B_Y(t))$, $0\le t<\til T$, such that $\til Z$ and $\til Y$ satisfy the following SDEs on $[0,\til T)$:
$$d\til Z=\sqrt{1-\til Z^2}d\til B_{Z}-(a_{Z}\til Z+b_{Z}) dt,  $$
$$ d\til Y=\til Y\sqrt{ 1-\til Y}d\til B_Y -\til Y(a_Y(\til Y-1)+b_Y)dt,  $$
	where  $a_{Z}=(\rho_++\rho_-+4)/\kappa$, $b_{Z}=(\rho_+-\rho_-)/\kappa$, $b_Y=(\rho_0+2)/\kappa$, $a_{Y}=a_{Z}+b_Y$.

	Let $\Theta=\arcsin (\til Z)\in [-\pi/2,\pi/2]$ and $\Phi=\log(\frac{1+\sqrt{1-\til Y}}{1-\sqrt{1-\til Y}})\in\R_+$. Then $(\til Z(t),\til Y(t))\to (0,1)$ iff $\Theta(t)^2+\Phi(t)^2\to 0$, and  $\Theta$ and $\Phi$ satisfy the following SDEs on $[0,\til T)$:
	$$ d\Theta=d\til B_{Z}-(a_{Z}-\frac 12)\tan\Theta dt-b_{Z} \sec\Theta dt; $$
	$$ d\Phi=-d\til B_Y+ (b_Y -\frac 14 )\coth_2(\Phi) dt+(\frac 34-a_Y)\tanh_2(\Phi) dt.$$
Here $\tanh_2:=\tanh(\cdot /2)$ and $\coth_2:=\coth(\cdot/2)$.
So for some $1$-dimensional Brownian motion $B_{\Theta,\Phi}$, $\Theta^2+\Phi^2$ satisfies the SDE
$$d(\Theta^2+\Phi^2)=2\sqrt{\Theta^2+\Phi^2}dB_{\Theta,\Phi} +2dt+(2b_Y-\frac 12)\Phi \coth_2(\Phi)dt $$ $$+(\frac 32-2a_Y) \Theta\tanh_2(\Phi)dt-(2a_{Z}-1)\Theta \tan\Theta dt-2b_{Z} \sec\Theta dt.$$
From the power series expansions of $\coth_2,\tanh_2,\tan,\sec$ at $0$,
we see that when $\Theta^2+\Phi^2$ is close to $0$, it behaves like a squared Bessel process of dimension $4b_Y+1=\frac 4\kappa(\rho_0+2)+1\ge 2$ because $\rho_0\ge \frac\kappa 4-2$. Thus, a.s.\ $\lim_{t\uparrow \til T} {\Theta(t)^2+\Phi(t)^2}\ne  0$,  as desired.
\end{proof}

\begin{Remark}
	The assumption $\rho_0\ge \frac\kappa4-2$ is used in the last line of the above proof.
\label{Remark-rho0}
\end{Remark}

\begin{Lemma}
  For every $N>0$ and $L\ge 2$, there is $C>0$ depending only on $\kappa,\ulin\rho,N,L$ such that for any $v_0\in [(-1)^+,1^-]$, $v_+\in [1^+,\infty)$ and $v_-\in (-\infty,(-1)^-]$ with $|v_+-v_-|\le L$, if  $(\eta_+,\eta_-)$ is a commuting pair of chordal SLE$_\kappa(2,\ulin\rho)$ curves started from $(1,-1;v_0,v_+,v_-)$, then for any $y\in(0,N]$, $\PP[E_+(y)\cap E_-(y)]\ge C$, where for $\sigma\in\{+,-\}$, $E_\sigma(y)$ is the event that $\eta_\sigma$ reaches $\{\Imm z=y\}$ before $\{\Ree z=\sigma \frac 12\}\cup \{\Ree z=\sigma\frac 32\}$.
  \label{lower-bound}
\end{Lemma}
\begin{proof}
In this proof, a constant depends only on $\kappa,\ulin\rho,N,L$. Since $E_\pm(y)$ is decreasing in $y$, it suffices to prove that $\PP[E_+(N)\cap E_-(N)]$ is bounded from below by a positive constant. By \cite[Lemma 2.4]{MW}, there is a constant $\til C >0$  such that $\PP[E_\sigma(N)]\ge \til C$ for $\sigma\in \{+,-\}$. Thus, if $(\eta_+',\eta_-')$ is an independent coupling of $\eta_+$ and $\eta_-$, then the events $E_\pm'(N)$ for $(\eta_+',\eta_-')$ satisfy that $\PP[E_+'(N)\cap E_-'(N)]\ge \til C^2$. Let $\xi_\sigma=\HH\cap \pa\{x+iy:|x-\sigma 1|\le \frac 12,0\le y\le N\}$, $\sigma\in\{+,-\}$. Since the law of $(\eta_+,\eta_-)$ restricted to $\F^+_{\tau^+_{\xi_+}}\vee \F^-_{\tau^-_{\xi_-}}$ is absolutely continuous w.r.t.\ that of $(\eta_+',\eta_-')$ (Lemma \ref{RN-M-ic}), and the logarithm of the Radon-Nikodym derivative is bounded in absolute value by a constant (Lemma \ref{uniform}), we get the desired lower bound for $\PP[E_+(N)\cap E_-(N)]$.
\end{proof}

\begin{Corollary}
  For any $r\in(0,1]$, there is $\delta>0$ depending only on $\kappa,\ulin\rho,r$ such that the following holds. Suppose $w_+>w_-\in\R$, $v_0\in [w_-^+,w_+^-]$, $v_+\in [w_+^+,\infty)$, $v_-\in (-\infty,w_-^-]$ satisfy $|v_+-v_0|=|v_--v_0|$ and $|w_+-w_-|\ge r|v_+-v_-|$. Let $(\eta_+,\eta_-)$ be a commuting pair of chordal SLE$_\kappa(2,\ulin\rho)$ curves started from $(w_+,w_-;v_0,v_+,v_-)$. Let $\ulin\xi=(\xi_+,\xi_-)$, where $\xi_\sigma=\HH\cap \pa\{x+iy: |x-w_\sigma|\le |w_+-w_-|/4,0\le y\le e^2|v_+-v_-|\}$, $\sigma\in\{+,-\}$.  Let $\tau^u_{\ulin\xi}$ be as defined in Section \ref{section-diffusion}.  Then $\PP[\tau^u_{\ulin\xi}\ge 1]\ge \delta$. \label{lower-bound-Cor}
\end{Corollary}
\begin{proof}
  Let $E$ denote the event that for both $\sigma\in\{+,-\}$, $\eta_\sigma$ hits $\xi_\sigma$ at its top for the first time. Suppose that $E$ happens. By the definition of $\tau^u_{\ulin\xi}$, for one of $\sigma\in\{+,-\}$, the imaginary part of $\eta_\sigma(u_\sigma(\tau^u_{\ulin\xi}))$  is $e^2|v_+-v_-|$. So $\rad_{v_0}(\eta_\sigma[0,u_\sigma(\tau^u_{\ulin\xi})])\ge e^2 |v_+-v_-|$. By (\ref{V-V'}) we then get $\tau^u_{\ulin\xi}\ge 1$. Thus, $\PP[\tau^u_{\ulin\xi}\ge 1]\ge \PP[E]$, which by  Lemma \ref{lower-bound} and scaling is bounded from below by a positive constant depending only on $\kappa,\ulin\rho,r$ whenever $|v_+-v_-|\le |w_+-w_-|/r$.
 \end{proof}

\begin{proof}[Proof of Theorem \ref{Thm-SDE-whole-R}]
We have known that (\ref{SDE-R}) holds up to $T^u_{\disj}$.
We will combine it with the DMP of commuting pairs of chordal SLE$_\kappa(2,\ulin\rho)$ curves (Lemma \ref{DMP}).

Let $\eta^0_\pm=\eta_\pm$. Let ${\cal G}^0$ be the trivial $\sigma$-algebra. We will inductively define the following random objects. Let $n=1$. We have the $\sigma$-algebra ${\cal G}^{n-1}$ and the pair $(\eta^{n-1}_+,\eta^{n-1}_-)$, whose law conditionally on ${\cal G}^{n-1}$ is that of a commuting pair of chordal SLE$_\kappa(2,\ulin\rho)$ curves. Let $K^{n-1},\mA^{n-1},W^{n-1}_\pm,V^{n-1}_0,V^{n-1}_\pm$,  be respectively the hull function, capacity function, driving functions and force point functions. Let ${\cal D}^{n-1}_{\disj}$ and $\Xi^{n-1}$ be respectively the ${\cal D}_{\disj}$ and $\Xi$ defined  for the pair. Let $\F^{n-1}$ be the $\R_+^2$-indexed filtration defined by $ \F^{n-1}_{(t_+,t_-)}=\sigma({\cal G}^{n-1},\eta_\sigma^{n-1}|_{[0,t_\sigma]},\sigma\in\{+,-\}) $. Let $\ulin u^{n-1}$ be the time curve for $(\eta^{n-1}_+,\eta^{n-1}_-)$ as defined in Section \ref{time curve}, which exits for $n=1$ because we assume that $|v_+-v_0|=|v_0-v_-|$.

Let $\ulin\xi^{n-1}$ be the $\ulin\xi$ obtained from applying Corollary \ref{lower-bound-Cor} to $w_\pm=W^{n-1}_\pm(\ulin 0)$ and $v_\pm=V^{n-1}_\pm(\ulin 0)$. Then it is a ${\cal G}^{n-1}$-measurable random element in $\Xi^{n-1}$. Let $\tau^{n-1}_{\ulin\xi^{n-1}}$ be the random time $\tau^u_{\ulin\xi}$ introduced in Section \ref{section-diffusion} for the $(\eta^{n-1}_+,\eta^{n-1}_-)$ and $\ulin\xi^{n-1}$ here. Let $\ulin\tau^{n-1}=(\tau^{n-1}_+,\tau^{n-1}_-)=\ulin u^{n-1}(\tau^{n-1,u}_{\ulin\xi^{n-1}})$. Then $\ulin\tau^{n-1}$ is a finite $\F^{n-1}$-stopping time that lies in ${\cal D}^{n-1}_{\disj}$. Let ${\cal G}^n=\F^{n-1}_{\ulin\tau^{n-1}}$. We then obtain by Lemma \ref{DMP} a random commuting pair of chordal Loewner curves $(\til\eta^n_+,\til\eta^n_-)$ with some speeds, which is the part of $(\eta^{n-1}_+,\eta^{n-1}_-)$ after ${\ulin\tau^{n-1}}$ up to a conformal map, and
the normalization of $(\til\eta^n_+,\til\eta^n_-)$, denoted by $(\eta^n_+,\eta^n_-)$, conditionally on ${\cal G}^n$, is a commuting pair of chordal SLE$_\kappa(2,\ulin\rho)$ curve started from $(W^{n-1}_+,W^{n-1}_-;V^{n-1}_0,V^{n-1}_+,V^{n-1}_-)|_{\ulin\tau^{n-1}}$.
If for some $\sigma\in\{+,-\}$ and $\nu\in\{0,+,-\}$, $V^{n-1}_\nu(\ulin \tau^{n-1})=W^{n-1}_\sigma(\ulin \tau^{n-1})$, then as a force point, $V^{n-1}_\nu(\ulin \tau^{n-1})$ is treated as $(W^{n-1}_\sigma(\ulin \tau^{n-1}))^{\sign(v_\nu-w_\sigma)}$. By the assumption of $\ulin u^{n-1}$, we have $|V^{n-1}_+-V^{n-1}_0|=|V^{n-1}_--V^{n-1}_0|$ at $\ulin\tau^{n-1}$. So we may increase $n$ by $1$ and repeat the above construction.

Iterating the above procedure, we obtain two sequences of pairs $(\eta^n_+,\eta^n_-)$, $n\ge 0$, and  $(\til\eta^n_+,\til\eta^n_-)$, $n\ge 1$. They satisfy that for any $n\in\N$, $(\eta^n_+,\eta^n_-)$ is the normalization of $(\til\eta^n_+,\til\eta^n_-)$, and $(\til\eta^n_+,\til\eta^n_-)$ is the part of $(\eta^{n-1}_+,\eta^{n-1}_-)$ after ${\ulin\tau^{n-1}}$ up to a conformal map. Let $\phi^n_\pm$ be the speed of $\til\eta^n_\pm$, and   $\phi^n_\oplus(t_+,t_-)=(\phi^n_+(t_+),\phi^n_-(t_-))$. By Lemma \ref{DMP-determin-1}, for any $n\in\N$ and $Z\in\{ W_+,W_-,V_0,V_+,V_-\}$, $\til Z^n=Z^n\circ \phi^n_\oplus$ and  $\til Z^n= Z^{n-1}(\ulin\tau^{n-1}+\cdot)$.


 Recall that, for $n\ge 0$, $\ulin u^{n}$ is characterized by the property that
$ |V^{n}_\pm(\ulin u^{n}(t))-V^{n}_0(\ulin u^{n}(t))|=e^{2t} |V^{n}_\pm(\ulin 0)-V^{n}_0(\ulin 0)|$, $t\ge 0$.
So we get $\ulin u^n=\phi^n_\oplus(\ulin u^{n-1}(\tau^{n-1}_{\ulin\xi^{n-1}}+\cdot)-\ulin u^{n-1}(\tau^{n-1}_{\ulin\xi^{n-1}}))$,
which then implies that $Z^{n-1}\circ \ulin u^{n-1}(\tau^{n-1}_{\ulin\xi^{n-1}}+\cdot)=Z^{n}\circ \ulin u^{n}$, $Z\in \{W_+,W_-,V_0,V_+,V_-\}$. Let $R^n_\pm$ be the $R_\pm$ defined in Section \ref{section-diffusion} for $(\eta^n_+,\eta^n_-)$.  Then we have
$R^{n-1}_\pm(\tau^{n-1}_{\ulin\xi^{n-1}}+\cdot)=R^n_\pm$. Let $T^n=\sum_{j=0}^{n-1} \tau^{j}_{\ulin \xi^j}$, $n\ge 0$.  Since $R_\pm= R_\pm^0$, we get $R_\pm(T^n+\cdot)=R^n_\pm$. For $n\ge 0$, since conditionally on ${\cal G}^{n}$, $(\eta^n_+,\eta^n_-)$ has the law of a commuting pair of chordal SLE$_\kappa(2,\ulin\rho)$ curves started from $(W^n_+,W^n_-;V^n_0,V^n_+,V^n_-)|_{\ulin 0}$, by the previous subsection, there is a stopped two-dimensional Brownian motion $(B^n_+,B^n_-)$ w.r.t.\ $\F^n_{\ulin u^n(\cdot)}$ such that $R^n_+$ and $R^n_-$ satisfy the  $\F^n_{\ulin u^n(\cdot)}$-adapted SDE (\ref{SDE-R}) with $(B^n_+,B^n_-)$ in place of $(B_+,B_-)$ up to $\tau^n_{\ulin\xi^n}$. Let $T^\infty=\lim_{n\to\infty}T^n= \sum_{j=0}^{\infty} \tau^{j}_{\ulin \xi^j}$, and define a continuous processes $B_\pm$ on $[0,T^\infty)$ such that $B_\pm(t)-B_\pm(T^n)=B^n(t -T^n)$ for each $t\in[T^n,T^{n+1}]$ and $n\ge 0$. Then $(B_+,B_-)$ is a stopped two-dimensional Brownian motion, and $R_+$ and $R_-$ satisfy (\ref{SDE-R}) up to $T^\infty$. It remains to show that  a.s.\ $T^\infty=\infty$.

Suppose it does not hold that a.s.\ $T^\infty=\infty$. By Lemma \ref{lemma-0-1}, there is an event $E$ with positive probability and a number $r\in(0,1]$ such that on the event $E$, $R_++R_-\ge 2r$ on $[0,T^\infty)$.
For $n\ge 0$, let $E_n=\{|W^n_+(\ulin 0)-W^n_-(\ulin 0)|\ge r |V^n_+(\ulin 0)-V^n_-(\ulin 0)|\}=\{R^n_+(0)+R^n_-(0)\ge 2r\}$, which is ${\cal G}^n$-measurable. Since $R^n_\pm=R_\pm(T^n+\cdot)$, we get $E\subset \bigcap E_n$. Let  $A_n=\{\tau^n_{\ulin\xi^n}\ge 1\}$. By Corollary \ref{lower-bound-Cor}, there is $\delta>0$ depending only on $\kappa,\ulin\rho,r$ such that for $n\ge 0$, $\PP[A_n|{\cal G}^n, E_n]\ge \delta$. Since $E\subset \{\sum_n \tau^n_{\ulin\xi^n}<\infty\}$, we get $E\subset \liminf (E^n\cap A_n^c)$. For any $m\ge n\in\N$,
$$\PP[  \bigcap_{k=n}^m (E_{k}\cap A_k^c)]=\EE[\PP[  \bigcap_{k=n}^m (E_{k}\cap A_k^c)|{\cal G}^m]]\le (1-\delta)\PP[  \bigcap_{k=n}^{m-1} (E_{k}\cap A_k^c)]. $$
So we get $\PP[  \bigcap_{k=n}^m (E_{k}\cap A_k^c)]\le (1-\delta)^{n-m+1}$, which implies that $\PP[ \bigcap_{k=n}^\infty (E_k\cap A_k^c)]=0$ for every $n\in\N$, and so $\PP[E]=0$. This contradiction completes the proof.
\end{proof}

 \begin{Corollary}
 Almost surely the path $(R_+(t),R_-(t))$, $t\in\R_+$, avoids $(0,0)$.  \label{Cor-0}
\end{Corollary}\begin{proof}
  Let $\tau$ be the first $t$ such that $(R_+(t),R_-(t))= (0,0)$, if such $t$ exists; and be $\infty$ if otherwise. Then $\tau$ is a stopping time, and on the event $\{\tau<\infty\}$, $\lim_{t\uparrow \tau} R_\pm(t)=0$. By Lemma \ref{lemma-0-1}, such event has probability zero.
\end{proof}

\subsection{Transition density}
In this subsection, we are going to use orthogonal polynomials to derive the transition density of $\ulin R(t)=(R_+(t),R_-(t))$, $t\ge 0$,   against the Lebesgue measure restricted to $[0,1]^2$. A similar approach was first used in \cite[Appendix B]{tip} to calculate the transition density of radial Bessel processes, where one-variable orthogonal polynomials were used. Two-variable orthogonal polynomials were used in \cite[Section 5]{Two-Green-interior} to calculate the transition density of a two-dimensional diffusion process. Here we will use another family of two-variable orthogonal polynomials to calculate the transition density of the $(\ulin R)$ here. In addition, we are going to derive the invariant density of $(\ulin R)$, and estimate the convergence of the transition density to the invariant density.

Let $X=R_+-R_-$ and $Y=1-R_+R_-$. Since $R_+$ and $R_-$ satisfy (\ref{SDE-R}) throughout $\R_+$, $X$ and $Y$ then satisfy (\ref{dX},\ref{dY},\ref{d<X,Y>}) throughout $\R_+$. Moreover, {by Corollary \ref{Cor-0}, a.s.\ } $(X,Y)\in \lin\Delta\sem \{(0,1)\}$, where $\Delta=\{(x,y)\in\R^2:0<|x|<y<1\}$. We will first find the transition density of $(X(t),Y(t))$. Assume that the transition density $p(t,(x,y),(x^*,y^*))$ exists, and is smooth in $(x,y)$, then it should satisf{y} the Kolmogorov's backward equation:
\BGE -\pa_t p+{\cal L} p=0,\label{PDE-boundary}\EDE
where ${\cal L}$ is the second order differential operator defined by
{
\begin{align*}
 {\cal L}=&\frac{\kappa}{2} (y-x^2)\pa_x^2+\kappa x(1-y)\pa_x\pa_y+\frac \kappa 2 y(1-y)\pa_y^2 \\
 &-[(\rho_++\rho_-+\rho_0+6)x+(\rho_+-\rho_-)]\pa_x-[(\rho_++\rho_-+\rho_0+6)y-(\rho_++\rho_-+4)]\pa_y.
\end{align*}}
We perform a change of coordinate $(x,y)\mapsto (r,h)$ by $x=rh$ and $y=h$ at $y\ne 0$. Direct calculation shows that
$$\pa_r=h\pa_x,\quad \pa_h=r\pa_x+\pa_y,\quad \pa_r^2=h^2\pa_x^2,\quad
\pa_h^2=r^2\pa_x^2 +2r\pa_x \pa_y+\pa_y^2,\quad \pa_r\pa_h=rh\pa_x^2+h\pa_x\pa_y.$$
Define  $$\alpha_0=\frac {2 }\kappa(\rho_0+2)-1,\quad\alpha_\pm=\frac 2\kappa(\rho_\pm+2)-1, \quad  \beta=\alpha_++\alpha_-+1;$$
$${\cal L}^{(r)}=(1-r^2)\pa_r^2-  [(\alpha_++\alpha_-+2 )r+(\alpha_+-\alpha_-)]\pa_r;$$
$${\cal L}^{(h)}=h(1-h)\pa_h^2- [(\alpha_0+\beta+2)h-(\beta+1)]\pa_h.$$
Then in the $(r,h)$-coordinate, ${\cal L}=\frac{\kappa}{2} [{\cal L}^{(h)}+\frac 1 h {\cal L}^{(r)}]$. Let
$$  \lambda _n=-n(n+\alpha_0+\beta+1),\quad\lambda^{(r)}_n= -n(n+\beta),\quad n\ge 0.$$
 Direct calculation shows that
\BGE [{\cal L}^{(h)}+\frac 1 h \lambda_n^{(r)}]h^n =h^n [{\cal L}^{(h)}-2n(h-1)\pa_h+\lambda_n ],\label{Lie}\EDE
where each $h^n$ in the formula is understood as a multiplication by the $n$-th power of $h$.
From (\ref{Jacobi-ODE}) we know that Jacobi polynomials $P_n^{(\alpha_+,\alpha_-)}(r)$, $n\ge0$,  satisfy that
\BGE \L^{(r)} P_n^{(\alpha_+,\alpha_-)}(r)=\lambda^{(r)}_n  P_n^{(\alpha_+,\alpha_-)}(r),\quad n\ge 0;\label{r-eigen}\EDE
and the functions $P_m^{(\alpha_0,\beta+2n)}(2h-1)$, $m\ge 0$, satisfy that
\BGE (\L^{(h)}-2n(h-1)\pa_h+\lambda_n )P_m^{(\alpha_0,\beta+2n)}(2h-1)=\lambda_{m+n} P_m^{(\alpha_0,\beta+2n)}(2h-1),\quad m\ge 0.\label{h-eigen}\EDE

For  $n\ge 0$, define a homogeneous  two-variable polynomial $Q_n^{(\alpha_+,\alpha_-)}(x,y)$ of degree $n$ such that
$Q_n^{(\alpha_+,\alpha_-)}(x,y)=y^n P_n^{(\alpha_+,\alpha_-)}(x/y)$ if $y\ne 0$.
It has nonzero coefficient for $x^n$. For every pair of integers $n,m\ge 0$, define a two-variable polynomial $ v_{n,m}(x,y)$ of degree $n+m$ by
$$ v_{n,m}(x,y)=P^{(\alpha_0,\beta+2n)}_m(2y-1) Q^{(\alpha_+,\alpha_-)}_n(x,y).$$
Then $ v_{n,m}$ is also a polynomial in $r,h$ with the expression:
\BGE  v_{n,m}(r,h)=h^n P^{(\alpha_0,\beta+2n)}_m(2h-1)P^{(\alpha_+,\alpha_-)}_n(r).\label{vhr}\EDE
By (\ref{Lie},\ref{r-eigen},\ref{h-eigen}), on $\R^2\sem\{y\ne 0\}$,
$$\frac 2\kappa {\cal L} v_{n,m}=\frac 2\kappa[{\cal L}^{(h)}+\frac {1} h {\cal L}^{(r)}] v_{n,m}
=  [{\cal L}^{(h)}+\frac {1} h \lambda_n^{(r)}] (h^n P^{(\alpha_0,\beta+2n)}_m(2h-1)P^{(\alpha_+,\alpha_-)}_n(r))$$
$$=h^n [{\cal L}^{(h)}-2n(h-1)\pa_h+\lambda_n ] (P^{(\alpha_0,\beta+2n)}_m(2h-1)P^{(\alpha_+,\alpha_-)}_n(r))
=\lambda_{n+m}  v_{n,m}.$$
Since $v_{n,m}$ is a polynomial in $x,y$, by continuity the above equation holds throughout $\R^2$.
Thus, for every $n,m\ge 0$, $ v_{n,m}(x,y)e^{\frac \kappa 2\lambda_{n+m}t}$ solves (\ref{PDE-boundary}), and the same is true for any linear combination of such functions.
From (\ref{vhr}) we get an upper bound of $\|v_{n,m}\|_\infty:=\sup_{(x,y)\in\Delta} |v_{n,m}(x,y)|$:
\BGE  \|  v_{n,m}\|_\infty\le   \|P_m^{(\alpha_0,\beta+2n)}\|_\infty \| P_n^{(\alpha_+,\alpha_-)}\|_\infty,
\label{supernorm'}\EDE
where the supernorm of the Jacobi polynomials are taken on $[-1,1]$ as in (\ref{super-exact},\ref{super-upper}).

Since $P_n^{(\alpha_+,\alpha_-)}(r)$, $n\ge0$, are mutually orthogonal w.r.t.\ the weight function $\Psi^{(\alpha_+,\alpha_-)}(r)$, and for any fixed $n\ge 0$, $P_m^{(\alpha_0,\beta+2n)}(2h -1)$, $m\ge 0$, are mutually orthogonal w.r.t.\ the weight function $\Psi^{(\alpha_0,\beta+2n)}(2h-1)=2^{\alpha_0+\beta+2n}{\bf 1}_{(0,1)}(h)(1-h)^{\alpha_0}h^{\beta+2n}$, using a change of coordinates we conclude that $ v_{n,m}(x,y)$, $n,m\in\N\cup\{0\}$, are mutually orthogonal w.r.t.\ the weight function
$$\Psi(x,y):={\bf 1}_\Delta (x,y) (y-x)^{\alpha_+}(y+x)^{\alpha_-} (1-y)^{\alpha_0}.$$
Moreover, we have
\BGE \| v_{n,m}\|^2_\Psi=2^{-(\alpha_0+\beta+2n+1)}\|P_m^{(\alpha_0,\beta+2n)} \|^2_{\Psi^{(\alpha_0,\beta+2n)} }\cdot \|P^{(\alpha_+,\alpha_-)}_n\|^2_{\Psi^{(\alpha_+,\alpha_-)}}.\label{L2-norm'}\EDE

Let $f(x,y)$ be a polynomial in two variables. Then $f$ can be expressed by a linear combination  $f(x,y)=\sum_{n=0}^\infty\sum_{m=0}^\infty a_{n,m}v_{n,m}(x,y)$, where  $a_{n,m}:=\langle f,v_{(n,m)}\rangle_\Psi/\|v_{n,m}\|_\Psi^2$ are zero for all but finitely many $(n,m)$. In fact,  every polynomial in $x,y$ of degree $\le k$ can be expressed as a linear combination of $v_{n,m}$ with $n+m\le k$. In fact, the number of such $v_{n,m}$ is $\frac{(k+1)(k+2)}2$.
Define $$f(t,(x,y))=\sum_{n=0}^\infty\sum_{m=0}^\infty  a_{n,m} v_{n,m}(x,y)e^{\frac\kappa 2\lambda_{n+m}t}=\sum_{n=0}^\infty\sum_{m=0}^\infty  \frac{\langle f,v_{n,m}\rangle_\Psi} {\|v_{n,m}\|_\Psi^2} \cdot v_{n,m}(x,y) e^{\frac\kappa 2\lambda_{n+m}t}.$$
Then $f(t,(x,y))$ solves (\ref{PDE-boundary}). Let $(X(t),Y(t))$ be a diffusion process in $\Delta$, which solves (\ref{dX},\ref{dY},\ref{d<X,Y>}) with initial value $(x,y)$.  Fix $t_0>0$ and define $M_t=f(t_0-t,(X(t),Y(t)))$, $0\le t\le t_0$. By It\^o's formula, $M$ is a bounded martingale, which implies that
$$\EE[f(X({t_0}),Y(t_0))]=\EE[M_{t_0}]=M_0=f(t_0,(x,y))  $$
\BGE =\sum_{n=0}^\infty\sum_{m=0}^\infty \int\!\!\int_{\Delta} f(x^*,y^*) \Psi(x^*,y^*)\frac {v_{n,m}(x^*,y^*) v_{n,m}(x,y)}{\|v_{n,m}\|_\Psi^2}\cdot e^{\frac\kappa 2\lambda_{n+m}t_0}  dx^*dy^*.\label{Ef-boundary}\EDE

For $t>0$, $(x,y)\in\lin\Delta $, and $(x^*,y^*)\in \Delta $, define
\BGE p_t((x,y),(x^*,y^*))=\Psi(x^*,y^*)\sum_{n=0}^\infty\sum_{m=0}^\infty  \frac {v_{n,m}(x,y) v_{n,m}(x^*,y^*)}{\|v_{n,m}\|_\Psi^2}\cdot e^{\frac\kappa 2\lambda_{n+m}t}. \label{prs}\EDE
Let $p_\infty(x^*,y^*)=  {\bf 1}_{\Delta}(x^*,y^*)\Psi(x^*,y^*) /\|1\|_\Psi^2$. Note that $\lambda_0=0$ and $v_{0,0}\equiv 1$ since $P^{\alpha_0,\beta}_0=P^{\alpha_+,\alpha_-}_0\equiv 1$. So $p_\infty(x^*,y^*)$ corresponds to the first term in the series.

\begin{Lemma}
{(i)}	For any $t_0>0$, the series in (\ref{prs}) (without the factor $\Psi(x^*,y^*)$) converges uniformly on $[t_0,\infty)\times \lin\Delta\times \lin\Delta$. {(ii)} There is $C_{t_0}\in(0,\infty)$ depending only on $\kappa, \ulin\rho,t_0$ such that for any $(x,y)\in\lin\Delta$ and $(x^*,y^*)\in\Delta$,
	\BGE |p_t((x,y),(x^*,y^*))-p_\infty(x^*,y^*)|\le  C_{t_0} e^{-  (\rho_++\rho_-+\rho_0+6)t}\Psi(x^*,y^*),\quad t\ge t_0 .\label{asym}\EDE
{(iii)} For any $t>0$ and $(x^*,y^*)\in \lin \Delta$,
	\BGE p_\infty (x^*,y^*)=\int\!\!\int_{\Delta} p_\infty (x,y)p_t((x,y),(x^*,y^*))dxdy.\label{invar}\EDE \label{asym-lemma}
\end{Lemma}
\begin{proof}
{The statements (i) and (ii)} both follow from Stirling's formula, (\ref{supernorm'},\ref{L2-norm'},\ref{norm},\ref{super-upper}), and the facts that $0>\lambda_1= -\frac 2\kappa (\rho_++\rho_-+\rho_0+6)>\lambda_n$ for any $n>1$ and $\lambda_n\asymp - n^2$ for big $n$. {The statement (iii)} follows from {the statement (i)} and the orthogonality of $v_{n,m}$ w.r.t.\ $\langle\cdot,\cdot\rangle_\Psi$.
\end{proof}

\begin{Lemma}
	The process $(X(t),Y(t))$   has a transition density, which is $p_t((x,y),(x^*,y^*))$, and an invariant density, which is $p_\infty(x^*,y^*)$.
\end{Lemma}
\begin{proof}
Fix $(x,y)\in \lin\Delta\sem\{(0,1)\}$. Let $(X(t),Y(t))$ be the process that satisfies (\ref{dX},\ref{dY},\ref{d<X,Y>}) with initial value $(x,y)$. It suffices to show that, for any continuous function $f$ on $\lin\Delta$,  we have
\BGE \EE[f(X(t_0),Y(t_0))]=\int \!\!\int_\Delta p_{t_0}((x,y),(x^*,y^*))f(x^*,y^*)dx^*dy^*.\label{Efp'-boundary}\EDE
By Stone-Weierstrass theorem, $f$ can be approximated by a polynomial in two variables uniformly on $\Delta$. Thus, it suffices to show that (\ref{Efp'-boundary}) holds whenever $f$ is a polynomial in $x,y$, which follows immediately from (\ref{Ef-boundary}). The statement on $p_\infty(x^*,y^*)$ follows from (\ref{invar}).
\end{proof}

Since $X=R_+-R_-$, $Y=1-R_+R_-$, and the Jacobian of the transformation is $-(R_++R_-)$, we arrive at the following statement.

\begin{Corollary} The process $(\ulin R(t))$   has a transition density:
$$p^R_t(\ulin r,\ulin r^*):=  p_t ((r_+-r_-,1-r_+r_-),(r_+^*-r_-^*,1-r_+^*r_-^*))\cdot (r_+^*+r_-^*),$$  and  an invariant density: $p^R_\infty(\ulin r^*):= p_\infty( r_+^*-r_-^*,1-r_+^*r_-^*)\cdot (r_+^*+r_-^*)$; and for any $t_0>0$, there is $C_{t_0}\in(0,\infty)$ depending only on $\kappa$, $\ulin\rho$, and $t_0$ such that for any $\ulin r\in [0,1]^2$ and $\ulin r^*\in(0,1)^2$,
$$ |p^R_t(\ulin r,\ulin r^*)-p^R_\infty(\ulin r^*)|\le  C_{t_0} e^{-  (\rho_++\rho_-+\rho_0+6)t}p^R_\infty(\ulin r^*),\quad t\ge t_0 .$$
\label{transition-R-infty}
\end{Corollary}

\section{Commuting Pair{s} of hSLE Curves}\label{section-other-commut}
In this section, we study three commuting pairs of hSLE$_\kappa$ curves. {Each commuting pair corresponds to one case in Theorem \ref{main-Thm1}, and the section will be split into a subsection for each.} It turns out that each {commuting pair} is ``locally'' absolutely continuous w.r.t.\ a commuting pair of chordal SLE$_\kappa(2,\ulin\rho)$ curves for some suitable force values. So the results in the previous section can be applied here.  Fix $\kappa\in(0,8)$ and $v_-<w_-<w_+<v_+\in\R$. We write $\ulin w=(w_+,w_-)$ and $\ulin v=(v_+,v_-)$. For $\ulin\rho=(\rho_+,\rho_-)$ that satisfies the conditions in Section \ref{subsection-commuting-SLE-kappa-rho}, let $\PP^{\ulin\rho}_{\ulin w;\ulin v}$ denote the law of the driving functions of a commuting pair of choral SLE$_\kappa(2,\ulin\rho)$ curves started from $(\ulin w;\ulin v)$.

\subsection{Two curves in a $2$-SLE$_\kappa$} \label{section-two-curve}
Suppose  that $(\ha\eta_+,\ha \eta_-)$ is a $2$-SLE$_\kappa$ in $\HH$ with link pattern $(w_+\to v_+;w_-\to v_-)$. By \cite[Proposition 6.10]{Wu-hSLE}, for $\sigma \in\{+,-\}$, $\ha \eta_\sigma$ is an hSLE$_\kappa$ curve in $\HH$ from $w_\sigma$ to $v_\sigma$ with force points $w_{-\sigma}$ and $v_{-\sigma}$. Stopping $\ha\eta_\sigma$ at the first time {$t$, denoted by $T_\sigma$, such} that {$\ha\eta_\sigma[0,t]$} disconnects $\infty$ from any of its force points, and parametrizing the stopped curve by $\HH$-capacity, we get a chordal Loewner curve $\eta_\sigma(t)$, $0\le t<T_\sigma$, which is an hSLE$_\kappa$ curve in the chordal coordinate. Let $\ha w_\sigma$ and $K_\sigma(\cdot)$ be respectively the chordal Loewner driving function and hulls for $\eta_\sigma$;  and let $\F^\sigma$ be the filtration generated by $\eta_{\sigma}$. Let $\F$ be the separable  $\R_+^2$-indexed filtration generated by $\F^+$ and $\F^-$.

For $\sigma\in\{+,-\}$, if $\tau_{-\sigma} $ is an $\F^{-\sigma}$-stopping time, then conditionally on $\F^{-\sigma}_{\tau_{-\sigma}}$ and the event $\{\tau_{-\sigma}<T_{-\sigma}\}$,  the whole $\eta_{ \sigma}$ and the part of $\ha\eta_{-\sigma}$ after $\eta(\tau_{-\sigma})$ together
 form a $2$-SLE$_\kappa$ in $\HH\sem K_{-\sigma}(\tau_{-\sigma})$ with link pattern $(w_\sigma\to v_\sigma;\eta_{ -\sigma}(\tau_{-\sigma})\to v_{-\sigma})$. This in particular implies that the conditional law of $\ha\eta_\sigma$ is that of an hSLE$_\kappa$ curve from $w_\sigma$ to $v_\sigma$ in $\HH\sem K_{-\sigma}(\tau_{-\sigma})$  with force points $\eta_{-\sigma}(\tau_{-\sigma})$ and $v_{-\sigma}$. Since $ f_{K_{-\sigma}(\tau_{-\sigma})}$ maps $\HH$ conformally onto  $\HH\sem K_{-\sigma}(\tau_{-\sigma})$, and sends $\ha w_{-\sigma}(\tau_{-\sigma})$, $g_{K_{-\sigma}(\tau_{-\sigma})}(w_\sigma)$ and $g_{K_{-\sigma}(\tau_{-\sigma})}(v_\nu)$, $\nu\in\{+,-\}$, respectively to $\eta_{-\sigma}(\tau_{-\sigma})$, $w_\sigma$ and $v_\nu$, $\nu\in\{+,-\}$, we see that  there a.s.\ exists a chordal Loewner curve $\eta_{\sigma}^{\tau_{-\sigma}}$ with some speed such that $\eta_ \sigma= f_{K_{-\sigma}(\tau_{-\sigma})}\circ \eta_ {\sigma,\tau_{-\sigma}}$, and the conditional law of the normalization of $\eta_ {\sigma,\tau_{-\sigma}}$ given $\F^{-\sigma}_{\tau_{-\sigma}}$ is that of an hSLE$_\kappa$ curve in $\HH$ from $g_{K_{-\sigma}(\tau_{-\sigma})}(w_\sigma)$ to $g_{K_{-\sigma}(\tau_{-\sigma})}(v_\sigma) $ with force points $\ha w_{-\sigma}(\tau_{-\sigma})$ and $g_{K_{-\sigma}(\tau_{-\sigma})}(v_{-\sigma})$, in the chordal coordinate.

 Thus, $(\eta_+,\eta_-)$ a.s.\ satisfies the conditions in Definition \ref{commuting-Loewner} with $I_\sigma=[0,T_\sigma)$, $\I_\sigma^*=\I_\sigma \cap\Q$, $\sigma\in\{+,-\}$, and  ${{\cal D}_1}:=\I_+\times \I_-$. By discarding a null event, we assume that $(\eta_+,\eta_-;{\cal D}_1)$ is always a commuting pair of chordal Loewner curves, and call $(\eta_+,\eta_-;{\cal D}_1)$ a commuting pair of hSLE$_\kappa$ curves in the chordal coordinate started from $(\ulin w;\ulin v)$.  We  adopt the functions from Section \ref{section-deterministic}.
Define a function $M_1$ on ${\cal D}_1$ by $M_1=G_1(W_+,W_-;V_+,V_-)$, where $G_1$ is given by (\ref{G1(w,v)}).   Since $F$ is continuous and positive on $[0,1]$, $ |W_\sigma-V_\nu|\le |V_+-V_-|$ for $\sigma,\nu\in\{+,-\}$, and $\frac 8\kappa-1,\frac 4\kappa>0$, there is a constant $C>0$ depending only on $\kappa$ such that
  \BGE M_1\le C|V_+-V_-|^{\frac {16}\kappa-1} \min_{\sigma\in\{+,-\}} \{|W_\sigma-V_\sigma|\}^{\frac 8\kappa -1}\le C |V_+-V_-|^{2(\frac{12}\kappa -1)}.\label{M1-boundedness}\EDE
Note that $M_1>0$ on $\cal D$ because $|W_\sigma-V_\sigma|>0$, $\sigma\in\{+,-\}$, on $\cal D$.
We will prove that $M_1$ extends to an $\F$-martingale on $\R_+^2$, {and the extended martingale plays the role of a} Radon-Nikodym derivative between two measures. We first need some deterministic properties of $M_1$.

\begin{Lemma}
   $M_1$ a.s.\ extends continuously to $\R_+^2$ with  $M_1\equiv 0$ on $\R_+^2\sem {\cal D}_1$. \label{M-cont}
\end{Lemma}
\begin{proof}
It suffices to show that for $\sigma\in\{+,-\}$, as $t_\sigma\uparrow T_\sigma$, $M_1\to 0$ uniformly in $t_{-\sigma}\in [0,T_{-\sigma})$. By symmetry, we may assume that $\sigma=+$.
Since the union of (the whole) $\eta_+$ and $\eta_-$ is bounded, by (\ref{V-V})  $|V_+-V_-|$ is bounded (by random numbers) on ${\cal D}_1$.
For a fixed $t_-\in[0,T_-)$, as $t_+\uparrow T_+$, $\eta_+(t_+)$ tends to either some point on $[v_+,\infty)$ or some point on $(-\infty, v_-)$.
By (\ref{M1-boundedness}), it suffices to show that when $\eta_+$ terminates at $[v_+,\infty)$ (resp.\ at $(-\infty,v_-)$), $W_+-V_+\to 0$ (resp.\ $W_--V_-\to 0$) as $t_+\uparrow T_+$, uniformly in $[0,T_-)$.

For any $\ulin t=(t_+,t_-)\in{\cal D}_1$, neither $\eta_+[0,t_+]$ nor $\eta_-[0,t_-]$ hit $(-\infty, v_-]\cup [v_+,\infty)$, which implies that $v_+,v_-\not\in \lin{K(\ulin t)}$ and $V_\pm(\ulin t)= g_{K(\ulin t)}(v_\pm)$. Suppose that $\eta_+$ terminates at $x_0\in [v_+,\infty)$. Since SLE$_\kappa$ is not boundary-filling for $\kappa\in(0,8)$, we know that $ \dist(x_0,\eta_-)>0$.  Let $r=\min\{|w_+-v_+|,\dist(x_0,\eta_-)\}>0$.
 	 Fix $\eps\in(0,r)$. Since $x_0=\lim_ {t\uparrow T_+} \eta_+(t)$, there is $\delta>0$ such that $|\eta_+(t)-x_0|<\eps$ for $t\in(T_+-\del,T_+)$.  Fix $t_+\in(T_+-\delta,T_+)$ and $t_-\in[0,T_-)$. Let $J$ be the connected component of $\{|z-x_0|=\eps\}\cap (\HH\sem K(\ulin t))$ whose closure contains $x_0+\eps$. Then $J$ disconnects $v_+$ and $\eta_+(t_+,T_+)\cap(\HH\sem K(\ulin t))$ from $\infty$ in $\HH\sem K(\ulin t)$. Thus, $g_{K(\ulin t)}(J)$ disconnects $V_+(\ulin t)$ and $W_+(\ulin t)$ from $\infty$.  Since $\eta_+\cup\eta_-$ is bounded, there is a (random) $R\in(0,\infty)$ such that $\eta_+\cup\eta_-\subset \{|z-x_0|<R\}$. Let $\xi=\{|z-x_0|=2R\}\cap\HH$. By comparison principle, the extremal length (\cite{Ahl}) of the family of curves in $\HH\sem K(\ulin t)$ that separate $J$ from $\xi$ is   $\le \frac{\pi}{\log(R/\eps)}$. By conformal invariance, the extremal length  of the family of curves in $\HH$ that separate $g_{K(\ulin t)}(J)$ from $g_{K(\ulin t)}(\xi)$ is also $\le \frac{\pi}{\log(R/\eps)}$. Now $g_{K(\ulin t)}(\xi)$ and $g_{K(\ulin t)}(J)$ are crosscuts of $\HH$ such that the former encloses the latter. Let $D$ denote the subdomain of $\HH$ bounded by  $g_{K(\ulin t)}(\xi)$. From Proposition \ref{g-z-sup} we know that $D\subset\{|z-x_0|\le 5R\}$. So the Euclidean area of $D$ is less than $13\pi R^2$.  By the definition of extremal length, there is a curve $\gamma$ in $D$ that separates $g_{K(\ulin t)}(J)$ from $g_{K(\ulin t)}(\xi)$  with Euclidean distance less than $2\sqrt{13\pi R^2*\frac{\pi}{\log(R/\eps)}}<8\pi R*\log(R/\eps)^{-1/2}$. Since $g_{K(\ulin t)}(J)$ disconnects $V_+(\ulin t)$ and $W_+(\ulin t)$ from $\infty$, $\gamma$ also separates $V_+(\ulin t)$ and $W_+(\ulin t)$ from $\infty$. Thus, $|W_+(\ulin t)-V_+(\ulin t)|<8\pi R*\log(R/\eps)^{-1/2}$ if $t_+\in(T_+-\del,T_+)$ and $t_-\in[0,T_-)$. This proves the uniform convergence of $\lim_{t_+\uparrow T_+} |W_+-V_+|=0$ in $t_-\in[0,T_-)$ in the case that $\lim_{t_+\uparrow T_+}\eta_+(t_+)\in[v_+,\infty)$. The proof of the uniform convergence of $\lim_{t_+\uparrow T_+} |W_--V_-|=0$ in $t_-\in[0,T_-)$ in the case that $\lim_{t_+\uparrow T_+}\eta_+(t_+)\in(-\infty,v_-)$ is similar.
\end{proof}

From now on, {we extend $M_1$ to $\R_+^2$ using Lemma \ref{M-cont}. It is then} a continuous stochastic process defined on $\R_+^2$ with constant {value} zero on ${\R_+^2}\sem {\cal D}_1$.
For $\sigma\in\{+,-\}$ and $R>|v_+-v_-|/2$, let $\tau^\sigma_R$ be the first time that $|\eta_\sigma(t)-(v_++v_-)/2|=R$ if such time exists; otherwise $\tau^\sigma_R=T_\sigma$. Let $\ulin \tau_R=(\tau^+_R,\tau^-_R)$. Note that $\tau^+_R,\tau^-_R\le \mA(\ulin \tau_R)\le R^2/2$ because   $K(\ulin\tau_R)\subset \{z\in\HH:|z-(v_++v_-)/2|\le R\}$.

\begin{Lemma}
  For every $R>0$, $M_1(\cdot\wedge \ulin \tau_R)$ is an $\R_+^2$-indexed martingale w.r.t.\ the filtration $(\F^+_{t_+\wedge \tau^+_R}\vee \F^-_{t_-\wedge \tau^-_R})_{(t_+,t_-)\in\R_+^2}$ closed by  $M_1(\ulin \tau_R)$. Moreover, if the underlying probability measure {for the $(\eta_+,\eta_-)$ described at the beginning of this subsection} is weighted by $M_1(\ulin\tau_R)/M_1(\ulin 0)$, then the new law of the driving functions $(\ha w_+,\ha w_-)$ agrees with   $\PP^{(2,2)}_{\ulin w;\ulin v} $  on the $\sigma$-algebra $\F^+_{ \tau^+_R}\vee \F^-_{  \tau^-_R}$. \label{M-mart}
\end{Lemma}
\begin{proof} Let $R>0$, $\sigma\in\{+,-\}$, $t_{-\sigma}\ge 0$, and $\tau_{-\sigma}=t_{-\sigma}\wedge \tau^{-\sigma}_R$. Since $W_\sigma|^{-\sigma}_{\tau_{-\sigma}}$, $W_{-\sigma}|^{-\sigma}_{\tau_{-\sigma}}$ and $V_\nu|^{-\sigma}_{\tau_{-\sigma}}$, $\nu\in\{+,-\}$, are all $(\F^\sigma_t\vee \F^{-\sigma}_{\tau_{-\sigma}})_{t\ge 0}$-adapted, and are driving function and force point functions for hSLE$_\kappa$ curves with some speeds in the chordal coordinate conditional on $\F^{-\sigma}_{\tau_{-\sigma}}$, by Proposition \ref{Prop-iSLE-3} (with a time-change), $M_1|^{-\sigma}_{\tau_{-\sigma}}(t)$, $0\le t<T_\sigma$, is an $(\F^\sigma_t\vee \F^{-\sigma}_{\tau_{-\sigma}})_{t\ge 0}$-local martingale. Since $M_1$ is uniformly bounded on $[\ulin 0,\ulin\tau_R]$ and $\tau^\pm_R\le R^2/2$, $M_1|^{-\sigma}_{\tau_{-\sigma}}(\cdot \wedge \tau^\sigma_R)$ is an $(\F^+_{t_+\wedge \tau^+_R}\vee \F^-_{t_-\wedge \tau^-_R})_{t_\sigma\ge 0}$-martingale closed by $M_1|^{-\sigma}_{\tau_{-\sigma}}(\tau^\sigma_R)$. Applying this result twice respectively for $\sigma=+$ and $-$, we obtain the martingale property of $M_1(\cdot\wedge \ulin \tau_R)$

Let $\PP_1$ denote the  underlying probability measure.
By weighting $\PP_1$ by $M_1(\ulin\tau_R)/M_1(\ulin 0)$, we get another probability measure, denoted by $\PP_0$. To describe the restriction of $\PP_0$ to $\F _{\ulin\tau_R}$, we study the new marginal law of $\eta_-$ up to $\tau^-_R$ and the new conditional law of $\eta_+$ up to $\tau^+_R$ given that part of $\eta_-$. We may first weight $\PP_1$ by $N_1:=M_1(0,\tau^-_R)/M_1(0, 0)$ to get a new probability measure $\PP_{.5}$, and then weight $\PP_{.5}$ by $N_2:=M_1(\tau^+_R,\tau^-_R)/M_1(0,\tau^-_R)$ to get $\PP_0$.

By Proposition \ref{Prop-iSLE-3}, the $\eta_-$ up to $\tau^-_R$ under $\PP_{.5}$ is a chordal SLE$_\kappa (2,2,2)$ curve in $\HH$ started from $w_-$ with force points $v_-,w_+,v_+$, respectively, up to $\tau^-_R$. Since $N_1$ depends only on $\eta_-$, the conditional law of $\eta_+$ given any part of $\eta_-$ under $\PP_{.5}$ agrees with that under $\PP_1$ . Since $M_1(0,\tau^-_R)=0$ implies that $N_1=0$, and $\PP_{.5}$-a.s.\ $N_1>0$, we see that  $N_2$ is $\PP_{.5}$-a.s.\ well defined. Since $\EE[N_2|\F^-_{\tau^-_R}]=1$, the law of $\eta_-$ up to $\tau^-_R$ under $\PP_0$ agrees with that under $\PP_{.5}$.  To describe the conditional law of $\eta_+$ up to $\tau^+_R=\tau^+_R(\eta_+)$ given $\eta_-$ up to $\tau^-_R$, it suffices to consider the conditional law of $\eta_{+}^{\tau^-_R}$ up to $\tau^+_R(\eta_+)$ since we may recover $\eta_+$ from $\eta_+^{\tau^-_R}$ using $\eta_+=f_{K_-(\tau^-_R)}\circ \eta_{+}^{\tau^-_R}$. By Proposition \ref{Prop-iSLE-3} again, the  conditional law of the normalization of $\eta_{+}^{\tau^-_R}$ up to $\tau^+_R(\eta_+)$ under $\PP_0$ is that of a chordal SLE$_\kappa(2,2,2)$ curve in $\HH$ started from $W_+(0,\tau^-_R)$ with force points at $V_+(0,\tau^-_R)$, $W_-(0,\tau^-_R)$ and $V_-(0,\tau^-_R)$, respectively. Thus, under $\PP_0$ the joint law of $\eta_+$ up to $\tau^+_R$ and $\eta_-$ up to $\tau^-_R$ agrees with that of a commuting pair of SLE$_\kappa(2,2,2)$ curves started from $(\ulin w;\ulin v)$ respectively up to $\tau^+_R$ and $\tau^-_R$. So the proof is complete
\end{proof}

We let $\PP_1$ denote the joint law of the driving functions $\ha w_+$ and $\ha w_-$ here, and let $\PP^0_1=\PP^{(2,2)}_{\ulin w;\ulin v}$. {We will later define $\PP_j$ and $\PP^0_j$, $j=2,3$, in Sections \ref{section-iSLE-1} and \ref{section-iSLE-2}, and the three pairs of measures $\PP_j,\PP^0_j$, $j=1,2,3$, will be referred in Section \ref{summary}.}
From the lemma, we find that, for any $\ulin t=(t_+,t_-)\in\R_+^2$ and $R>0$,
\BGE \frac{d\PP^0_1|(\F^+_{t_+\wedge\tau^+_R}\vee \F^-_{t_-\wedge \tau^-_R} )}{d\PP_1|(\F^+_{t_+\wedge\tau^+_R}\vee \F^-_{t_-\wedge\tau^-_R}) }=\frac {M_1(\ulin t\wedge \ulin \tau_R)} {M_1(\ulin 0)}.\label{RN-formula0}\EDE

\begin{Lemma}
Under $\PP_1$, $M_1 $ is an $\F$-martingale; and for any $\F$-stopping time $\ulin \tau$,
	\BGE \frac{d\PP^0_1| \F _{\ulin \tau } \cap\{\ulin \tau\in\R_+^2\}   }{d\PP_1| \F _{\ulin \tau } \cap\{\ulin \tau\in\R_+^2\} }=\frac {M_1(\ulin \tau)} {M_1(\ulin 0)}.\label{RN-formula}\EDE
	\label{RN-Thm1}
\end{Lemma}
\begin{proof} For $\ulin t\in\R_+^2$ and $R>0$,
since  $\F^+_{t_+\wedge \tau^+_R}\vee \F^-_{t_-\wedge \tau^-_R}$ agrees with $\F^+_{t_+}\vee \F^-_{t_-}=\F_{\ulin t}$ on $\{\ulin t\le \ulin\tau_R\}$, by (\ref{RN-formula0}),
$$\frac{d\PP^0_1|(\F_{\ulin t}\cap \{\ulin t\le \ulin \tau_R\}) }{d\PP_1|(\F_{\ulin t}\cap \{\ulin t\le \ulin \tau_R\}) }=\frac {M_1(\ulin t)} {M_1(\ulin 0)}.$$
Sending $R\to \infty$, we get ${d\PP^0_1| \F_{\ulin t}  }/{d\PP_1| \F_{\ulin t}  }={M_1(\ulin t)}/ {M_1(\ulin 0)}$ for all $\ulin t\in\R_+^2$. So $M_1$ is an $\F$-martingale  under $\PP_1$.
Let $\ulin \tau$ be an $\F$-stopping time. Fix $A\in\F_{ \ulin \tau}\cap\{\ulin \tau\in\R_+^2\}$. Let $\ulin t\in\R_+^2$.  Define the $\F$-stopping time $\ulin \tau^{\ulin t}$ as in Proposition \ref{prop-local}. Then $A\cap \{\ulin \tau< \ulin t\}=A\cap \{\ulin\tau<\ulin \tau^{\ulin t}\}\in \F_{\ulin \tau^{\ulin t}}\subset\F_{\ulin t}$. So we get $$\PP^0_1[A\cap \{\ulin \tau< \ulin t\}]=\EE_1 \Big[{\bf 1}_{A\cap \{\ulin \tau< \ulin t\}}\frac{M_1(\ulin t)}{M_1(\ulin 0)}\Big ]=\EE_1 \Big[{\bf 1}_{A\cap \{\ulin \tau< \ulin t\}}\frac{M_1(\ulin \tau^{\ulin t})}{M_1(\ulin 0)}\Big ]=\EE_1\Big [{\bf 1}_{A\cap \{\ulin \tau< \ulin t\}}\frac{M_1( \ulin\tau)}{M_1(\ulin 0)} \Big],$$
where the second ``$=$'' follows from Proposition \ref{OST}. Sending both coordinates of $\ulin t$ to $\infty$, we get $\PP^0_1[A ]=\EE_1 [{\bf 1}_{A }\frac{M_1(\ulin \tau)}{M_1(\ulin 0)} ]$. So we get the desired (\ref{RN-formula}).
\end{proof}

\begin{Lemma}
  For any $\F$-stopping time $\ulin \tau$,
	$$ \frac{d\PP_1| \F_{\ulin \tau }\cap \{\ulin \tau\in{\cal D}_1\}  }{d\PP^0_1| \F_{\ulin \tau } \cap \{\ulin \tau\in{\cal D}_1\}  }=\frac {M_1(\ulin 0) }{M_1(\ulin \tau) } .$$
	\label{RN-Thm1-inv}
\end{Lemma}
\begin{proof}
  This follows from Lemma \ref{RN-Thm1} and the fact that $M_1>0$ on ${\cal D}_1$.
\end{proof}

 Assume now that $v_0:=(v_++v_-)/2\in [w_-,w_+]$. We understand $v_0$ as $w_\sigma^{-\sigma}$ if $(v_++v_-)/2=w_\sigma$, $\sigma\in\{+,-\}$. Let $V_0$ be the force point function started from $v_0$.  By Section \ref{time curve}, we may  define the time curve $\ulin u:[0,T^u)\to {\cal D}_1$ such that $V_\sigma(\ulin u(t))-V_0(\ulin u(t))=e^{2t}(v_\sigma-v_0)$, $0\le t<T^u$, $\sigma\in\{+,-\}$, and $\ulin u$ {cannot} be extended beyond $T^u$ {while still satisfying this property}. We follow the notation there: for every $X$ defined on $\cal D$, we use $X^u$ to denote the function $X\circ \ulin u$ defined on $[0,T^u)$. We also define the processes $R_\sigma =\frac{W_\sigma^u -V_0^u }{V_\sigma^u -V_0^u }\in [0,1]$, $\sigma\in\{+,-\}$,  and $\ulin R=(R_+,R_-)$. Since $T_\sigma$ is an $\F^\sigma$-stopping time for $\sigma\in\{+,-\}$, ${\cal D}_1=[0,T_+)\times [0,T_-)$ is an $\F$-stopping region. As before we extend $\ulin u$ to $\R_+$ such that if $s\ge T^u$ then $\ulin u(s)=\lim_{t\uparrow T^u}\ulin u(t)$. By Proposition \ref{Prop-u(t)}, for any $t\ge 0$, $\ulin u(t)$ is an $\F$-stopping time.

Let $I=v_+-v_0=v_0-v_-$ and define $G_1^*$ on $[0,1]^2$ by $G_1^*(r_+,r_-)=G_1(r_+,-r_-;1,-1)$.
Then $M_1^u(t)=(e^{2t}I)^{\alpha_1} G_1^*(\ulin R(t) )$ for $t\in [0,T^u)$, where $\alpha_1=2(\frac{12}\kappa-1)$ is as in Theorem \ref{main-Thm1}

We 
now derive the transition density of the process $(\ulin R(t) )_{0\le t<T^u}$ under $\PP_1$. In fact,  $T^u$ is $\PP_1$-a.s.\ finite. By saying that $\til p^{R}_1(t,\ulin r,\ulin r^*)$ is the transition density of $(\ulin R)$ under $\PP_1$, we mean that, if $(\ulin R(t))$ starts from $\ulin r$, then for any bounded measurable function $f$ on $(0,1)^2$,
$$\EE_1[{\bf 1}_{\{T^u>t\}}f(\ulin R(t) )]=\int_{[0,1]^2} f(\ulin r^*)\til p^{R}_1(t,\ulin r,\ulin r^*) d\ulin r^*,\quad t>0.$$

Applying Lemma \ref{RN-Thm1-inv} to the $\F$-stopping time $\ulin u(t)$, and using that $\ulin u(t)\in{\cal D}_1$ iff $t<T^u$, we get
$$\frac{d\PP_{1}| \F^u_{t}\cap \{T^u>t\}  }{d\PP^0_1| \F^u_{t} \cap \{T^u>t\}  }=\frac {M_1^u(0) }{M_1^u(t) } =e^{-2\alpha_1t} \frac{G_1^* (\ulin R(0)) }{G_1^* (\ulin R(t)) },\quad t\ge 0.$$
Combining it with Corollary \ref{transition-R-infty}, we get the following transition density.

\begin{Lemma}
Let $p^1_t(\ulin r,\ulin r^*)$ be the transition density $p^R_t(\ulin r,\ulin r^*)$ given in Corollary \ref{transition-R-infty} with $\rho_0=0$ and $\rho_+=\rho_-=2$. Then under $\PP_{1}$, the transition density of $(\ulin R)$ is
  $$\til p^1_t(\ulin r,\ulin r^*):= e^{-2\alpha_1 t}p^1_t(\ulin r,\ulin r^*) {G_1^* (\ulin r)}/{G_1^* (\ulin r^*)}.$$    \label{transition-1}
\end{Lemma}

\subsection{Opposite pair{s} of hSLE$_\kappa$ curves, the generic case} \label{section-iSLE-1}
Second, we consider another pair of random curves. Let $\ulin w=(w_+,w_-)$ and $\ulin v=(v_+,v_-)$ be as before. Let $(\eta_w,\eta_v)$ be a $2$-SLE$_\kappa$ in $\HH$ with link pattern $(w_+\lr w_-;v_+\lr v_-)$. For $\sigma\in\{+,-\}$, let $\ha \eta_\sigma$ be the curve $\eta_w$ oriented from $w_\sigma$ to $w_{-\sigma}$ and parametrized by the capacity viewed from $w_{-\sigma}$, which is an hSLE$_\kappa$ curve in $\HH$  from $w_\sigma$ to $w_{-\sigma}$ with force points $v_\sigma$ and $v_{-\sigma}$. Then $\ha \eta_+$ and $\ha \eta_-$ are time-reversal of each other.

For $\sigma\in\{+,-\}$, parametrizing the part of $\ha \eta_\sigma$ up to the time that it disconnects $w_{-\sigma}$ from $\infty$ by $\HH$-capacity, we get a chordal Loewner curve:  $\eta_\sigma(t)$, $0\le t<T_\sigma$, which is an hSLE$_\kappa$ curve in the chordal coordinate. Let $\ha w_\sigma$ and $K_\sigma(\cdot)$ denote the chordal Loewner driving function and hulls for $\eta_\sigma$. Let  $K(t_+,t_-)=\Hull(K_+(t_+)\cup K_-(t_-))$, $(t_+,t_-)\in[0,T_+)\times [0,T_-)$, and define an HC region:
\BGE {\cal D}_2=\{\ulin t\in[0,T_+)\times [0,T_-):K(\ulin t)\subsetneqq \Hull(\eta_w)\}.
\label{D-before-overlap}\EDE

For $\sigma\in \{+,-\}$, let $\F^\sigma$  be the filtration generated by $\eta_\sigma$.  Let $\tau_-$ be an $\F^-$-stopping time. Conditionally on $\F^-_{\tau_-}$ and the event $\{\tau_-<T_-\}$, the part of $\ha\eta_w$ between $\eta_-(\tau_-)$ and $w_+$ and the whole  $\eta_v$ form a $2$-SLE$_\kappa$ in $\HH\sem K_-(\tau_-)$ with link pattern $(w_+\lr \eta_-(\tau_-);v_+\lr v_-)$. So the conditional law of the part of $\ha\eta_+$ up to hitting $\eta_-(\tau_-)$ is that of an hSLE$_\kappa$ curve in $\HH\sem K_-(\tau_-)$ from $w_+$ to $\eta_-(\tau_-)$ with force points $v_+,v_-$, up to a time-change. This implies that there is a random curve $\ha\eta_+^{\tau_-}$ such that the $f_{K_-(\tau_-)}$-image of $\ha\eta_+^{\tau_-}$ is the above part of $\ha\eta_+$, and the conditional law of a time-change of $\ha\eta_+^{\tau_-}$ is that of an hSLE$_\kappa$ curve in $\HH$ from $g_{K_-(\tau_-)}(w_+)$ to $\ha w_-(\tau_-)$ with force points $g_{K_-(\tau_-)}(v_+),g_{K_-(\tau_-)}(v_-)$. By the definition of ${\cal D}_2$, the part of $\eta_+$ up to $T^{{\cal D}_2}_+(\tau_-)$ is a time-change of the part of $\ha\eta_+$ up to the first time that it hits $\eta_-(\tau_-)$ or separates $\eta_-(\tau_-)$ from $\infty$, which is then the $f_{K_-(\tau_-)}$-image of the part of $\ha\eta_+^{\tau_-}$ up to the first time that it hits $\ha w_-(\tau_-)$ or separates $\ha w_-(\tau_-)$ from $\infty$. So there is a random curve $\eta_+^{\tau_-}$ such that the $f_{K_-(\tau_-)}$-image of $\eta_+^{\tau_-}$ is the part of $\eta_+$ up to $T^{{\cal D}_2}_+(\tau_-)$, and the conditional law of a time-change of $\eta_+^{\tau_-}$ is that of an hSLE$_\kappa$ curve in $\HH$ from $g_{K_-(\tau_-)}(w_+)$ to $\ha w_-(\tau_-)$ with force points $g_{K_-(\tau_-)}(v_+),g_{K_-(\tau_-)}(v_-)$, in the chordal coordinate. A similar statement holds with ``$+$'' and ``$-$'' swapped.

Taking the stopping times in the previous paragraph to be deterministic numbers, we find that $(\eta_+,\eta_-;{\cal D}_2)$ a.s.\  satisfies the conditions in Definition \ref{commuting-Loewner} with $\I_\pm=[0,T_\pm)$ and $\I_{\pm}^*=\I_{\pm}\cap\Q$. By removing a null event, we may assume that $(\eta_+,\eta_-;{\cal D}_2)$ is always a commuting pair of chordal Loewner curves. We call $(\eta_+,\eta_-;{\cal D}_2)$ a commuting pair of hSLE$_\kappa$ curves in the chordal coordinate started from $(w_+\lr w_-;v_+,v_-)$.

Let $\F$ be the separable $\R_+^2$-indexed filtration generated by $\F^+$ and $\F^-$, and let $\lin\F$ be the right-continuous augmentation of $\F$. Then ${\cal D}_2$ is an $\lin\F$-stopping region because by Lemma \ref{lem-strict},
$$\{\ulin t\in{\cal D}_2\}=\lim_{\Q^2\ni \ulin s\downarrow\ulin t}(\{\ulin s<(T_+,T_-)\}\cap \{K(\ulin t)\subsetneqq K(\ulin s)\})\in\lin\F_{\ulin t},\quad \forall \ulin t\in\R_+^2.$$

 Define $M_2:{\cal D}_2\to\R_+$ by $M_2=G_2(W_+,W_-;V_+,V_-)$, where $G_2$ is given by (\ref{G2(w,v)}).
Since $V_+\ge W_+\ge W_-\ge V_-$, and $F$ is uniformly positive on $[0,1]$, there is a constant $C>0$ depending only on $\kappa$ such that
 \BGE M_2\le C |W_+-W_-|^{\frac8\kappa-1} |V_+-V_-|^{\frac{16}\kappa -1}\le C |V_+-V_-|^{\frac{2}\kappa(12-\kappa)}.\label{est-M2}\EDE

\begin{Lemma}
   $M_2$ a.s.\ extends continuously to $\R_+^2$ with  $M_2\equiv 0$ on $\R_+^2\sem {\cal D}_2$. \label{M-cont2}
\end{Lemma}
\begin{proof}
Since for $\sigma\in\{+,-\}$, $\eta_\sigma$ a.s.\ extends continuously to $[0,T_\sigma]$, by Remark \ref{Remark-continuity-W}, $W_+$ and $W_-$ a.s.\ extend continuously to $\lin{{\cal D}_2}$. From (\ref{V-V}) we know that a.s.\ $|V_+-V_-|$ is bounded on ${\cal D}_2$. Thus, by (\ref{est-M2}) it suffices to show that the continuations of $W_+$ and $W_-$ agree on $\pa{{\cal D}_2}\cap\R_+^2$. Define $A_\sigma=\{t_\sigma\ulin e_\sigma+T^{{\cal D}_2}_{-\sigma}(t_\sigma)\ulin e_{-\sigma}:t_\sigma\in (0,T_\sigma)\}$, $\sigma\in\{+,-\}$.
Then $A_+\cup A_-$ is dense in $\pa{{\cal D}_2}\cap(0,\infty)^2$. By symmetry, it suffices to show that $W_+$ and $W_-$ agree on $A_+$. 
If this is not true, then there exists $(s_+,s_-)\in{\cal D}_2$ such that $W_+(s_+,\cdot)>W_-(s_+,\cdot)$ on $[s_-,T^{{\cal D}_2}_-(s_+)]$.  

Let $K_-^{\ulin s}(t)=K(s_+,s_-+t)/K(\ulin s)=K_-^{s_+}(s_-+t)/K_-^{s_+}(s_-)$, $0\le t<T':=T^{{\cal D}_2}_-(s_+)-s_-$. Since $K_{-}^{s_+}(t_-)$, $0\le t_-<T^{{\cal D}_2}_-(s_+)$, are chordal Loewner hulls driven by $W_-(s_+,\cdot)$ with speed $\mA(s_+,\cdot)$, $K_{-}^{\ulin s}(t)$, $0\le t<T'$, are chordal Loewner hulls driven by $W_-(s_+,s_-+\cdot)$ with speed $\mA(s_+,s_-+\cdot)$. By Lemma \ref{W=gw} and Proposition \ref{prop-comp-g}, $W_+(s_+,s_-+t)=g_{K_-^{\ulin s}(t)}^{W_-(\ulin s)}(W_+(\ulin s))$, $0\le t< T'$. Since $W_+(s_+,s_-+\cdot)>W_-(s_+,s_-+\cdot)$ on $[0,T')$, we have $\dist(W_+(\ulin s),K_-^{\ulin s}(t))>0$ for $0\le t<T'$. Since $W_+(s_+,s_-+\cdot)$, $W_-(s_+,s_-+\cdot)$ and $\mA(s_+,s_-+\cdot)$ all extend continuously to $[0,T']$, and $W_+(s_+,s_-+T')>W_-(w_+,s_-+T')$, the chordal Loewner process driven by $W_-(s_+,s_-+t)$, $0\le t\le T'$, with speed $\mA(s_+,s_-+\cdot)$ does not swallow $W_+(\ulin s)$ at the time $T'$, which implies that $\dist(W_+(\ulin s),\Hull(\bigcup_{0\le t<T'} K_-^{\ulin s}(t)))>0$.

Since $\ulin s\in{\cal D}_2$, by Lemma \ref{lem-strict} we may choose a (random) sequence $\delta_n\downarrow 0$ such that $\eta_+(s_++\delta_n)\in\HH\sem K(s_+,s_-)$ for all $n$.
Let $z_n=g_{K(s_+,s_-)}(\eta_+(s_++\delta_n))\in K(s_++\delta_n,s_-)/K(s_+,s_-)$, $n\in\N$, then $z_n\to W_+(\ulin s)$ by (\ref{W-def}). So  $\dist(z_{n},\Hull(\bigcup_{0\le t<T'} K_-^{\ulin s}(t)))>0$ for $n$ big enough. However, from
$$\eta_+(s_++\delta_{n})\in \Hull(\eta_w)\sem K(\ulin s)=K(s_+,s_-+T')\sem K(\ulin s)=\Hull(\bigcup_{0\le t<T'} K(s_+,s_-+t))\sem K(\ulin s)$$
we get $z_n\in \Hull(\bigcup_{0\le t<T'} K_-^{\ulin s}(t))$ for all $n$, which is a contradiction.
\end{proof}

From now on, we understand $M_2$ as the continuous extension defined in Lemma \ref{M-cont2}. Let $\tau^\pm_R$ and $\ulin\tau_R$, $R>0$, be as defined before Lemma \ref{M-mart}.

\begin{Lemma}
	For any $R>0$, $M_2(\cdot\wedge \ulin \tau_R)$ is  an $(\F^+_{t_+\wedge \tau^+_R}\vee \F^-_{t_-\wedge \tau^-_R})_{(t_+,t_-)\in\R_+^2}$-martingale closed by $M_2(\ulin \tau_R)$, and if the underlying probability measure is weighted by $M_2(\ulin\tau_R)/M_2(\ulin 0)$, then the new law of $(\ha w_+,\ha w_-)$ agrees with the probability measure $\PP^{(2,2)}_{\ulin w;\ulin v} $  on  $\F^+_{ \tau^+_R}\vee \F^-_{  \tau^-_R}$.
	\label{M-mart2}
\end{Lemma}
\begin{proof}
We follow the argument in the proof of Lemma \ref{M-mart}, where Proposition \ref{Prop-iSLE-3} is the key ingredient, except that here we use (\ref{est-M2})  instead of (\ref{M1-boundedness}). \end{proof}

Let $\PP_2$ denote the joint law of the driving functions $\ha w_+$ and $\ha w_-$ here, and let $\PP^0_2=\PP^{(2,2)}_{\ulin w;\ulin v}$.  Following  the proof of Lemma \ref{RN-Thm1} and using  Lemma \ref{M-mart2}, we get the following lemma.

\begin{Lemma}
A revision of Lemma \ref{RN-Thm1} holds with all subscripts ``$1$'' replaced by ``$2$'' and the filtration $\F$ replaced by $\lin\F$.
 \label{RN-Thm2-right}
\end{Lemma}

\begin{Lemma}
  For any $\lin\F$-stopping time $\ulin \tau$, $M_2(\ulin \tau)$ is  $\PP_2$-a.s.\ positive on  $\{\ulin \tau\in{\cal D}_2\}$. \label{positive-M-T}
\end{Lemma}
\begin{proof}
  Let $\ulin \tau$ be an $\lin\F$-stopping time.  Let $A=\{\ulin \tau\in{\cal D}_2\}\cap\{M_2(\ulin \tau)=0\}$.  We are going to show that $\PP_2[A]=0$. Since ${\cal D}_2$ is an $\lin\F$-stopping region,
we have $\{\ulin \tau\in{\cal D}_2\}\in \lin\F_{\ulin \tau}$, and $A \in  \lin\F_{\ulin \tau}$. Since $M_2(\ulin \tau)=0$ on $A$, by Lemma \ref{RN-Thm2-right},   $\PP^0_2[A]=0$. For any $\ulin t\in\Q_+^2$, since $A\in\lin\F_{\ulin\tau+\ulin t}$, by  Lemma \ref{RN-Thm2-right}, $\PP_2$-a.s\ $M_2(\ulin \tau+\ulin t)=0$  on $A$. Thus, on the event $A$, $\PP_2$-a.s.\ $M_2(\ulin \tau+\ulin t)=0$ for any $\ulin t\in\Q_+^2$, which  implies by the continuity that $M_2\equiv 0$ on $\ulin \tau+\R_+^2$, which further implies that $W_+\equiv W_-$ on $(\ulin \tau+\R_+^2)\cap {\cal D}_2$, which in turn implies by Lemma \ref{Lem-W} that $\eta_+(\tau_++t_+)=\eta_-(\tau_-+t_-)$ for any $\ulin t=(t_+,t_-)\in\R_+^2$ such that $\ulin\tau+\ulin t\in {\cal D}_2$. This is impossible since it implies (by setting $t_-=0$) that  $\eta_+$ stays constant on $[\tau_+,T^{{\cal D}_2}_+(\tau_-))$. So we have $\PP_2[A]=0$.
\end{proof}

\begin{Remark}
  We do not have  $M_2>0$ on ${\cal D}_2$ if there is $(t_+,t_-)\in {\cal D}_2$ such that $\eta_+(t_+)= \eta_-(t_-)$, which almost surely happens when $\kappa\in(4,8)$.
\end{Remark}

\begin{Lemma}
A revision of Lemma \ref{RN-Thm1-inv} holds with all subscripts ``$1$'' replaced by ``$2$'' and the filtration $\F$ replaced by $\lin\F$.
	\label{RN-Thm2-inv}
\end{Lemma}
\begin{proof}
  This follows from Lemmas \ref{RN-Thm2-right} and   \ref{positive-M-T}.
\end{proof}

Assume that $v_0:=(v_++v_-)/2\in [w_-,w_+]$, and let $V_0$ be the force point function started from $v_0$. Here if $v_0=w_\sigma$ for some $\sigma\in\{+,-\}$, we treat it as $w_\sigma^{-\sigma}$. We may   define the time curve $\ulin u:[0,T^u)\to {\cal D}_2$ and the processes $R_\sigma(t)$, $\sigma\in\{+,-\}$, and $\ulin R(t)$ as in Section \ref{time curve}, and extend $\ulin u$ to $\R_+$ such that $\ulin u(s)=\lim_{t\uparrow T^u} \ulin u(t)$ for $s\ge T^u$. Since ${\cal D}_2$ is an $\lin\F$-stopping region, by Proposition \ref{Prop-u(t)}, for any $t\ge 0$, $\ulin u(t)$ is an $\lin\F$-stopping time.

Define $G_2^*$ on $[0,1]^2$ by $G_2^*(r_+,r_-)=G_2(r_+,-r_-;1,-1)$.
Then $M_2^u(t)=(e^{2t}I)^{\alpha_1} G_2^*(\ulin R(t) )$ for $t\in [0,T^u)$, where $\alpha_2=2(\frac{12}\kappa-1)$ is as in Theorem \ref{main-Thm1}.
Applying Lemma \ref{RN-Thm2-inv} to $\ulin u(t)$, we get
the following lemma, which is similar to Lemma \ref{transition-1}.

\begin{Lemma}
  Let $p^2_t(\ulin r,\ulin r^*)$ be the transition density $p^R_t(\ulin r,\ulin r^*)$ given in Corollary \ref{transition-R-infty} with $\rho_0=0$ and $\rho_+=\rho_-=2$. Then under $\PP _{2}$, the transition density of $(\ulin R)$ is
  $$ \til p^2_t(\ulin r,\ulin r^*):= e^{-2\alpha_2 t}p^2_t(\ulin r,\ulin r^*) {G_2^* (\ulin r)}/{G_2^* (\ulin r^*)}.$$
  \label{transition-2}
\end{Lemma}

\subsection{Opposite pair{s} of hSLE$_\kappa$ curves,  a limit case } \label{section-iSLE-2}
Let $w_-<w_+<v_+\in\R$.  Let $(\eta_w,\eta_v)$ be a $2$-SLE$_\kappa$ in $\HH$ with link pattern $(w_+\lr w_-;v_+\lr \infty)$. For $\sigma\in\{+,-\}$, let $\ha \eta_\sigma$ be the curve $\eta_w$ oriented from $w_\sigma$ to $w_{-\sigma}$ and parametrized by the capacity viewed from $w_{-\sigma}$, which is an hSLE$_\kappa$ curve in $\HH$ from $w_\sigma$ to $w_{-\sigma}$. Then $\ha \eta_+$ and $\ha \eta_-$ are time-reversal of each other.

For $\sigma\in\{+,-\}$, parametrizing the part of $\ha \eta_\sigma$ up to the time that it disconnects $w_{-\sigma}$ from $\infty$ by $\HH$-capacity, we get a chordal Loewner curve:  $\eta_\sigma(t)$, $0\le t<T_\sigma$, which is an hSLE$_\kappa$ curve from $w_\sigma$ to $w_{-\sigma}$ in the chordal coordinate. Define ${\cal D}_3$ using (\ref{D-before-overlap}) for the $(\eta_+,\eta_-)$ here. Then $(\eta_+,\eta_-;{\cal D}_3)$ is a.s.\ a commuting pair of chordal Loewner curves.  Define $W_\pm$, $V_+$ and $\lin\F$ for the $(\eta_+,\eta_-)$ here in the same way as in the previous subsection. Then ${\cal D}_3$ is an $\lin\F$-stopping region. We call $(\eta_+,\eta_-;{\cal D}_3)$ a commuting pair of hSLE$_\kappa$ curves in the chordal coordinate started from $(w_+\lr w_-;v_+)$.

 Define $M_3:{\cal D}_3\to\R_+$ by $M_3=G_3(W_+,W_-;V_+)$, where $G_3$ is given by (\ref{G3(w,v)}).
Since $V_+\ge W_+\ge W_-$, we have $M_3\le C |W_+-W_-|^{\frac 8\kappa -1}|V_{+}-V_-|^{\frac {4}\kappa}\le C |V_{+}-V_-|^{\frac {12}\kappa-1}$ for some constant $C>0$  depending on $\kappa$.
 Then the exactly same proof of Lemma \ref{M-cont2} can be used here to prove the following lemma.

\begin{Lemma}
 $M_3$ a.s.\ extends continuously to $\R_+^2$ with  $M_3\equiv 0$ on $\R_+^2\sem {\cal D}_3$. \label{M-cont3}
\end{Lemma}

Let $\PP_3$   denote the joint law of the driving functions $\ha w_+$ and $\ha w_-$ here, and let  $\PP^0_3$ be the joint law of the driving functions for a commuting pair of chordal SLE$_\kappa(2,2)$ started from $(w_+,w_-;v_+)$. Then  similar arguments as in the previous subsection give  the following lemma.

\begin{Lemma}
Revision of Lemmas \ref{positive-M-T} and \ref{RN-Thm2-inv} hold with all subscripts ``$2$'' replaced by ``$3$''.
\label{RN-Thm3-inv}
\end{Lemma}

Introduce two new points: $v_0=(w_++w_-)/2$ and $v_-=2v_0-v_+$. Let $V_0$ and $V_-$ be respectively the force point functions started from $v_0$ and $v_-$.
Since $v_0=(v_++v_-)/2$, we may  define the time curve $\ulin u:[0,T^u)\to {\cal D}_2$ and the processes $R_\sigma (t)$, $\sigma\in\{+,-\}$, and $\ulin R(t)$  as in Section \ref{time curve}.
Let $G_3^*(r_+,r_-)=G_3(r_+,r_-;1)$. Then $M_3^u(t)=(e^{2t}I)^{\alpha_1} G_3^*(\ulin R(t) )$ for $t\in [0,T^u)$, where $\alpha_3=\frac{12}\kappa-1$ is as in Theorem \ref{main-Thm1}.
Applying Lemma \ref{RN-Thm3-inv} to $\ulin u(t)$,
we get the following lemma, which is similar to Lemma \ref{transition-1}.

\begin{Lemma}
Let $p^3_t(\ulin r,\ulin r^*)$ be the transition density $p^R_t(\ulin r,\ulin r^*)$ given in Corollary \ref{transition-R-infty} with $\rho_0=\rho_-=0$ and $\rho_+=2$. Then under $\PP_3$, the transition density of $(\ulin R)$ is
$$ \til p^3_t(\ulin r,\ulin r^*):= e^{-2\alpha_3 t}p^3_t(\ulin r,\ulin r^*) {G_3(\ulin r)}/{G_3(\ulin r^*)}.$$ \label{transition-3}
\end{Lemma}

\subsection{A summary}\label{summary}

For $j=1,2,3$, using Lemmas \ref{transition-1}, \ref{transition-2}, and \ref{transition-3}, we can obtain a quasi-invariant density of $\ulin R$ under   $\PP_j$ as follows. Let $G_j^*(r_+,r_-)=G_j(r_+,-r_-;1,-1)$, $j=1,2$, and $G_3^*(r_+,r_-)=G_3(r_+,-r_-;1)$.
Let $p^j_\infty$ be the invariant density $p^R_\infty$ of $\ulin R$ under  $\PP_0^j$ given by Corollary \ref{transition-R-infty}, where $\PP_0^1=\PP_0^2=\PP_{r_+,-r_-;1,-1}^{(2,2)}$ and $\PP_0^3=\PP_{r_+,-r_-;1}^{(2)}$. Define
\BGE {\cal Z}_j=\int_{(0,1)^2} \frac{p^j_\infty(\ulin r^*)}{G_j^*(\ulin r^*)}d\ulin r^*,\quad \til p^j_\infty=\frac 1{{\cal Z}_j} \frac{p^j_\infty}{G_j^* },\quad j=1,2,3. \label{tilZ}\EDE
It is straightforward to check that ${\cal Z}_j\in(0,\infty)$, $j=1,2,3$. To see this, we compute $p^j_\infty(r_+,r_-)\asymp (1-r_+)^{\frac  8\kappa -1}(1-r_-)^{\frac  8\kappa -1} (r_+r_-)^{\frac 4\kappa -1}$ for $j=1,2$, and $\asymp (1-r_+)^{\frac 8\kappa -1}(1-r_-)^{\frac 4\kappa -1}(r_+r_-)^{\frac 4\kappa -1}$ for $j=3$; $G_j^*(r_+,r_-)\asymp (1-r_+)^{\frac  8\kappa -1}(1-r_-)^{\frac  8\kappa -1}$ for $j=1$, and $\asymp (r_++r_-)^{\frac 8\kappa -1}$ for $j=2,3$.

\begin{Lemma} The following statements hold.
\begin{enumerate}
  \item [(i)] For any $j\in\{1,2,3\}$, $t>0$ and $\ulin r^*\in(0,1)^2$, $ \int_{[0,1]^2}  \til p^j_\infty(\ulin r)\til p^j_t(\ulin r,\ulin r^*)d\ulin r=e^{-2\alpha_jt}\til p^j_\infty(\ulin r^*)$.
This means, under the law $\PP_j$, if the process $(\ulin R)$ starts from a random point in $ (0,1)^2$ with density $\til p^j_\infty$, then for any deterministic $t\ge 0$, the density of (the survived) $\ulin R(t)$ is $e^{-2\alpha_j t} \til p^j_\infty$. So we call $\til p^R_j$ a quasi-invariant density for $(\ulin R)$ under $\PP_j$.
\item [(ii)] Let $\beta_1=\beta_2=10$ and $\beta_3=8$. For $j\in\{1,2,3\}$ and $\ulin r\in(0,1)^2$, if $\ulin R$ starts from $\ulin r$, then
	\BGE\PP_{j}[T^u>t]={{\cal Z}_j} G_j^*(\ulin r) e^{-2\alpha_j t}(1+O(e^{- \beta_j t}));
	\label{P[T>t]}\EDE
	\BGE \til p^R_j(t,\ulin r,\ulin r^*)=\PP_{j}[T^u>t]\til p^j_\infty(\ulin r^*)(1+O(e^{-\beta_j t})).\label{tilptT>t}\EDE
Here we emphasize that the implicit constants in the $O$ symbols do not depend on $\ulin r$.
\end{enumerate}
\label{property-til-p}
\end{Lemma}
\begin{proof}
	Part (i) follows easily from (\ref{invar}). For part (ii), suppose $\ulin R$  starts from $\ulin r$. Using Corollary \ref{transition-R-infty}, Lemmas \ref{transition-1}, \ref{transition-2}, and \ref{transition-3}, and formulas (\ref{tilZ}), we get
	$$\PP_{j}[T^u>t]=\int_{(0,1)^2} \til p^j_t(\ulin r,\ulin r^*)d\ulin r^*=\int_{(0,1)^2}  e^{-2\alpha_j t}  p^j_t(\ulin r,\ulin r^*) \frac{G_j^*(\ulin r)}{G_j^*(\ulin r^*)}d\ulin r^*$$
	$$=\int_{(0,1)^2} e^{-2\alpha_j t}  p^j_\infty(\ulin r^*)(1+O(e^{-\beta_j t})) \frac{G_j^*(\ulin r)}{G_j^*(\ulin r^*)}d\ulin r^*={\cal Z}_j G_j^*(\ulin r) e^{-2\alpha_jt}(1+O(e^{-\beta_j t})),$$
	which is (\ref{P[T>t]}); and
	$$ \til p^j_t(\ulin r,\ulin r^*)=e^{-2\alpha_j t}  p^j_\infty(\ulin r^*)(1+O(e^{-\beta_j t})) \frac{G_j^*(\ulin r)}{G_j^*(\ulin r^*)}
	=e^{-2\alpha_j t} {\cal Z}_j \til p^j_\infty(\ulin r^*)(1+O(e^{-\beta_j t})) G_j^*(\ulin r),$$
	which together with (\ref{P[T>t]}) implies (\ref{tilptT>t}).
\end{proof}

We will need the following lemma, which follows from the argument in   \cite[Appendix A]{Green-cut}.

\begin{Lemma}
For $j=1,2,3$, the $(\eta_+,\eta_-;{\cal D}_j)$ in the three subsections  satisfies the two-curve DMP as described in Lemma \ref{DMP} except that the conditional law of the normalization of $(\til\eta_+,\til\eta_-;\til{\cal D}_j)$ has the law of a commuting pair of hSLE$_\kappa$ curves in the chordal coordinate respectively started from $(W_+,W_-;V_+,V_-)|_{\ulin\tau}$, $(W_+\lr W_-;V_+,V_-)|_{\ulin\tau}$, and $(W_+\lr W_-;V_+)|_{\ulin\tau}$.
\label{DMP-123}
\end{Lemma}

\section{Boundary Green's Functions}
We are going to prove the main theorem in this section.

\begin{Lemma}
For $j=1,2$, let $U_j$ be a simply connected subdomain of the Riemann sphere $\ha\C$,  which contains $\infty$ but no{t} $0$, and let $f_j$ be a conformal map from $\D^*:=\ha\C\sem \{|z|\le 1\}$  onto $U_j$, which fixes $\infty$. Let $a_j=\lim_{z\to \infty} |f_j(z)|/|z|>0$, $j=1,2$, and $a=a_2/a_1$. If $R>4a_1$, then $\{|z|> R\}\subset U_1$, and  $\{|z|> aR+4a_2\}\subset f_2\circ f_1^{-1}(\{|z|> R\}) \subset \{|z|\ge aR-4a_2\}$. \label{distortion}
\end{Lemma}
\begin{proof}
By scaling we may assume that $a_1=a_2=1$. Let $f=f_2\circ f_1^{-1}$.
That $\{|z|>4\}\subset U_1$ follows from Koebe's $1/4$ theorem applied to $J\circ f_1\circ J$, where $J(z):=1/z$. Fix $z_1\in U_1$. Let $z_0=f_1^{-1}(z_1)\in\D^*$ and $z_2=f_2(z_0)\in U_2$.   Let $r_j=|z_j|$, $j=0,1,2$. Applying Koebe's distortion theorem to $J\circ f_j\circ  J$, we find that $r_0+\frac 1{r_0}-2\le r_j\le r_0+\frac 1{r_0}+2$, $j=1,2$, which implies that $ |r_1-r_2|\le 4$. Thus, for $R>4$,  $f(\{|z|> R\})\subset\{|z|> R-4 \}$, and $f(\{|z|=R\})\subset \{|z|\le R+4\}$. The latter inclusion implies that $f(\{|z|> R\})\supset \{|z|> R+4\}$.
\end{proof}

\begin{Theorem}
  Let $v_-<w_-<w_+<v_+\in\R$ be such that $0\in [v_-,v_+]$. Let $(\ha\eta_+,\ha\eta_-)$ be a $2$-SLE$_\kappa$ in $\HH$ with link pattern $(w_+\lr v_+;w_-\lr v_-)$. Let $\alpha_1=2(\frac{12}\kappa -1)$,  $\beta_1'=\frac{5}{6}$, and
 $G_1(\ulin w;\ulin v)$ be as in (\ref{G1(w,v)}).
  Then there is a constant $C>0$ depending only on $\kappa$ such that,
  \BGE  \PP[\ha\eta_\sigma\cap\{|z|>L\}\ne \emptyset,\sigma\in\{+,-\}]= CL^{-\alpha_1} G_1(\ulin w;\ulin v)(1+O( {|v_+-v_-|}/L)^{\beta_1'}),\label{Thm1-est}\EDE
  as $L\to \infty$,
  where the implicit constants in the $O(\cdot)$ symbol depend only on $\kappa$. \label{Thm1}
\end{Theorem}
\begin{proof}
Let $p(\ulin w;\ulin v;L)$ denote the LHS of (\ref{Thm1-est}).
Construct the random commuting pair of chordal Loewner curves $(\eta_+,\eta_-;{\cal D}_1)$ from $\ha\eta_+$ and $\ha \eta_-$ as in Section \ref{section-two-curve}, where ${\cal D}_1=[0,T_+)\times [0,T_-)$, and $T_\sigma$ is the lifetime of $\eta_\sigma$, $\sigma\in\{+,-\}$. We adopt the symbols from Sections \ref{section-deterministic1}. Note that, when $L>|v_+|\vee |v_-|$, $\ha\eta_+$ and $\ha\eta_-$ both intersect $\{|z|>L\}$ if and only if $\eta_+$ and $\eta_-$ both intersect $\{|z|>L\}$. {The fact is:} for any $\sigma\in\{+,-\}$,  $\eta_\sigma$ either disconnects $v_{{\sigma}}$ from $\infty$, or disconnects $v_{{-\sigma}}$ from $\infty$. If $\eta_\sigma$ does not intersect $\{|z|>L\}$, then in the former case, $\ha\eta_\sigma$ grows in a bounded connected component of $\HH\sem \eta_\sigma$ after the end of $\eta_\sigma$, and so {cannot} hit $\{|z|>L\}$; and in the latter case $\eta_{-\sigma}$ grows in a bounded connected component of $\HH\sem \eta_\sigma$, and {cannot} hit $\{|z|>L\}$. We first  consider a special case:  $v_\pm=\pm 1$ and $w_\pm=\pm r_\pm $, where $r_\pm \in[0,1)$. Let $v_0=0$. {This case corresponds to the additional assumption (\ref{add-assump}) up to translation and dilation.} Let $V_\nu$ be the force point function started from $v_\nu$, $\nu\in\{0,+,-\}$, as before. Since $|v_+-v_0|=|v_0-v_-|$, we may define a time curve $\ulin u:[0,T^u)\to \cal D$ as in Section \ref{time curve} and adopt the symbols from there.
Define $p(\ulin r;L)=p(r_+,-r_-;1,-1;L)$.

Suppose $L>2e^6$, and so $\frac 12 \log(L/2)>3$. Let $t_0\in [3,\frac 12\log(L/2))$.  If both $\eta_+$ and $\eta_-$ intersect $\{|z|>L\}$, then there is some $t'\in[0,T^u)$ such that either $\eta_+\circ u_+[0,t']$ or $\eta_-\circ u_-[0,t']$ intersects $\{|z|>L\}$, which by (\ref{V-V'}) implies that $L\le 2e^{2t'}$, and so $T^u>t'\ge \log(L/2)/2>t_0$.
Thus, $\{\eta_\sigma\cap\{|z|>L\}\ne\emptyset,\sigma\in\{+,-\}\}\subset \{T^u>t_0\}$. By (\ref{V-V'}) again, $\rad_{0}(\eta_\sigma[0,u_\sigma(t_0)])\le 2e^{2t_0}<L$. So $\eta_\sigma\circ u_\sigma[0,t_0]$, $\sigma\in\{+,-\}$, do not  intersect $\{|z|>L\}$.

Let $\ha g^u_{t_0}(z)= (g_{K(\ulin u(t_0))}(z)-V^u_0(t_0))/e^{2t_0}$. Then $\ha g^u_{t_0}$ maps $\C\sem (K(\ulin u(t_0))^{\doub}\cup [v_-,v_+])$ conformally onto $\C\sem [-1,1]$, and fixes $\infty$ with $\ha g^u_{t_0}(z)/z\to e^{-2t_0}$ as $z\to\infty$.
 From $V^u_- \le v_-<0$, $V^u_+ \ge v_+>0$, and $V^u_0=(V^u_++V^u_-)/2$, we get  $|V^u_0(t_0)|\le |V^u_+(t_0)-V^u_-(t_0)|/2 =e^{2t_0}$. Applying Lemma \ref{distortion} to $f_2(z)=(z+1/z)/2$, $a_2=1/2$,  $f_1=(\ha g^u_{t_0})^{-1}\circ f_2$ and $a_1=e^{2t_0}/2$, and using that $L>2 e^{2t_0}$, we get $\{|z|> L\}\subset \C\sem (K(\ulin u(t_0))^{\doub}\cup [v_-,v_+])$ and
\BGE \{|z|> L/e^{2t_0}-2\}\supset \ha g^u_{t_0}(\{|z|> L\})\supset \{|z|> L/e^{2t_0}+2\}.\label{R1R2}\EDE

Note that both $\eta_+$ and $\eta_-$ intersect $\{|z|>L\}$ if and only if $T^u>t_0$ and the $\ha g^u_{t_0}$-image of the parts of $\eta_\sigma$ after $u_\sigma(t_0)$, $\sigma\in\{+,-\}$, both intersect the $\ha g^u_{t_0}$-image of  $\{|z|>L\}$.
By Lemma \ref{DMP-123} for $j=1$, conditionally on $\F_{\ulin u(t_0)}$ and the event $\{T^u>t_0\}$, the $\ha g^u_{t_0}$-image of the parts of $\eta_\sigma$ after $u_\sigma(t_0)$, $\sigma\in\{+,-\}$, after normalization, form a commuting pair of hSLE$_\kappa$ curves in the chordal coordinate started from $(R_+(t_0),-R_-(t_0);1,-1)$. The condition that $\eta_\sigma(u_\sigma(t_0))\not\in \eta_{-\sigma}[0,u_{-\sigma}(t_0)]$, $\sigma\in\{+,-\}$, is a.s.\ satisfied on $\{T^u>t_0\}$,which follows from Lemma \ref{W=V} and the fact that a.s.\ $R_\sigma(t_0)=(W^u_\sigma(t_0)-V_0^u(t_0))/(V_\sigma^u(t_0)-V_0^u(t_0))>0$, $\sigma\in\{+,-\}$, on $\{T^u>t_0\}$ (because of the transition density of $(\ulin R)$ vanishes outside $(0,1)^2$).
 From (\ref{R1R2}) we get
\BGE  \PP[\eta_\sigma\cap \{|z|>L\}\ne\emptyset, \sigma\in\{+,-\}|\lin\F_{\ulin u(t_0)},T^u>t_0]  \gtreqless p(\ulin R(t_0);   L/{e^{2t_0}}\pm 2)].\label{inclusion}\EDE
Here when we choose $+$ (resp.\ $-$) in $\pm$, the inequality holds with $\ge$ (resp.\ $\le$).

We use the approach of \cite{LR} to prove the convergence of $\lim_{L\to\infty} L^{\alpha_1} p(\ulin r,L)$.
{Note that the underlying probability measure for the $(\eta_1,\eta_2)$ here is the $\PP_1$ introduced in Section \ref{section-two-curve}.}
We first estimate $p(L):=\int_{(0,1)^2} p(\ulin r;L) \til p^1_\infty(\ulin r)d\ulin r$, where $\til p^1_\infty$ is the quasi-invariant density for the process $(\ulin R)$ under $\PP_1$ given in Lemma \ref{property-til-p}.
 This is the probability that the two curves in a $2$-SLE$_\kappa$ in $\HH$ with link pattern $(r_+\lr 1;-r_-\lr -1)$ both hit $\{|z|>L\}$, where $(r_+,r_-)$ is a random point in $(0,1)^2$ that follows the density $\til p^1_\infty$. From Lemma \ref{property-til-p} we know that, for {a} deterministic time ${t}$, $\PP[T^u>{t}]=e^{-\alpha_1 {t}}$, and the law of $(\ulin R({t}))$ conditionally on  $\{T^u>{t}\}$ still has density $\til p^1_\infty$.  Thus, the conditional joint law of the $\ha g^u_{{t}}$-images of the parts of $\ha\eta_\sigma$ after $\eta_\sigma(u_\sigma({t}))$, $\sigma\in\{+,-\}$ given $\F^u_{{t}}$ and $\{T^u>{t}\}$ agrees with that of $(\ha\eta_+,\ha\eta_-)$. From (\ref{inclusion})
we get $p(L)\gtreqless  e^{-2\alpha_1 {t}} p( L/e^{2{t}}\pm 2)$.
Let $q(L)=L^{\alpha_1} p(L)$. Then
\BGE q(L)\gtreqless (1\pm 2e^{2{t}}/L)^{-\alpha_1} q( L/e^{2{t}}\pm 2),\quad {\mbox{if }t\ge 3\mbox{ and }L>2e^{2t}}.\label{L-L}\EDE
Suppose $L_0> 4$ and $L\ge e^6(L_0+2)$.  Let $t_\pm=\log(L/(L_0\mp 2))/2$. Then $L/e^{2t_\pm}\pm 2=L_0$, $t_+> t_-\ge 3$ and $L=(L_0-2)e^{2t_+}>2e^{2t_+}> 2e^{2t_-}$. From (\ref{L-L}) (applied here with $t_\pm$   in place of ${t}$) we get
\BGE   q(L)\gtreqless (1\mp 2/L_0)^{\alpha_1} q(L_0), \quad \mbox{if }L\ge e^6(L_0+2)\mbox{ and }L_0>4.\label{q-lim}\EDE
From (\ref{V-V'}) we know that $T^u>{t}$ implies that both $\eta_+$ and $\eta_-$ intersect $\{|z|>e^{2{t}}/64\}$. Since $\PP[T^u>{t}]=e^{-2\alpha_1 {t}}>0$ for all ${t}\ge 0$, we see that $p$ is positive on $[0,\infty)$, and so is $q$. From (\ref{q-lim}) we see that $\lim_{L\to \infty} q(L)$ {exists and lies in} $(0,\infty)$. Denote it by $q(\infty)$. By fixing $L_0\ge 4$ and sending $L\to \infty$ in (\ref{q-lim}), we get
\BGE   p(L_0)\gtreqless q(\infty) L_0^{-\alpha_1} (1\mp 2/L_0)^{-\alpha_1},\quad \mbox{if }L_0\ge 4. \label{p-asymp}\EDE

Now we estimate $p(\ulin r;L)$ for a fixed deterministic $\ulin r\in [0,1)^2\sem \{(0,0)\}$.
The process $(\ulin R)$   starts from $\ulin r$ and has transition density $\til p^1_t$ given by Lemma \ref{transition-1}. Fix $L>2e^6$ and choose $t_0\in[3, \log(L/2)/2)$. {The event that} both $\eta_+$ and $\eta_-$ intersect $\{|z|>L\}$ {is then contained in the event} ${\{}T^u>t_0{\}}$. Let $\beta_1=10$. From Lemma \ref{property-til-p} we know that $\PP_{1}[T^u>t_0]={{\cal Z}_1} G_1^*(\ulin r) e^{-2\alpha_1 t_0}(1+O(e^{- \beta_1 t_0}))$ and the law of $\ulin R(t_0)$ conditionally on   $\{T^u>t_0\}$ has a density on $(0,1)^2$, which equals $\til p^1_\infty\cdot (1+O(e^{-\beta_1 t_0}))$, where $\beta_1=10$. Using Lemma \ref{DMP-123} and (\ref{inclusion},\ref{p-asymp}) we get
$$p(\ulin r;L)={\cal Z}_1 q(\infty) G_1^*(\ulin r) e^{-2\alpha_1 t_0}(L/e^{2t_0})^{-\alpha_1}(1+O(e^{-\beta_1 t_0}))(1+O(e^{2t_0}/L)).$$
For $L>e^{36}$, by choosing $t_0>3$ such that $e^{2t_0}=L^{2/(2+\beta_1)}$ and letting $C_0={\cal Z} q(\infty)$, we get $p(\ulin r;L)=C_0  G_1^*(\ulin r) L^{-\alpha_1} (1+O(L^{-\beta_1'}))$. Here we note that $\beta_1'=\beta_1/(\beta_1+2)$.

Since $G_1^*(r_+,r_-)=G_1(r_+,-r_-;1,-1)$, we proved (\ref{Thm1-est}) for $v_\pm =\pm 1$,  $w_+\in[0,1)$, and $w_-\in(-1,0]$.
Since $G_1(aw_++b,aw_-+b;av_++b ,av_-+b)=a^{-\alpha_1} G_1(w_+,w_-;v_+,v_-)$ for any $a>0$ and $b\in\R$, by a translation and a dilation, we get (\ref{Thm1-est}) in the case that $(v_++v_-)/2\in[w_-,w_+]$. Here we use the assumption that $0\in[v_-,v_+]$ to control the amount of translation.

Finally, we consider all other cases, i.e., $(v_++v_-)/2\not\in [w_-,w_+]$. By symmetry, we may assume that $(v_++v_-)/2<w_-$. Let $v_0=(w_++w_-)/2$. Then $v_+>w_+>v_0>w_->v_-$, but $v_+-v_0<v_0-v_-$. We still let $V_\nu$ be the force point functions started from $v_\nu$, $\nu\in\{0,+,-\}$. By (\ref{pa-X}), $V_\nu$ satisfies the PDE $\pa_+ V_\nu\aeq \frac{2W_{+,1}^2}{V_\nu -W_+ }$ on ${\cal D}_1^{\disj}$ as defined in Section \ref{section-deterministic-2}.  Thus, on ${\cal D}_1^{\disj}$, for any $\nu_1\ne \nu_2\in\{+,-,0\}$,
$\pa_+\log |V_{\nu_1} -V_{\nu_2} |\aeq \frac{-2W_{+,1}^2} {(V_{\nu_2} -W_+ )(V_{\nu_1} -W_+ )}$,
which implies that
\BGE \frac{\pa_+(\frac{V_+-V_0}{V_0-V_-})}{\pa_+\log(V_+-V_-)}
=\frac{V_+-V_0}{W_+-V_0}\cdot \frac{V_+-V_-}{V_0-V_-} >1.\label{>1}\EDE
The displayed formula means that $\frac{V_+ -V_0 }{V_0 -V_- }|^-_0$ is increasing  faster than $\log(V_+ -V_-)|^-_0$. From the assumption, $\frac{V_+(\ulin 0)-V_0(\ulin 0)}{V_0(\ulin 0)-V_-(\ulin 0)}=\frac{v_+-v_0}{v_0-v_-}\in(0,1)$. Let $\tau_+$ be the first $t$ such that $\frac{V_+(t,0)-V_0(t,0)}{V_0(t,0)-V_-(t,0)}=1$; if such time does not exist, then set $\tau_+=T_+$. Then $\tau_+$ is an $\F^+$-stopping time, and from (\ref{>1}) we know that, for any $0\le t<\tau_+$, $|V_+(t,0)-V_-(t,0)|<e|v_+-v_-|$, which implies by (\ref{V-V}) that $\diam([v_-,v_+]\cup\eta_+[0,t])< e|v_+-v_-|$. Let $L=e|v_+-v_-|$. From  $0\in [v_-,v_+]$ we get $\tau_+\le \tau^+_L$.

Here and below, we write $\ulin W$ and $\ulin V$ for $(W_+,W_-)$ and $(V_+,V_-)$, respectively. From Lemma  \ref{M-mart} we know that $M_1(\cdot\wedge \tau_L^+,0)$ is a martingale closed by $M_1(\tau_L^+,0)$. By Proposition \ref{OST} and the facts that $M_1=G_1(\ulin W;\ulin V)$ and $M_1(t,0)=0$ for $t\ge T_+$, we get
\BGE \EE[{\bf 1}_{\{\tau_+<T_+\}} G_1(\ulin W ;\ulin V )|_{(\tau_+,0)}]=\EE[M_1(\tau_+,0)]=M_1(0,0)=G_1(\ulin w;\ulin v).\label{EM1}\EDE
Using the same argument as in the proof of (\ref{inclusion}) with $(\tau_+,0)$ in place of $\ulin u(t_0)$ and $g_{K(\tau_+,0)}$ in place of $\ha g^u_{t_0}$, we get
\BGE  \PP[\eta_\sigma\cap\{|z|=L\}\ne\emptyset,\sigma\in\{+,-\}|\F^+_{\tau_+},\tau_+<T_+]\gtreqless p((\ulin W;\ulin V)|_{(\tau_+,0)} ;L\pm(V_+ -V_- )|_{(\tau_+,0)}).\label{p=exp}\EDE

Suppose $\tau_+<T_+$. Then the middle point of $[V_-(\tau_+,0),V_+(\tau_+,0)]$ is $V_0(\tau_+,0)$, which lies in $[W_-(\tau_+,0),W_+(\tau_+,0)]$. Also note that $0\in [V_-(\tau_+,0),V_+(\tau_+,0)]$ since $V_\pm (\tau_+,0)\gtreqless v_\pm \gtreqless 0$. Let $L_\pm=L\pm(V_+ -V_- )|_{(\tau_+,0)}$.
We may apply the result in the particular case to get
\begin{align}
p((\ulin W;\ulin V)|_{(\tau_+,0)} ;L_\pm) \nonumber  =& C_0 G_1(\ulin W;\ulin V)|_{(\tau_+,0)}\cdot L_\pm^{-\alpha_1} (1+O( {(V_+ -V_- )|_{(\tau_+,0)}}/{L_\pm} )^{\beta_1'}  ) \nonumber \\
  =& C_0 G_1(\ulin W;\ulin V)|_{(\tau,0)}\cdot L ^{-\alpha_1} (1+O({|v_+-v_-|}/{L})^{\beta_1'}). \label{p=est}
\end{align}
Here in the last step we used $(V_+ -V_- )|_{(\tau_+,0)}\le e|v_+-v_-|$ and $L_\pm /L=1+O(|v_+-v_-|/L)$. Plugging (\ref{p=est}) into (\ref{p=exp}), taking expectation on both sides of (\ref{p=exp}), and using the fact that $\{\eta_+\cap \{|z|=L\}\ne\emptyset\}\subset\{\tau_+<T_+\}$, we get
\begin{align*}
p(\ulin w;\ulin v;L)  =&C_0 \EE[{\bf 1}_{\{\tau<T_+\}} G_1(\ulin W;\ulin V)|_{(\tau,0)}]\cdot L^{-\alpha_1} (1+O(|v_+-v_-|/L)^{\beta_1'})\\
  =&C_0G_1(\ulin w;\ulin v) \cdot L^{-\alpha_1} (1+O(|v_+-v_-|/L)^{\beta_1'}),
\end{align*}
 where in the last step we used (\ref{EM1}). The proof is now complete.
\end{proof}

\begin{Theorem}
Let $\kappa\in(4,8)$. Then	Theorem \ref{Thm1} holds with  the same  $\alpha_1,\beta_1,G_1$ but a different positive constant $C$ under either of the following two modifications:
\begin{enumerate}
	\item [(i)]   the set $\{|z|>L\}$ is replaced by $(L,\infty)$, $(-\infty,-L)$, or $(L,\infty)\cup (-\infty,-L)$;
	\item [(ii)]   the event that $\eta_\sigma\cap \{|z|>L\}\ne \emptyset$, $\sigma\in \{+,-\}$, is replaced by $\eta_+\cap \eta_-\cap \{|z|>L\}\ne\emptyset$.
\end{enumerate}
\label{Thm1'}
\end{Theorem}
\begin{proof}
	The same argument in the proof of Theorem \ref{Thm1} works here, where the assumption that $\kappa\in(4,8)$ is used to guarantee that the probability of all event are positive for any $L>0$.
\end{proof}

\begin{Theorem}
   Let $v_-<w_-<w_+<v_+\in\R$ be such that $0\in [v_-,v_+]$. Let $\eta_w$ be an hSLE$_\kappa$ curve in $\HH$ connecting $w_+$ and $w_-$ with force points $v_+$ and $v_-$.
    Let $\alpha_2=\frac{2}\kappa(12-\kappa)$, $\beta_2'=\frac 56$, and $G_2$ be as in (\ref{G2(w,v)}.
  Then there is a constant $C>0$ depending only on $\kappa$ such that, as $L\to \infty$,
  $$  \PP[\ha\eta_w\cap\{|z|>L\}\ne \emptyset  ]= CL^{-\alpha_2} G_2(\ulin w;\ulin v) (1+O ( {|v_+-v_-|}/L )^{\beta_2'}),$$
  where the implicit constants in the $O(\cdot)$ symbol depend only on $\kappa$.
 \label{Thm2}
\end{Theorem}
\begin{proof}
  Let $(\eta_+,\eta_-;{\cal D}_2)$ be the  random commuting pair of chordal Loewner curves as defined in Section \ref{section-iSLE-1}. Then for $L>\max\{|v_+|,|v_-|\}$, $\ha \eta_w\cap\{|z|>L\}\ne \emptyset$ if and only if $\eta_{\sigma}\cap \{|z|>L\}\ne\emptyset$ for $\sigma\in \{+,-\}$.
   The rest of the proof follows  that  of Theorem \ref{Thm1} except that we now apply Lemmas \ref{property-til-p} and \ref{DMP-123} with $j=2$ and  use Lemma \ref{M-mart2}
    in place of Lemma  \ref{M-mart}.
\end{proof}

\begin{Theorem}
    Let $ w_-<w_+<v_+ \in\R$ be such that $0\in [w_-,v_+]$. Let $\eta_w$ be an hSLE$_\kappa$ in $\HH$ connecting $w_+$ and $w_-$ with force points $v_+$ and $\infty$. Let $\alpha_3=\frac {12}\kappa-1$, $\beta_3'=\frac 45$, and
   $G_3(\ulin w;v_+)$ be as in (\ref{G3(w,v)}).
  Then there is a constant $C>0$ depending only on $\kappa$ such that, as $L\to \infty$,
  $$  \PP[\ha\eta_w \cap\{|z|>L\}\ne \emptyset ]= CL^{-\alpha_3} G_3(\ulin w;v_+) (1+O ( {|w_+-v_-|}/L )^{ {\beta_3'} } ),$$
  where the implicit constants in the $O(\cdot)$ symbol depend only on $\kappa$.
  \label{Thm3}
\end{Theorem}
\begin{proof}
  The proof follows  those of Theorems \ref{Thm2} and \ref{Thm1} except that we now introduce $v_0:=(w_++w_-)/2$ and $v_-:=2v_0-v_+$ as in Section \ref{section-iSLE-2}. Then we can define the time curve $\ulin u$ as in Section \ref{time curve} and apply Lemmas \ref{property-til-p} and \ref{DMP-123} with $j=3$. 
\end{proof}

\begin{proof}[Proof of Theorem \ref{main-Thm1}]
By conformal invariance of $2$-SLE$_\kappa$, we may assume that $D=\HH$ and $z_0=\infty$. Case (A1) follows immediately from Theorem \ref{Thm1}. Cases (A2) and (B) respectively follow from Theorems \ref{Thm2} and \ref{Thm3} since we only need to consider the Green's function for the curve connecting $w_+$ and $w_-$, which is an hSLE$_\kappa$ curve.
\end{proof}

\begin{Remark}
 The  hSLE$_\kappa$ curve  is a special case of the intermediate SLE$_\kappa(\rho)$ (iSLE$_\kappa(\rho)$ for short) curves in \cite{kappa-rho} with $\rho=2$. An iSLE$_\kappa(\rho)$  curve is defined using Definition \ref{Def-hSLE} with $F:= \,_2F_1(1-\frac{4}\kappa,\frac{2\rho}\kappa; \frac{2\rho+4}\kappa;\cdot )$ and $\til G:=\kappa \frac{F'}F+\rho$. The curve is well defined for $\kappa\in(0,8)$ and $\rho>\min\{-2,\frac \kappa 2-4\}$, and satisfies reversibility when $\kappa\in(0,4]$ and $\rho>-2$ or $\kappa\in (4,8)$ and $\rho\ge \frac \kappa 2-2$ (cf.\ \cite{multipl-kappa-rho}). When an iSLE$_\kappa(\rho)$ satisfies reversibility, we can obtain a commuting pair of iSLE$_\kappa(\rho)$ curves in the chordal coordinate started from $(w_+\lr w_-;v_+,v_-)$ or $(w_+\lr w_-;v_+)$ for given points $v_-<w_-<w_+<v_+$, which satisfy two-curve DMP. Following similar arguments, we find that Theorems \ref{Thm2} and \ref{Thm3} respectively hold for iSLE$_\kappa(\rho)$ curves with $\alpha_2=\frac{\rho+2}\kappa(\rho+4-\frac \kappa 2)$, $\alpha_3=\frac 2\kappa (\rho+4-\frac \kappa 2)$, $\beta_2'=\frac{2\rho+6}{2\rho+8}$, $\beta_3'=\frac{\rho+6}{\rho+8}$, and (with $F= \,_2F_1(1-\frac{4}\kappa,\frac{2\rho}\kappa; \frac{2\rho+4}\kappa;\cdot )$)
$$G_2(\ulin w;\ulin v)=|w_+-w_-|^{\frac 8\kappa -1}|v_+-v_-|^{\frac{\rho(2\rho+4-\kappa)}{2\kappa}}\prod_{\sigma\in\{+,-\}} |w_\sigma-v_{-\sigma}|^{\frac{2\rho}\kappa}  F\Big(\frac{(v_+-w_+)(w_--v_-)}{(w_+-v_-)(v_+-w_-)}\Big)^{-1},$$
$$G_3(\ulin w;v_+)=|w_+-w_-|^{\frac 8\kappa -1}|v_+-w_-|^{\frac{2\rho}\kappa}  F\Big(\frac{v_+-w_+ }{ w_+-w_-}\Big)^{-1}.$$
 The proofs use the estimate on the transition density of $\ulin R$ under $\PP_{\ulin w;\ulin v}^{(\rho,\rho)}$ and $\PP_{\ulin w;\ulin v}^{(\rho)}$ (Corollary \ref{transition-R-infty}) and revisions of Lemmas \ref{RN-Thm2-inv} and  \ref{RN-Thm3-inv})  with $\PP_2^0$ and $\PP_3^0$ now respectively representing  $\PP_{\ulin w;\ulin v}^{(\rho,\rho)}$ and $\PP_{\ulin w;\ulin v}^{(\rho)}$, $\PP_2$ and $\PP_3$ now respectively representing the joint law of the driving functions for a commuting pair of iSLE$_\kappa(\rho)$ curves in the chordal coordinate started from $(w_+\lr w_-;v_+,v_-)$ and from $(w_+\lr w_-;v_+)$, and $M_2$ and $M_3$ replaced by $G_2(W_+,W_-;V_+,V_-)$ and $G_3(W_+,W_-;V_+)$ for the current $G_2$ and $G_3$.

The revision of Theorem \ref{Thm2} (resp.\ \ref{Thm3}) also holds in the degenerate case: $v_+=w_+^+$, in which  the $\eta_w$ oriented from $w_-$ to $w_+$ is a chordal SLE$_\kappa(\rho)$ curve in $\HH$ from $w_-$ to $w_+$ with the force point at $v_-$ (resp.\ $\infty$). After a conformal map, we then obtain the boundary Green's function for a chordal SLE$_\kappa(\rho)$ curve in $\HH$ from $0$ to $\infty$ with the force point $v>0$ at a point $z_0\in (v,\infty)$ or at $z_0=v$. Such Green's functions may also be obtained from the traditional one-curve approach in \cite{Mink-real}.
The exponents $\alpha_2$ and $\alpha_3$ have appeared in \cite[Theorem 3.1]{MW} with a rougher estimate on the intersection probability.
\end{Remark}

\end{document}